[1994/12/01]
\documentclass{ijmart}

\usepackage{amsmath,amssymb,amsfonts,amssymb,amsthm}
\usepackage{rotating, verbatim,hyperref,url,tikz,tikz-qtree,ifthen,cancel,array,graphicx,adjustbox,appendix,pifont,enumitem,vruler}
\usetikzlibrary{positioning}
\usepackage{mathrsfs}

\chardef\bslash=`\\ 

\newtheorem{thm}{Theorem}[section]

\newtheorem{lem}[thm]{Lemma}

\newtheorem*{thm.indivisibility}{Theorem \ref{thm.indivisibility}}

\newtheorem{fact}[thm]{Fact}

\newtheorem*{CharBRDSDAP}{Simple Characterization of  big Ramsey degrees}

\theoremstyle{remark}
\newtheorem{rem}[thm]{Remark}

\theoremstyle{definition}
\newtheorem{defn}[thm]{Definition}
\newtheorem{convention}[thm]{Convention}

\newtheorem{notation}[thm]{Notation}

\newtheorem{example}[thm]{Example}

\theoremstyle{remark}

\newcommand{\al}{\alpha}
\newcommand{\om}{\omega}

\newcommand{\sse}{\subseteq}
\newcommand{\contains}{\supseteq}
\newcommand{\forces}{\Vdash}

\newcommand{\bJ}{\mathbf{J}}
\newcommand{\bK}{\mathbf{K}}

\newcommand{\bM}{\mathbf{M}}

\newcommand{\bP}{\mathbb{P}}
\newcommand{\bQ}{\mathbb{Q}}

\newcommand{\bT}{\mathbb{T}}
\newcommand{\bS}{\mathbb{S}}

\newcommand{\bU}{\mathbb{U}}

\newcommand{\K}{\mathrm{K}}
\newcommand{\A}{\mathrm{A}}
\newcommand{\B}{\mathrm{B}}

\newcommand{\M}{\mathrm{M}}
\newcommand{\N}{\mathrm{N}}
\newcommand{\re}{\!\restriction\!}
\newcommand{\bA}{\mathbf{A}}
\newcommand{\bB}{\mathbf{B}}
\newcommand{\bF}{\mathbf{F}}
\newcommand{\bG}{\mathbf{G}}

\newcommand{\bfA}{\mathbf{A}}
\newcommand{\bfB}{\mathbf{B}}
\newcommand{\bfC}{\mathbf{C}}
\newcommand{\bfD}{\mathbf{D}}
\newcommand{\bfE}{\mathbf{E}}
\newcommand{\bfM}{\mathbf{M}}
\newcommand{\bfN}{\mathbf{N}}
\newcommand{\bfO}{\mathbf{O}}

\newcommand{\wsim}{\stackrel{w}{\sim}}

\newcommand{\plussim}{\stackrel{+}{\sim}}
\newcommand{\Lplussim}{\stackrel{\mathrm{L}+}{\sim}}
\newcommand{\Lsim}{\stackrel{\mathrm{L}}{\sim}}
\newcommand{\Lwsim}{\stackrel{\mathrm{Lw}}{\sim}}

\newcommand{\ra}{\rightarrow}

\newcommand{\lgl}{\langle}
\newcommand{\rgl}{\rangle}

\newcommand{\rl}{\!\downarrow\!}

\newcommand{\Erdos}{Erd{\H{o}}s}
\newcommand{\Fraisse}{Fra{\"{i}}ss{\'{e}}}
\newcommand{\Hubicka}{Hubi{\v{c}}ka}

\newcommand{\EEAP}{SDAP}
\newcommand{\EEAPnonacronym}{Substructure  Disjoint Amalgamation Property}
\newcommand{\SFAP}{SFAP}
\newcommand{\SFAPnonacronym}{Substructure  Free  Amalgamation Property}

\newcommand{\noprint}[1]{\relax}

\DeclareMathOperator{\ran}{ran}

\DeclareMathOperator{\Sim}{Sim}
\DeclareMathOperator{\Ext}{Ext}

\DeclareMathOperator{\type}{tp}

\newcommand{\eval}[2][\right]{\relax%
  \ifx#1\right\relax \left.\fi#2#1\rvert}

\begin{document}
\title[Indivisible Fra{\"{i}}ss{\'{e}} structures]{Fra{\"{i}}ss{\'{e}}  structures with SDAP$^+$, Part I:  Indivisibility}

\author{R. Coulson}
\address{United State Military Academy, West Point\\
Department of Mathematical Sciences, Thayer Hall 233, West Point, USA}
\email{rebecca.coulson@westpoint.edu}
\urladdr{\url{https://sites.google.com/view/rebeccacoulson/home}}

\author{N. Dobrinen}
\address{University of Denver\\%
Department of Mathematics, 2390 S. York St., Denver, CO USA}
\email{natasha.dobrinen@du.edu}
  \urladdr{\url{http://math.du.edu/~ndobrine}}

  \author{R. Patel}
  \address{%
 African Institute for Mathematical Sciences\\%
    M'bour-Thi\`{e}s, Senegal}
\email{rpatel@aims-senegal.org}

\thanks{The second author is grateful for support from   National Science Foundation Grant DMS-1901753, which also supported research visits to the University of Denver by the first and third authors.
She also is grateful for support
from Menachem Magidor for hosting her visit to
 The Hebrew University of Jerusalem  in December 2019, during which some of the ideas in this paper were formed.%
The third author's work on this paper was supported by the National Science Foundation under Grant No. DMS-1928930 while she was in residence at the Mathematical Sciences Research Institute in Berkeley, California, during the Fall 2020 semester.
}




\begin{abstract}
This is Part I of a two-part series regarding Ramsey properties of \Fraisse\ structures satisfying 
a property called SDAP$^+$, which strengthens the Disjoint  Amalgamation Property.
We prove that every \Fraisse\ structure
in a finite relational language with relation symbols of any finite arity satisfying this property is indivisible.
Novelties    include
a new  formulation of  coding trees in terms of 1-types over initial segments of the \Fraisse\ structure, 
and a direct  proof of indivisibility  which uses the method of forcing to conduct unbounded searches for finite sets. 
In Part II, we prove 
 that
 every \Fraisse\  structure
 in a finite relational language with relation symbols of arity at most two having
this  property
has finite big Ramsey degrees which have a  simple characterization.
It follows that
any such \Fraisse\ structure
admits
 a big  Ramsey structure.
 Part II utilizes a theorem from Part I  as a pigeonhole principle for induction arguments. 
 This work offers a
 streamlined  and unifying approach to  Ramsey theory on some seemingly disparate classes of \Fraisse\ structures.
\end{abstract}
\maketitle
\tableofcontents

\section{Introduction}

\label{sec.intro}

In recent years,
the
Ramsey theory
of
infinite structures has seen quite an expansion.
This area seeks to understand which infinite structures satisfy some analogue of the infinite
Ramsey theorem for the natural numbers.

\begin{thm}[Ramsey, \cite{Ramsey30}]\label{thm.RamseyInfinite}
Given integers $k,r\ge 1$ and a coloring
of the $k$-element subsets of the natural numbers into $r$ colors,
there is an infinite set of natural numbers, $N$, such that all $k$-element subsets of $N$ have the same color.
\end{thm}

For infinite structures, 
exact analogues of  Ramsey's theorem usually fail, even when the class of finite substructures has the Ramsey property.
This is due to some unseen structure  which persists in  every
infinite substructure isomorphic to the original, but  which dissolves
when considering Ramsey properties of
classes of finite substructures.
This was first seen in
Sierpi\'{n}ski's use of a well-ordering on the rationals to construct a coloring of  unordered pairs of rationals
with two colors such that both colors persist in any subcopy of the rationals (see \cite{Sierpinski}).
The interplay between the well-ordering and the rational order forms
 additional structure   which is in some sense essential, as it  persists upon taking any subset forming another dense linear order
 without endpoints.
The quest to characterize and quantify  the  often hidden but essential  structure for infinite structures, more generally, is the area of
 {\em big Ramsey degrees}.

Given an infinite structure
$\bfM$, we say that
 $\bfM$ has {\em finite big Ramsey degrees} if for each finite substructure
$\bfA$ of $\bfM$, there is an integer $T$
such that the following  holds:
For any coloring of
the copies of $\bfA$ in $\bfM$
into finitely many colors, there is a
 substructure $\bfM'$ of $\bfM$
such that $\bfM'$ is isomorphic to $\bfM$, and
 the copies of $\bfA$ in $\bfM'$ take no more than $T$ colors.
When a $T$ having this property exists, the least  such value
is called the {\em big Ramsey degree} of $\bfA$ in
$\bfM$,
denoted
$T(\bfA, \bM)$.
In particular, if
the big Ramsey degree of $\bfA$ in $\bfM$ is one,
then any finite coloring of the copies of
$\bfA$ in $\bfM$ is
constant
on some subcopy of $\bfM$.

While the area of big Ramsey degrees on infinite structures traces back to Sierpi\'{n}ski's result  that
the big Ramsey degree for unordered pairs of rationals is at least two, and progress on the rationals and other binary relational structures was made in  the decades since,
the question of which infinite structures have finite big Ramsey degrees
attracted
extended interest due to
the flurry of
results in
 \cite{Laflamme/NVT/Sauer10},
\cite{Laflamme/Sauer/Vuksanovic06}, \cite{NVT08},
and \cite{Sauer06}
in tandem with
the publication
of
 \cite{Kechris/Pestov/Todorcevic05}, in which
 Kechris, Pestov, and Todorcevic
asked for an analogue of their
correspondence between
 the Ramsey property  of \Fraisse\ classes and  extreme amenability to the setting of
 big Ramsey degrees for \Fraisse\ limits.
 This was
 addressed
 by Zucker in \cite{Zucker19},
 where he proved a
 connection
between
 \Fraisse\
 limits
 with finite big Ramsey degrees and completion flows in topological dynamics.
Zucker's results apply to {\em big Ramsey structures},
expansions of  \Fraisse\
limits
in which the big Ramsey degrees of the \Fraisse\
limits
can be exactly characterized using the additional structure induced by the expanded language.
This additional structure
involves a well-ordering, and
characterizes
the  essential structure which persists  in every infinite subcopy of the \Fraisse\
limit.
It is this essential structure  we seek to understand
 in the study of big Ramsey degrees.

We  describe
an amalgamation property,
called the \EEAPnonacronym\ (\EEAP),
forming a strengthened version of  disjoint  amalgamation.
The \Fraisse\ limit of a \Fraisse\ class satisfying \EEAP\ is said to satisfy \EEAP$^+$ if it satisfies two additional properties,
 which  we call
the Diagonal Coding Tree Property and the Extension Property.
The motivation behind \EEAP$^+$ was to distill properties  inherent in proofs of big Ramsey degrees which have a simple characterization, and it  has led to 
Theorems \ref{thm.indivisibility} and \ref{thm.main} below.

A particular case of Ramsey theory on infinite structures is when one colors copies of a given substructure  with universe of size one. 
A \Fraisse\ limit  $\bK$ is  called
{\em indivisible} if  every one-element substructure of $\bK$ has big Ramsey degree equal to one.
In Part I, we prove indivisibility for \Fraisse\ limits in finite relational languages with relation symbols of any finite arity satisfying \EEAP$^+$.
In the case when $\bK$ has exactly one substructure of size one
(up to isomorphism),
as happens for instance when the language of $\bK$ has no unary relation symbols and there are no ``loops''
in $\bK$, this definition reduces to the usual one for indivisibility of structures like the Rado graph and the Henson graphs
(see \cite{Sauer06}, \cite{Komjath/Rodl86}, and  \cite{El-Zahar/Sauer89}).

\begin{thm}\label{thm.indivisibility}
Suppose  $\mathcal{K}$ is a  \Fraisse\ class
in a finite relational language
with relation symbols  in any arity such that
its  \Fraisse\ limit  $\bK$ satisfies
\EEAP$^+$.
Then $\bK$ is indivisible.
\end{thm}

Theorem \ref{thm.indivisibility} provides new  classes of examples of indivisible \Fraisse\ structures, in particular for ordered structures,
while recovering results in
\cite{El-Zahar/Sauer89},
\cite{Komjath/Rodl86},  and
\cite{El-Zahar/Sauer94}.

In Part II,
we characterize the exact big Ramsey degrees for all \Fraisse\  limits in
finite relational languages
with relation symbols of arity at most two
satisfying \EEAP$^+$, or a related property called LSDAP$^+$.
 Our characterization, together with results of
Zucker in \cite{Zucker19}, imply that   such \Fraisse\ limits 
admit big Ramsey structures, and their automorphism groups have metrizable universal completion flows.

\begin{thm}\label{thm.main}
Let $\mathcal{K}$ be a  \Fraisse\  class
in a finite relational language
with relation symbols of arity at most two
such that the \Fraisse\ limit  
$\bK$
of $\mathcal{K}$ has
\EEAP$^+$ or LSDAP$^+$. 
Then
$\bK$
has finite big Ramsey degrees which have a simple characterization and, moreover,
admits  a big Ramsey structure.
Hence, the topological group
{\rm{Aut(Flim}}$(\bK))$
 has a metrizable universal completion flow, which is unique up to isomorphism.
\end{thm}

Theorem \ref{thm.main} provides new classes of examples
 of big Ramsey structures while
 recovering results in
 \cite{DevlinThesis},
\cite{HoweThesis},
 \cite{Laflamme/NVT/Sauer10}, and
 \cite{Laflamme/Sauer/Vuksanovic06}
and extending special cases of the results in \cite{Zucker20} to obtain exact big Ramsey degrees.
Theorem \ref{thm.matrixHL}  in this paper will serve as the starting point for proving 
Theorem \ref{thm.main} in Part II.

We now discuss
 several theorems which  follow from  Theorem  \ref{thm.indivisibility} or Theorem \ref{thm.main},
 as well as  new examples
 obtained  from our results.
 A fuller description is provided in
Section 5 of  Part II  (\cite{CDPII}).

We show  in Part II that \EEAP$^+$ holds for  disjoint
 amalgamation classes which are ``unrestricted.''
 Particular instances 
 of unrestricted classes 
include 
 classes of
  structures with finitely many unary and binary relations such as
graphs, directed graphs,
tournaments,  graphs with several edge relations, etc., as well as
their ordered versions.
Our examples encompass those
unconstrained
binary relational structures considered in \cite{Laflamme/Sauer/Vuksanovic06}
as well as their ordered expansions.
We also  show in Part II that \EEAP$^+$ holds for \Fraisse\
limits
of  free amalgamation classes which  forbid $3$-irreducible substructures, namely, substructures in which any three distinct elements appear  in a tuple for which some relation holds, as well as for their ordered versions.
Hence, Theorem \ref{thm.indivisibility} implies that all unrestricted classes,
all free amalgamation classes which  forbid $3$-irreducible substructures, and their ordered expansions have \Fraisse\ limits which are indivisible.
Theorem  \ref{thm.main} implies that such classes with relation symbols of arity at most two have \Fraisse\ limits with  big Ramsey degrees which have a simple characterization.
See Propositions 5.2 and 5.4 and Theorem 5.5 of \cite{CDPII}
 for more details.

Our methods also apply to 
certain  \Fraisse\ structures
derived from the rational linear order.
In Part II, we will show that 
the structure
$\bQ_{\bQ}$, 
the dense linear order without endpoints with an equivalence relation such that  all equivalence classes are convex copies of the rationals,
satisfies a related property called LSDAP$^+$.
Theorem \ref{thm.main} (and hence also 
 the conclusion of 
Theorem \ref{thm.indivisibility})  holds for  $\bQ_{\bQ}$,
answering  a question raised by Zucker at the 2018 Banff Workshop on {\em Unifying Themes in Ramsey Theory}.
More generally, we show that 
members of
a natural hierarchy of finitely many convexly ordered equivalence relations, where each  successive  equivalence relation coarsens the previous one, also admit big Ramsey structures with a simple characterization.
  Theorem \ref{thm.main} 
  recovers known results 
including
Devlin's  characterization of  the  big Ramsey degrees of the rationals \cite{DevlinThesis}
 as well as
 results
 of Laflamme, Nguyen Van Th\'{e}, and Sauer  in
 \cite{Laflamme/NVT/Sauer10}
 characterizing the big Ramsey degrees of 
 $\bQ_n$,  the rational
linear order with a partition into $n$ dense
pieces, as these structures satisfy SDAP$^+$.
See Theorem 5.12 of  Part II for more details.

While many of the known big Ramsey degree results use sophisticated versions of Milliken's Ramsey theorem for trees \cite{Milliken79},
and while proofs using the method of forcing to produce  ZFC results have appeared in \cite{DobrinenRado19}, \cite{DobrinenH_k19}, \cite{DobrinenJML20}, and  \cite{Zucker20},
there are  three  novelties to our approach  which produce a clarity about indivisibility and more generally, about  big Ramsey degrees.
Given a \Fraisse\ class $\mathcal{K}$,
we fix an  enumerated \Fraisse\ limit of $\mathcal{K}$, which we denote by $\bK$.
By {\em enumerated \Fraisse\ limit}, we mean that the universe of $\bK$ is ordered via the natural numbers.
The first novelty is that  we
work with trees of   quantifier-free $1$-types
(see Definition
\ref{defn.treecodeK})
and develop  forcing arguments
directly on them to prove the Level Set Ramsey Theorem
(Theorem \ref{thm.matrixHL}).
It was suggested to the second author  by Sauer during the 2018 BIRS Workshop,
{\em Unifying Themes in Ramsey Theory}, to try moving the forcing methods from \cite{DobrinenH_k19}  and \cite{DobrinenJML20} to forcing directly on the structures.
Using trees of quantifier-free $1$-types seems to come as close as possible to fulfilling this request, as the $1$-types  allow one to see the essential hidden structure (the interplay of a well-ordering of the universe with first instances where $1$-types disagree), whereas working only on   the \Fraisse\ structures, with no reference to $1$-types, obscures this central feature of big Ramsey degrees from view.
We will be calling such trees {\em coding trees}, as there will be special nodes, called {\em coding nodes},
 representing the vertices of
 $\bK$:
 The $n$-th coding node will be the quantifier-free $1$-type of the  $n$-th vertex of $\bK$ over the
 substructure of $\bK$
 induced on the  first $n-1$ vertices of $\bK$.
 (The $0$-th coding node is the quantifier-free $1$-type of the $0$-th vertex over the empty
 set.)

 The second  novelty of our approach  is that  our Level Set Ramsey Theorem
 (Theorem \ref{thm.matrixHL}) immediately yields indivisibility for all \Fraisse\ structures satisfying SDAP$^+$ with finitely many relations of any finite arity.
This is a consequence of forcing on diagonal coding trees,  developed in the second author's work for big Ramsey degrees of Henson graphs 
in \cite{DobrinenJML20} and \cite{DobrinenH_k19}.
It is interesting to note that in general, indivisibility does not follow from forcing on widely branching coding trees; diagonal coding trees are necessary to obtain indivisibility directly.
The third novelty is that we find the exact big Ramsey degrees directly from the diagonal coding trees of $1$-types, without appeal to the standard method of ``envelopes''.
This means that
the upper bounds which we  find via forcing arguments
 are shown to be exact.

Using trees
of
quantifier-free
$1$-types
(partially ordered by inclusion)
allows us to prove   a  characterization of  big Ramsey degrees for \Fraisse\ limits with \EEAP$^+$ which is a simple extension of the so-called ``Devlin types'' for the rationals in \cite{DevlinThesis},
and of the  characterization of the big Ramsey degrees of the  Rado graph achieved by Laflamme, Sauer, and Vuksanovic in \cite{Laflamme/Sauer/Vuksanovic06}.
Here, we present  the characterization for  structures  without unary relations.
The full  characterization is given in  Part II.

\begin{CharBRDSDAP}
Let $\mathcal{L}$   be a  language
consisting of
finitely many  relation symbols, each of arity
two.
Suppose $\mathcal{K}$ is a \Fraisse\ class in $\mathcal{L}$
such that the \Fraisse\ limit
$\bK$
of $\mathcal{K}$ satisfies
\EEAP$^+$ or LSDAP$^+$.
Fix a structure  $\bfA\in\mathcal{K}$.
 Let
 $(\bfA,<)$ denote
  $\bfA$ together with a fixed enumeration
  $\lgl\mathrm{a}_i:i<n\rgl$
 of the universe
 of $\bfA$.
We say that a tree $T$ is a
{\em diagonal tree coding $(\bfA,<)$}
if  the following hold:
\begin{enumerate}
\item
$T$ is a finite tree with $n$ terminal nodes and
 branching degree two.
\item
$T$ has  at most one branching node  in
any given
level,
and
no two distinct nodes from among the branching nodes and terminal nodes have the same length.
Hence, $T$ has $2n-1$ many levels.
\item
Let $\lgl \mathrm{d}_i:i<n\rgl$ enumerate the terminal nodes in $T$ in order of increasing length.
Let $\bfD$ be the $\mathcal{L}$-structure induced on the set $\{\mathrm{d}_i:i<n\}$ by the increasing
bijection from   $\lgl\mathrm{a}_i:i<n\rgl$ to $\lgl \mathrm{d}_i:i<n\rgl$, so that $\bfD \cong \bfA$.
Let $\tau_i$ denote the quantifier-free
$1$-type of $\mathrm{d}_i$ over
$\bfD_i$,
the substructure of
$\bfD$ on vertices $\{\mathrm{d}_m:m<i\}$.
Given  $i<j<k<n$,
if $\mathrm{d}_j$ and $\mathrm{d}_k$
both
extend
some node
in $T$
that is at the same level as $\mathrm{d}_i$,
then $\mathrm{d}_j$ and $\mathrm{d}_k$
 have the same quantifier-free $1$-types over
$\mathbf{D}_i$.
That is,
$\tau_j\re \mathbf{D}_i
=\tau_k\re \mathbf{D}_i$.
\end{enumerate}
Let $\mathcal{D}(\bfA,<)$ denote the number of distinct diagonal trees coding $(\bfA,<)$;
let $\mathcal{OA}$ denote a set
consisting of
one representative from each isomorphism class of ordered copies of $\bfA$.
Then
 $$
 T(\bfA,\bK)
=
\sum_{(\bfA,<)\in\mathcal{OA}}\mathcal{D}(\bfA,<)
$$
\end{CharBRDSDAP}

If $\mathcal{L}$ also has unary relation symbols,
in the case that
$\mathcal{K}$ is a free amalgamation class, the simple characterization above holds when modified to  diagonal coding trees with the same number of roots as unary relations.
In the case that $\mathcal{K}$ contains a transitive relation,
then the  above characterization still holds.
The full characterization will be given in Theorem 4.8 in Part II.

 \vskip.1in

\it Acknowledgements. \rm
The second author thanks
Norbert Sauer for discussions at the  2018 Banff Workshop on
{\it Unifying Themes in Ramsey Theory},
where he suggested trying
  to move  the forcing  directly on the structures.
She also thanks
 Menachem Magidor
 for hosting her
  at the Hebrew University of Jerusalem in December 2019, and  for fruitful discussions on big Ramsey degrees during that time.
She thanks Itay Kaplan for  discussions on big Ramsey degrees and  higher arity relational structures during that visit, and Jan \Hubicka\ for helpful conversations.  The third author thanks Nathanael Ackerman, Cameron Freer and Lynn Scow for extensive and clarifying discussions.
All three authors thank Jan \Hubicka\ and Mat\v{e}j Kone\v{c}n\'{y} for pointing out  a mistake in an earlier version.


\section{Amalgamation Properties}\label{sec.Structures}


The inspiration for the amalgamation property \EEAP\ defined in this section comes from a
strengthening of the free amalgamation property,
which we  call the \SFAPnonacronym\ (\SFAP).
We originally found that any binary relational  \Fraisse\
structure  with an age satisfying  \SFAP\
has finite  big Ramsey degrees that are characterized in a manner  similar to  the characterizations,
  in \cite{Laflamme/Sauer/Vuksanovic06},
 of big Ramsey degrees
  for the
  Rado graph and other
 unconstrained
  binary relational structures
  with
  disjoint
  amalgamation.
  \SFAP\
  is satisfied by
  the ages of
  all
  unconstrained
  relational
structures having free amalgamation,
  as well as
by \Fraisse\ classes
  with forbidden
  irreducible and
  $3$-irreducible substructures.
 The \EEAPnonacronym\ (\EEAP)
  is a natural
 extension
 of \SFAP\
  to a broader collection of \Fraisse\ classes with
 disjoint amalgamation.
 

In Subsection \ref{subsec.Fcrs}
we review  the basics of \Fraisse\ theory,
the Ramsey property,
and indivisibility.
 Big Ramsey degrees and big Ramsey structures will be discussed in Section 2 of Part II.
More general background  on \Fraisse\ theory can be found in \Fraisse's original paper  \cite{Fraisse54},   as well as
\cite{HodgesBK97}.
  The
 properties
 \SFAP\ and
 \EEAP\
 are presented in
 Subsection \ref{subsec.EEAP}.
The presentation of   \EEAP$^+$  will  be given in Definition \ref{defn_EEAP_newplus},
 after   coding trees of $1$-types and related notions are defined in Section \ref{sec.sct}.


\subsection{
\Fraisse\ theory and indivisibility}\label{subsec.Fcrs}

All relations in this paper will be finitary,
and all languages will consist of finitely many
relation symbols (and no constant or function symbols).
We use the set-theoretic notation $\om$ to denote the set of natural numbers, $\{0,1,2,\dots\}$
and treat $n \in \om$ as the set $\{i\in\om:i<n\}$.

Let  $\mathcal{L}=\{R_i:i< I\}$ be a finite
language
 where each $R_i$ is a relation symbol with associated
arity $n_i\in \om$.
An {\em $\mathcal{L}$-structure} is an object
\begin{equation}
\bfM=\lgl \M, R_0^{\bfM},\dots, R_{I-1}^{\bfM}\rgl
\end{equation}
where
$\M$ is a nonempty set, called the \emph{universe} of $\bfM$,
and each $R^{\bfM}_i\sse \M^{n_i}$.
Finite structures will typically be denoted by $\bfA,\bfB$, etc.,
and  their universes  by $\A,\B$, etc.  Infinite structures will typically be denoted by $\bJ, \bK$ and their universes
by $\mathrm{J}, \K$. We will call the elements of the universe of a structure {\em vertices}.

An {\em embedding} between
$\mathcal{L}$-structures
$\bfM$ and $\bfN$
is an injection $\iota:\M\ra \N$ such that for each $i<I$
and for all $a_0, \ldots , \, a_{n_i - 1} \in \M$,
\begin{equation}
R_i^{\bfM}(a_0,\dots,a_{n_i - 1})\Longleftrightarrow
R_i^{\bfN}(\iota(a_0),\dots,\iota(a_{n_i - 1})).
\end{equation}
A surjective embedding is an \emph{isomorphism}, and an isomorphism from
$\bfM$ to
$\bfM$
is an {\em automorphism}.
The set of embeddings of $\bfM$ into $\bfN$ is denoted $\text{Emb}(\bfM, \bfN)$, and
the set of automorphisms
of $\bfM$ is denoted $\text{Aut}(\bfM)$.
When
$\M \subseteq \N$ and the inclusion map is an embedding, we say
$\bfM$ is a {\em substructure} of
$\bfN$.
When there exists an embedding $\iota$ from $\bfM$ to $\bfN$, the substructure of $\bfN$ having universe $\iota[\M]$ is
called a {\em copy} of $\bfM$ in $\bfN$, and it is a {\em subcopy} of $\bfN$ if $\bfM$ is isomorphic to $\bfN$. The {\em age} of $\bfM$,
written Age($\bfM$),
is the class of all finite $\mathcal{L}$-structures that embed into $\bfM$.
We write $\bfM \le \bfN$
when there is an embedding of
$\bfM$ into $\bfN$, and $\bfM\cong\bfN$ when there is an isomorphism from
$\bfM$ to $\bfN$.

A class $\mathcal{K}$ of finite structures
in
a
finite
relational language
 is called a {\em \Fraisse\ class}  if it is
nonempty,  closed under isomorphisms,
 hereditary, and satisfies the joint embedding and amalgamation properties.
The class $\mathcal{K}$ is  {\em hereditary} if whenever $\bfB\in\mathcal{K}$ and  $\bfA\le\bfB$, then also $\bfA\in\mathcal{K}$.
The class $\mathcal{K}$ satisfies the {\em joint embedding property} if for any $\bfA,\bfB\in\mathcal{K}$,
there is a $\bfC\in\mathcal{K}$ such that $\bfA\le\bfC$ and $\bfB\le\bfC$. The class
 $\mathcal{K}$ satisfies the {\em amalgamation property} if for any embeddings
$f:\bfA\ra\bfB$ and $g:\bfA\ra\bfC$, with $\bfA,\bfB,\bfC\in\mathcal{K}$,
there is a $\bfD\in\mathcal{K}$ and  there are embeddings $r:\bB\ra\bfD$ and $s:\bfC\ra\bfD$ such that
$r\circ f = s\circ g$.
 Note that in a finite relational language, there are only countably many finite structures up to isomorphism.

An $\mathcal{L}$-structure $\bK$ is called {\em ultrahomogeneous} if
every isomorphism between finite substructures of $\bK$  can be extended to an
automorphism of $\bK$.
We call a
countably infinite, ultrahomogeneous structure
a {\em \Fraisse\ structure}.
\Fraisse\ showed \cite{Fraisse54} that the age of a \Fraisse\ structure is a \Fraisse\ class, and that conversely,
given
a \Fraisse\ class $\mathcal{K}$, there is, up to isomorphism, a unique
\Fraisse\ structure
whose age is $\mathcal{K}$.
Such a
\Fraisse\
structure is called the {\em \Fraisse\ limit} of $\mathcal{K}$
or the {\em generic} structure for $\mathcal{K}$.

Throughout this paper, $\bK$ will denote the \Fraisse\ limit of a \Fraisse\ class $\mathcal{K}$.
We will sometimes write Flim$(\mathcal{K})$ for $\bK$.
We will assume  that $\bK$ has universe $\om$, and call such a structure
an {\em enumerated \Fraisse\ structure}.
 For $m<\om$, we let $\bK_m$
 denote the substructure of $\bK$ with universe $m = \{0, 1, \ldots , m -1\}$.

The following amalgamation property will be assumed in
this paper:
A \Fraisse\ class $\mathcal{K}$ satisfies the  {\em Disjoint Amalgamation Property}
if,
given
embeddings
$f:\bfA\ra\bfB$ and $g:\bfA\ra\bfC$, with $\bfA,\bfB,\bfC\in\mathcal{K}$,
there is an amalgam  $\bfD\in\mathcal{K}$ with  embeddings $r:\bB\ra\bfD$ and $s:\bfC\ra\bfD$ such that
$r\circ f = s\circ g$ and  moreover,
$r[\mathrm{B}]\cap s[\mathrm{C}]=r\circ f[\mathrm{A}]=s\circ g[\mathrm{A}]$.
The disjoint amalgamation property is also called  the {\em strong amalgamation property}.
It
is
 equivalent  to the {\em strong embedding property},
 which
 requires that
 for any $\bfA\in\mathcal{K}$, $v\in\mathrm{A}$, and embedding $\varphi:(\bfA-v) \ra\bK$,
there are infinitely many different extensions of $\varphi$ to embeddings of $\bfA$ into $\bK$.  (See \cite{CameronBK90}.)

A \Fraisse\ class has the {\em Free Amalgamation Property}
if  it satisfies
the Disjoint Amalgamation Property
and moreover,
the amalgam $\bfD$ can be chosen so that no tuple satisfying a relation in $\bfD$ includes elements of both
$r[\mathrm{B}] \setminus r\circ f[\mathrm{A}]$ and $s[\mathrm{C}] \setminus s\circ g[\mathrm{A}]$; in other words, $\bfD$ has
no additional relations on its universe other than those inherited from $\bfB$ and $\bfC$.

For languages $\mathcal{L}_0$ and $\mathcal{L}_1$
 such that $\mathcal{L}_0\cap\mathcal{L}_1=\emptyset$, and given
\Fraisse\ classes
$\mathcal{K}_0$ and $\mathcal{K}_1$ in $\mathcal{L}_0$ and $\mathcal{L}_1$, respectively,
 the {\em free superposition} of $\mathcal{K}_0$ and $\mathcal{K}_1$ is the \Fraisse\ class consisting of
all
 finite $(\mathcal{L}_0\cup\mathcal{L}_1)$-structures $\bfA$
 such that
  the $\mathcal{L}_i$-reduct of $\bfA$ is in $\mathcal{K}_i$,  for each $i<2$.
(See also \cite{Bodirsky15} and \cite{Hubicka_CS20}.)  Note that the free superposition of $\mathcal{K}_0$ and $\mathcal{K}_1$ has free  amalgamation if and only if each $\mathcal{K}_i$ has free amalgamation; and similarly for disjoint amalgamation.


Given a \Fraisse\ class $\mathcal{K}$ and substructures
$\bfM,\bfN$
of $\bK$  (finite or infinite)
 with $\bfM\le\bfN$,
we use
 $\binom{\bfN}{\bfM}$
to denote the set of all substructures of
$\bfN$ which are isomorphic to
$\bfM$. Given
$\bfM\le\bfN\le\bfO$,
substructures of $\bK$, we write
$$
\bfO\ra(\bfN)_{\ell}^{\bfM}
$$
to denote that for each coloring of
$\binom{\bfO}{\bfM}$
into $\ell$ colors, there is an
 $\bfN' \in \binom{\bfO}{\bfN}$
 such that
$\binom{\bfN'}{\bfM}$
is  {\em monochromatic}, meaning that
 all members of
 $\binom{\bfN'}{\bfM}$
 have the same color.



\begin{defn} \label{defn.indiv}
A \Fraisse\ structure $\bK$ is {\em indivisible} if for every singleton
substructure $\bA$ of $\bK$, 
$\bK\ra(\bK)_{\ell}^{\bfA}$ for every positive integer $\ell$.
\end{defn}

Note that when
there is only one quantifier-free 1-type over the
empty set satisfied by elements of $\bK$, so that $\bK$ has exactly one singleton substructure up to isomorphism,
indivisibility amounts to saying
that for any partition of the universe of $\bK$ into finitely many pieces, there is a subcopy of $\bK$ contained in one of the pieces.
Indivisibility
has been proved for  many structures,  including the triangle-free Henson graph in
\cite{Komjath/Rodl86},  the $k$-clique-free Henson graphs for all $k\ge 4$ in  \cite{El-Zahar/Sauer89},
more general binary relational free amalgamation structures in
\cite{Sauer03}, and for
$k$-uniform hypergraphs,
$k \ge 3$, that omit
finite substructures
in which all unordered
triples of vertices are contained in
at least one $k$-edge in
\cite{El-Zahar/Sauer94}.
For a  much broader discussion of \Fraisse\ structures and indivisibility, the reader is referred to
Nguyen Van Th\'{e}'s Habilitation \cite{NVTHabil}.


\subsection{Substructure Amalgamation Properties}\label{subsec.EEAP}

Recall that given
a \Fraisse\ class $\mathcal{K}$
in a finite relational language $\mathcal{L}$,
we let $\bK$ denote
an enumerated
\Fraisse\ limit of $\mathcal{K}$
with underlying set $\om$.
All results will hold regardless of which
enumeration
is chosen.
We make the following conventions and assumptions,
which will hold in the rest of this paper.

All types will be quantifier-free $1$-types, over a finite parameter set, that are realizable in $\bK$.
 With one exception, all such types will be complete; the exception is the case of ``passing types'', defined in  Section \ref{sec.sct}, which may be partial.
Complete types  will be denoted simply ``tp''.

We will assume that
for any relation symbol $R$ in $\mathcal{L}$, $R^{\bK}(\bar{a})$ can hold only for tuples
$\bar{a}$ of {\em distinct} elements of
$\om$.
In particular, we assume our structures have
 no loops.
We further assume that all relations
in $\bK$
 are  {\em non-trivial}:
 This means that
 for each relation symbol $R$ in $\mathcal{L}$, there exists a $k$-tuple $\bar{a}$ of (distinct) elements of $\om$ such that
 $R^{\bK}(\bar{a})$ holds, and a
$k$-tuple
  $\bar{b}$ of (distinct) elements of $\om$ such that $\neg R^{\bK}(\bar{b})$ holds.
 Since $\mathcal{K}$ has disjoint amalgamation by assumption, non-triviality will imply that there are
  infinitely many $k$-tuples from $\om$
 that satisfy $R^{\bK}$,
  and infinitely many that do not.
 We will further hold to  the following convention: 
 
 
 \begin{convention}\label{conv.unary}
 If
 $\mathcal{L}$ has at least one  unary relation
 symbol,
 then
 letting
  $R_0,\dots,$ $R_{n-1}$ list 
  them,
  we have that $n\ge 2$ and for each $\bfA\in\mathcal{K}$ and each $a\in\mathrm{A}$,
  $R_i^{\bfA}(a)$
 holds for exactly one $i<n$.
\end{convention}
By possibly adding new unary relation
symbols
to the language, any \Fraisse\ class with
unary relations can be assumed to meet this convention.

Finally, we assume that there is at least one non-unary relation symbol in $\mathcal{L}$.  This poses no real restriction,
as whenever a finite language has only unary relation symbols, any disjoint amalgamation class in that language will have a \Fraisse\ limit
that consists
of finitely many disjoint copies of $\om$, with vertices in a given copy all realizing the same quantifier-free 1-type over the empty set.
In this case, finitely many applications of Ramsey's Theorem will prove the existence of finite big Ramsey degrees.

We now present the Substructure Free Amalgamation Property.
This property also
provides the intuition behind the more general  amalgamation property \EEAP\  (Definition \ref{defn.EEAP_new}),
laying the foundation for the main ideas of this paper.

\begin{defn}[\SFAP]\label{defn.SFAP}
A \Fraisse\ class $\mathcal{K}$ has the
{\em  \SFAPnonacronym\ (\SFAP)} if $\mathcal{K}$ has free amalgamation,
and given  $\bfA,\bfB,\bfC,\bfD\in\mathcal{K}$, the following holds:
Suppose
\begin{enumerate}
\item[(1)]
$\bfA$  is a substructure of $\bfC$, where
 $\bfC$ extends  $\bfA$ by two vertices,
say $\mathrm{C}\setminus\mathrm{A}=\{v,w\}$;

\item[(2)]
 $\bfA$  is a substructure of $\bfB$ and
 $\sigma$ and $\tau$  are
 $1$-types over $\bfB$  with   $\sigma\re\bfA=\type(v/\bfA)$ and $\tau\re\bfA=\type(w/\bfA)$; and
\item[(3)]
$\bfB$ is a substructure of $\bfD$ which extends
 $\bfB$ by one vertex, say $v'$, such that $\type(v'/\bfB)=\sigma$.

\end{enumerate}
  Then there is
an   $\bfE\in\mathcal{K}$ extending  $\bfD$ by one vertex, say $w'$, such that
  $\type(w'/\bfB)=\tau$, $\bfE\re (
  \mathrm{A}\cup\{v',w'\})\cong \bfC$,
  and $\bfE$ adds no other relations over $\mathrm{D}$.
\end{defn}
The  definition of \SFAP\ can be stated using embeddings rather than substructures
in the standard way.
We remark that
requiring $\bfC$ in (1) to have only two more vertices than $\bfA$ is
sufficient for all our uses of the property in proofs of big Ramsey degrees,
and hence we have not formulated the property for $\bfC$ of arbitrary finite size.

\begin{rem}\label{rem.simple}
\SFAP\ is equivalent to
free amalgamation along with
 a model-theoretic property that may be termed {\em free 3-amalgamation}, a special case of the {\em disjoint 3-amalgamation} property defined
 in \cite{Kruckman19}: In the definition of disjoint $n$-amalgamation in Section 3 of \cite{Kruckman19}, take $n=3$ and impose the further condition that the ``solution'' or 3-amalgam disallows any relations (in any realization of the solution) that were not already stipulated in the initial 3-amalgamation ``problem''.
Kruckman shows
in \cite{Kruckman19} that if the age of a \Fraisse\ limit $\bK$ has disjoint amalgamation and disjoint 3-amalgamation, then $\bK$ exhibits a model-theoretic tameness property called  \emph{simplicity}.
\end{rem}

\SFAP\
ensures that  a finite substructure of a given
enumerated \Fraisse\ structure can be extended as desired without any requirements on its configuration  inside the larger structure.
\SFAP\
precludes any need for  the so-called ``witnessing properties'' which were necessary for the proofs of finite big Ramsey degrees
for
constrained binary
 free amalgamation classes,
as in
the  $k$-clique-free Henson graphs in \cite{DobrinenJML20} and \cite{DobrinenH_k19}, and the recent  more general extensions in  \cite{Zucker20}.
Free amalgamation classes with forbidden $3$-irreducible substructures satisfy \SFAP, as
shown in  Proposition 5.2 of Part II. 

The next amalgamation property extends \SFAP\ to disjoint amalgamation classes.
 In  the
 definition, we again use substructures rather than embeddings.

\begin{defn}[\EEAP]\label{defn.EEAP_new}
A \Fraisse\ class $\mathcal{K}$ has the
{\em  \EEAPnonacronym\ (\EEAP)} if $\mathcal{K}$
has disjoint amalgamation,
and the following holds:
Given   $\bfA, \bfC\in\mathcal{K}$, suppose that
 $\bfA$ is a substructure of $\bfC$, where   $\bfC$ extends  $\bfA$ by two vertices, say $v$ and $w$.
Then there exist  $\bfA',\bfC'\in\mathcal{K}$, where
$\bfA'$
contains a copy of $\bfA$ as a substructure
and
$\bfC'$ is a disjoint amalgamation of $\bfA'$ and $\bfC$ over $\bfA$, such that
letting   $v',w'$ denote the two vertices in
 $\mathrm{C}'\setminus \mathrm{A}'$ and
assuming (1) and (2), the conclusion holds:
 \begin{enumerate}
 \item[(1)]
Suppose
$\bfB\in\mathcal{K}$  is any structure
 containing $\bfA'$ as a substructure,
and let
 $\sigma$ and $\tau$  be
  $1$-types over $\bfB$  satisfying    $\sigma\re\bfA'=\type(v'/\bfA')$ and $\tau\re\bfA'=\type(w'/\bfA')$,
\item[(2)]
Suppose
$\bfD\in \mathcal{K}$  extends  $\bfB$
by one vertex, say $v''$, such that $\type(v''/\bfB)=\sigma$.
\end{enumerate}
Then
  there is
an  $\bfE\in\mathcal{K}$ extending   $\bfD$ by one vertex, say $w''$, such that
  $\type(w''/\bfB)=\tau$ and  $\bfE\re (
  \mathrm{A}\cup\{v'',w''\})\cong \bfC$.
\end{defn}

\begin{rem}\label{rem.freesup}
We note that
\SFAP\ implies \EEAP,
 taking $\bfA'=\bfA$ and $\bfC'=\bfC$
and because disjoint amalgamation is implied by free amalgamation.
Further, it
follows from
their definitions that \SFAP\ and \EEAP\ are each
preserved  under free superposition.
 \end{rem}

\begin{example}
The idea behind allowing for an
extension
$\bfA'$ of $\bfA$  in the definition of \EEAP\
 is most simply  demonstrated
 for
  the \Fraisse\ class
  $\mathcal{LO}$
  of finite linear orders.
Given  $\bfA,\bfC\in \mathcal{LO}$,
 suppose
 $v,w$ are the two vertices of $\mathrm{C}\setminus\mathrm{A}$ and
 suppose that
  $v<w$ holds in $\bfC$.
  We can require $\bfA'$ to  be some extension of $\bfA$ in $\mathcal{K}$
 containing some vertex $u$
  so  that
the formula $(x<u)$ is in
$\type(v/\bfA')$ and
$(u<x)$ is in
$\type(w/\bfA')$, where $x$ is a variable.
Then given any $1$-types  $\sigma,\tau$
extending $\type(v/\bfA'), \type(w/\bfA')$, respectively,
 over some structure $\bfB$ containing $\bfA'$ as a substructure,
any  two
vertices $v',w'$ satisfying $\sigma,\tau$
will automatically satisfy $v'<w'$, thus producing a copy of $\bfC$ extending $\bfA$.

In the case of  finitely many independent  linear orders,
we can similarly  produce an
$\bfA'$ which ensures that
{\em any} vertices  $v',w'$ satisfying such  $\sigma,\tau$  as above  produce a copy of
$\bfC$ extending $\bfA$.
In more general cases,  the use of $\bfA'$ only ensures that {\em there exist} such vertices $v',w'$.
\end{example}

\begin{rem} \label{rem.weak3dap}
Ivanov \cite{Ivanov99} and independently, Kechris and Rosendal \cite{KechrisRosendal07}, have formulated a weakening of the amalgamation property which is called
{\em almost amalgamation} in \cite{Ivanov99} and {\em weak amalgamation} in \cite{KechrisRosendal07}.  This property arises in the context of generic automorphisms of countable structures.  In the presence of disjoint amalgamation, \EEAP\  may be thought of as a ternary version of weak amalgamation (one of several possible such versions), and as a ``weak'' version of the disjoint 3-amalgamation property from \cite{Kruckman19} (again, one of several possible such weakenings).
\end{rem}

\begin{rem}\label{rem.observe}
We note that we could have used the definition of the free 3-amalgama\-tion property from Remark \ref{rem.simple}, and of
an appropriately formulated version of a
``weak'' disjoint 3-amalgamation property as in Remark \ref{rem.weak3dap}.
We have chosen to use Definitions  \ref{defn.SFAP} and \ref{defn.EEAP_new} instead, as
they are the forms used in the proof of Theorem \ref{thm.matrixHL}.
\end{rem}

In Section \ref{sec.sct}
onwards
we will be working with
the new notion of
{\em coding trees}
of $1$-types, which represent subcopies of a given  \Fraisse\ limit $\bK$.
For
\Fraisse\ structures in languages
with
relation symbols
of arity greater than two,
a priori, these  trees may have unbounded branching.
However, for all classes with \SFAP\ and for all classes with \EEAP\ which we have investigated, one can construct subtrees with bounded branching which still represent $\bK$.
Accordingly, we will 
formulate the
strengthened version \EEAP$^+$ of \EEAP\ in Subsection \ref{subsec.SDAPplus}
which imposes conditions on the branching in a coding tree for $\bK$.



\section{Coding trees of \texorpdfstring{$1$-types for \Fraisse\ structures}{codingtrees}}\label{sec.sct}

Fix throughout a \Fraisse\ class $\mathcal{K}$ in a finite  relational language $\mathcal{L}$.
Recall that  $\bK$  denotes
an {\em enumerated \Fraisse\ limit}
for $\mathcal{K}$,  meaning that
 $\bK$ has universe $\om$.
In order to  avoid confusion, we shall usually use $v_n$ instead of just $n$
to denote the $n$-th member of the universe  of $\bK$, and we shall call this  the {\em $n$-th  vertex} of $\bK$.
For $n < \om$,  we write $\bK_n$, and sometimes $\bK\re n$, to denote the
substructure of $\bK$ on the set of vertices $\{v_i:i<n\}$.
We call $\bK_n$ an {\em initial segment} of $\bK$.
Note that $\bK_0$ is the empty structure.

In Subsection \ref{subsec.3.1},
we present
a general  construction of
 trees of complete $1$-types over initial segments of $\bK$,
which we call {\em coding trees}.
Graphics of coding trees
are then presented for various  prototypical \Fraisse\ classes which will be proved in Part II to satisfy \EEAP$^+$.
In Subsection \ref{subsec.3.2},
we define
{\em passing types}, extending  the notion
of  {\em passing number}
due to Laflamme, Sauer, and Vuksanovic in \cite{Laflamme/Sauer/Vuksanovic06},  which has been central to all prior results on big Ramsey degrees   for binary relational structures.
Then we
extend the notion
from \cite{Laflamme/Sauer/Vuksanovic06}
of {\em similarity type} for binary relational structures
to
structures with relations of any arity.
In Subsection \ref{subsec.3.3}, we introduce
 {\em diagonal}  coding trees.
These  will be
 key
to obtaining indivisibility directly from Theorem \ref{thm.matrixHL} as well as
 precise big Ramsey degree results without appeal to
 the method of
 envelopes in Part II.
We
define the \emph{Diagonal Coding Tree Property}, one of the conditions for \EEAP$^+$ to hold.



\subsection{Coding trees of $1$-types}\label{subsec.3.1}

All types will be quantifier-free 1-types, with variable $x$, over some finite initial segment of $\bK$.
For $n \ge 1$, a type over $\bK_n$ must contain the formula
$\neg (x = v_i ) $ for each $i < n$.
Given a type $s$ over $\bK_n$, for any $i < n$, $s \re \bK_i$ denotes the restriction of $s$ to parameters from $\bK_i$.
Recall that the notation ``tp'' denotes a complete quantifer-free 1-type.

\begin{defn}[The  Coding  Tree of $1$-Types, $\bS(\bK)$]\label{defn.treecodeK}
The {\em coding  tree of $1$-types}
$\bS(\bK)$
for an enumerated \Fraisse\ structure $\bK$
 is the set of all complete
  $1$-types over initial segments of $\bK$
along with a function $c:\om\ra \bS(\bK)$ such that
$c(n)$ is the
$1$-type
of $v_n$
over $\bK_n$.
The tree-ordering is simply inclusion.
\end{defn}

We shall usually simply write $\bS$, rather than $\bS(\bK)$.
Note that we make no requirement at this point on $\bK$;
an enumerated \Fraisse\ limit of  any \Fraisse\ class (with no reference to its amalgamation or Ramsey properties)
naturally induces a  coding tree of $1$-types as above.
We say that $c(n)$ {\em represents} or {\em codes}  the vertex  $v_n$.
Instead of writing $c(n)$, we shall usually write $c_n$ for the $n$-th coding node  in $\bS$.

We let  $\bS(n)$
denote
 the collection
 of all   $1$-types  $\type(v_i/\bK_n)$, where
 $i\ge n$.
 Note that each $c(n)$ is a node in $\bS(n)$.
The set  $\bS(0)$ consists of
 the
 $1$-types over the empty structure $\bK_0$.
For $s\in \bS(n)$, the immediate successors of $s$   are  exactly those $t\in \bS(n+1)$ such that $s\sse t$.
For each $n<\om$, the set  $\bS(n)$ is finite, since the language
$\mathcal{L}$  consists of finitely many finitary relation symbols.

We say that each  node $s\in\bS(n)$ has {\em length}  $n+1$, and denote the length of $s$ by $|s|$.
Thus, all nodes in $\bS$ have length at least one.
While it is slightly unconventional to consider the roots
of $\bS$ as having length one,  this approach lines up with the  natural correspondence between  nodes in $\bS$ and certain sequences of partial $1$-types
that  we define
in the next paragraph.
The reader wishing for a  tree starting with a node of length zero  may consider adding the
empty set  to $\bS$, as this will have no
effect
on the results in this paper.
A {\em level set} is a subset $X\sse\bS$ such that all nodes in $X$ have the same length.

Let  $n<\om$ and   $s\in\bS(n)$ be given.
We let  $s(0)$  denote the set of formulas in $s$ involving
no parameters;
$s(0)$ is the unique member of $\bS(0)$ such that $s(0)\sse s$.
For  $1\le i\le n$,
we let
 $s(i)$ denote  the set of those formulas in $s\re \bK_i$
in which $v_{i-1}$ appears; in other words, the formulas in $s \re \bK_i$ that are not in  $s \re \bK_{i-1}$.
In this manner,  each $s\in \bS$ determines a unique sequence
$\lgl s(i):i< |s|\rgl$, where $\{s(i):i< |s|\}$ forms a partition of $s$.
For $j< |s|$, $\bigcup_{i\le j}s(i)$  is the  node in $\bS(j)$
such that $\bigcup_{i\le  j}s(i)\sse s$.
For $\ell\le |s|$,
we shall usually write $s \re \ell$ to denote $\bigcup_{i<\ell}s(i)$.

Given $s,t\in \bS$,
we define the {\em meet} of $s$ and $t$, denoted $s\wedge t$, to be
$s \re \bK_m$ for the maximum  $m\le \min(|s|,|t|)$ such that $s \re \bK_m=t \re \bK_m$.
It can be useful to think of  $s\in \bS$ as the sequence $\lgl s(0),\dots, s(|s|-1)\rgl$;
then $s\wedge t$ can be interpreted in the  usual way for trees of sequences.

It will be useful later  to have  specific notation for unary relations.
We will let $\Gamma$
denote $\bS(0)$,  the set of  complete
$1$-types
over the empty set
that are realized in $\bK$.
For $\gamma\in\Gamma$, we write
``$\gamma(v_n)$ holds in $\bK$'' when $\gamma$ is the 1-type of $v_n$ over the empty set; in practice,
it will be the unary relation symbols in $\gamma$ (if there are any) that will be of interest to us.

 \begin{rem}
 Our definition of $s(i)$
sets up   for the   definition of passing type
in Subsection \ref{subsec.3.2},
which  directly abstracts the notion of passing number used in  \cite{Sauer06} and
  \cite{Laflamme/Sauer/Vuksanovic06}, and in subsequent papers building on their ideas.
\end{rem}

\begin{rem}
In the case where all relation symbols in the language $\mathcal{L}$ have arity at most two,
  the  coding  tree  of $1$-types $\bS$  has bounded branching.
  If $\mathcal{L}$ has any relation symbol of arity three or greater, then  $\bS$ may
  have branching which increases as the levels increase.
If such a \Fraisse\ class satisfies \EEAP,  sometimes more work  still  must be  done  in order to guarantee
that its \Fraisse\ limit has
\EEAP$^+$.
\end{rem}


We now provide graphics for   coding trees of $1$-types which are prototypical for  \Fraisse\ classes which 
have \EEAP.
That their \Fraisse\ limits satisfy SDAP$^+$ will be proved in Section 5 of  Part II.
We start with the
rational linear order,
since its coding tree of $1$-types is the simplest, and  also because  the rationals were the first \Fraisse\ structure for which big Ramsey degrees were characterized (Devlin, \cite{DevlinThesis}).

\begin{example}[The coding tree of $1$-types  $\bS(\bQ)$]\label{ex.Q}
Figure \ref{fig.Qtree}
shows the coding tree of $1$-types for  $(\bQ,<)$, the rationals as a linear order.
This is the \Fraisse\ limit of $\mathcal{LO}$, the class of finite linear orders.
We assume that the universe  of $\bQ$  is linearly ordered in order-type $\om$ as $\lgl v_n:n<\om\rgl$.
For each $n$, the coding node $c_n$
is the $1$-type of
vertex $v_n$  over the initial segment $\{v_i:i<n\}$ of $\bQ$.
(Recall that $x$ is the variable in all  of our $1$-types.)
Thus, the coding node $c_0$ is the empty $1$-type, and
 $c_1$   is the  $1$-type
$\{v_0<x\}$.
Thus,  the coding nodes $\{c_0,c_1\}$ represent the linear order $v_0<v_1$.
Likewise, the coding node $c_2$ is  the $1$-type
$\{x<v_0,x<v_1\}$
 over the linear order $v_0<v_1$.
Hence, $c_2$ represents the vertex $v_2$ satisfying
$v_2<v_0<v_1$.
The coding node $c_3$ is the $1$-type $\{v_0<x,x<v_1,v_2<x\}$, so  $c_3$ represents the vertex $v_3$ satisfying
$v_2<v_0<v_3<v_1$.
Below the  tree, we  picture the linear order on the vertices $v_0,\dots,v_5$ induced by the coding nodes.
As the tree grows in height, the linear order represented by the coding nodes  grows into  the countable dense linear order with no endpoints.

Notice that only the coding nodes branch.
This is because
of the rigidity of the rationals:
Given  a non-coding node $s$ on the same level as a coding node $c_{n}$ (say $n\ge 1$),
$s$ is a  $1$-type which is satisfied by any vertex
which lies in some interval determined by the vertices $\{v_i:i< n\}$, and  $v_n$ is not in that interval.
Thus, the order between $v_n$ and any vertex satisfying $s$ is
predetermined,
so $s$ does not split.
Said another way,
letting $m$ denote the length of the meet of  $c_n$ and $s$,
$c_n$  and $s$ must disagree on the formula $x<v_{m}$;
hence,  $x<v_m$ is in $c_n$   if and only if  $v_m<x$ is in
  $s$.
In the case that the formula $x<v_m$ is in $c_n$,
then
it follows that  $v_n<v_m$.
On the other hand, any realization $v_i$ of the $1$-type $s$ must satisfy $v_m<v_i$.
Hence  every realization of $s$ by some vertex  $v_i$ must satisfy $v_n<v_i$.
Thus, there is only one immediate successor of $s$ in the tree of $1$-types.
The tree of $1$-types for $\bQ$
 eradicates the extraneous structure
which appears in the more traditional approach of using the full binary branching tree and Milliken's Theorem  to approach big Ramsey degrees of the rationals.
\end{example}


\begin{figure}
\begin{tikzpicture}[grow'=up,scale=.6]
\tikzstyle{level 1}=[sibling distance=4in]
\tikzstyle{level 2}=[sibling distance=2in]
\tikzstyle{level 3}=[sibling distance=1in]
\tikzstyle{level 4}=[sibling distance=0.5in]
\tikzstyle{level 5}=[sibling distance=0.2in]
\tikzstyle{level 6}=[sibling distance=0.1in]
\tikzstyle{level 7}=[sibling distance=0.07in]
\node [label=$c_0$] {} coordinate (t9)
child{ coordinate (t0) edge from parent[color=black,thick]
child{ coordinate (t00) edge from parent[color=black,thick]
child{ coordinate (t000) edge from parent[color=black,thick]
child{coordinate (t0000) edge from parent[color=black,thick]
child{coordinate (t00000) edge from parent[color=black,thick]
child{coordinate (t000000) edge from parent[color=black,thick]}
}
}
}
child{ coordinate (t001) edge from parent[color=black,thick]
child{ coordinate (t0010) edge from parent[color=black,thick]
child{coordinate (t00100) edge from parent[color=black,thick]
child{coordinate (t001000) edge from parent[color=black,thick]}
child{coordinate (t001001) edge from parent[color=black,thick]}
}
}
}
}
}
child{ coordinate (t1) edge from parent[color=black,thick]
child{ coordinate (t10) edge from parent[color=black,thick]
child{ coordinate (t100)
child{coordinate (t1000)
child{coordinate (t10000)
child{coordinate (t100000)}
}
}
child{coordinate (t1001)
child{coordinate (t10010)
child{coordinate (t100100)}
}
}
}
}
child{ coordinate (t11) [label={\small 1/16}]
child{coordinate (t110)
child{coordinate (t1100)
child{coordinate (t11000)
child{coordinate (t110000)}
}
child{coordinate (t11001)
child{coordinate (t110010)}
}
}
}
}
}
;

\node[circle, fill=blue,inner sep=0pt, minimum size=5pt] at (t9) {};
\node[circle, fill=blue,inner sep=0pt, minimum size=5pt,label=0:$c_2$,label=250:$\scriptstyle{x<v_1}$] at (t00) {};
\node[circle, fill=blue,inner sep=0pt, minimum size=5pt,label=0:$c_5$,label=280:${\scriptstyle x<v_4}$] at (t00100) {};
\node[circle, fill=blue,inner sep=0pt, minimum size=5pt, label=$c_1$,label=250:$\scriptstyle{v_0<x}$] at (t1) {};
\node[circle, fill=blue,inner sep=0pt, minimum size=5pt, label = 0:$c_3$,label=240:$\scriptstyle{v_2<x}$] at (t100) {};
\node[circle, fill=blue,inner sep=0pt, minimum size=5pt,label = 0:$c_4$,label=300:$\scriptstyle{v_3<x}$] at (t1100) {};
\node[label=280:$\scriptstyle{x<v_0}$] at (t0) {};
\node[label=270:$\scriptstyle{v_1 < x}$] at (t11) {};
\node[label=270:$\scriptstyle{x<v_1}$] at (t10) {};
\node[label=300:$\scriptstyle{v_2 < x}$] at (t110) {};
\node[label=240:${\scriptstyle x<v_3}$] at (t1000) {};
\node[label=280:${\scriptstyle v_3 < x}$] at (t1001) {};
\node[label=240:${\scriptstyle x<v_4}$] at (t10000) {};
\node[label=300:${\scriptstyle x < v_4}$] at (t10010) {};
\node[label=240:${\scriptstyle v_5<x}$] at (t100000) {};
\node[label=300:${\scriptstyle v_5<x}$] at (t100100) {};
\node[label=260:${\scriptstyle x<v_4}$] at (t11000) {};
\node[label=280:${\scriptstyle v_4<x}$] at (t11001) {};
\node[label=260:${\scriptstyle v_5<x}$] at (t110000) {};
\node[label=280:${\scriptstyle v_5<x}$] at (t110010) {};
\node[label=260:${\scriptstyle x<v_2}$] at (t000) {};
\node[label=280:${\scriptstyle v_2<x}$] at (t001) {};
\node[label=260:${\scriptstyle x<v_3}$] at (t0000) {};
\node[label=280:${\scriptstyle x<v_3}$] at (t0010) {};
\node[label=260:${\scriptstyle x<v_4}$] at (t00000) {};
\node[label=260:${\scriptstyle x<v_5}$] at (t000000) {};
\node[label=260:${\scriptstyle x<v_5}$] at (t001000) {};
\node[label=280:${\scriptstyle v_5<x}$] at (t001001) {};
\node[label=280:${\scriptstyle v_5<x}$] at (t001001) {};

\node[circle,inner sep=0pt, minimum size=5pt,label=90:$\vdots$] at (t000000) {};
\node[circle,inner sep=0pt, minimum size=5pt,label=90:$\vdots$] at (t001000) {};
\node[circle,inner sep=0pt, minimum size=5pt,label=90:$\vdots$] at (t001001) {};
\node[circle,inner sep=0pt, minimum size=5pt,label=90:$\vdots$] at (t100000) {};
\node[circle,inner sep=0pt, minimum size=5pt,label=90:$\vdots$] at (t100100) {};
\node[circle,inner sep=0pt, minimum size=5pt,label=90:$\vdots$] at (t100100) {};
\node[circle,inner sep=0pt, minimum size=5pt,label=90:$\vdots$] at (t110000) {};
\node[circle,inner sep=0pt, minimum size=5pt,label=90:$\vdots$] at (t110010) {};

\node[circle, fill=white,draw,inner sep=0pt, minimum size=4pt] at (t0) {};
\node[circle, fill=white,draw,inner sep=0pt, minimum size=4pt] at (t000) {};
\node[circle, fill=white,draw,inner sep=0pt, minimum size=4pt] at (t0000) {};
\node[circle, fill=white,draw,inner sep=0pt, minimum size=4pt] at (t00000) {};
\node[circle, fill=white,draw,inner sep=0pt, minimum size=4pt] at (t000000) {};
\node[circle, fill=white,draw,inner sep=0pt, minimum size=4pt] at (t001) {};
\node[circle, fill=white,draw,inner sep=0pt, minimum size=4pt] at (t0010) {};
\node[circle, fill=white,draw,inner sep=0pt, minimum size=4pt] at (t001000) {};
\node[circle, fill=white,draw,inner sep=0pt, minimum size=4pt] at (t001001) {};
\node[circle, fill=white,draw,inner sep=0pt, minimum size=4pt] at (t10) {};
\node[circle, fill=white,draw,inner sep=0pt, minimum size=4pt] at (t1000) {};
\node[circle, fill=white,draw,inner sep=0pt, minimum size=4pt] at (t10000) {};
\node[circle, fill=white,draw,inner sep=0pt, minimum size=4pt] at (t100000) {};
\node[circle, fill=white,draw,inner sep=0pt, minimum size=4pt] at (t1001) {};
\node[circle, fill=white,draw,inner sep=0pt, minimum size=4pt] at (t10010) {};
\node[circle, fill=white,draw,inner sep=0pt, minimum size=4pt] at (t100100) {};
\node[circle, fill=white,draw,inner sep=0pt, minimum size=4pt] at (t11) {};
\node[circle, fill=white,draw,inner sep=0pt, minimum size=4pt] at (t110) {};
\node[circle, fill=white,draw,inner sep=0pt, minimum size=4pt] at (t11000) {};
\node[circle, fill=white,draw,inner sep=0pt, minimum size=4pt] at (t110000) {};
\node[circle, fill=white,draw,inner sep=0pt, minimum size=4pt] at (t11001) {};
\node[circle, fill=white,draw,inner sep=0pt, minimum size=4pt] at (t110010) {};

\node[circle, fill=blue,inner sep=0pt, minimum size=5pt, label=$v_2$,below=3cm of t00] (v2) {};
\node[circle, fill=blue,inner sep=0pt, minimum size=5pt,label=$v_5$,below =3.9cm of t001] (v5) {};
\node[circle, fill=blue,inner sep=0pt, minimum size=5pt,label=$v_0$,below =1.25cm of t9] (v0) {};
\node[circle, fill=blue,inner sep=0pt, minimum size=5pt,label=$v_3$,below =3.95cm of t100] (v3) {};
\node[circle, fill=blue,inner sep=0pt, minimum size=5pt,label=$v_1$,below =2.1cm of t1] (v1) {};
\node[circle, fill=blue,inner sep=0pt, minimum size=5pt,label=$v_4$,below =3cm of t11] (v4) {};
\end{tikzpicture}
\caption{Coding tree of $1$-types for $(\bQ,<)$ and the  linear order represented by its coding nodes.}\label{fig.Qtree}
\end{figure}


\begin{example}[The coding tree of $1$-types $\bS(\bQ_2)$]\label{ex.ctbQ_2}
Next, we consider coding trees of $1$-types for linear orders with equivalence relations with finitely many equivalence classes, each of which is dense in the linear order.
Figure \ref{fig.Q2tree} provides a graphic for   the coding tree of $1$-types for the structure
$\bQ_2$,  the rationals with an equivalence relation with two equivalence classes which are each dense in  the linear order.
 We point out that
$c_0$ is the $1$-type $\{U_1(x)\}$,
$c_1$ is the $1$-type $\{U_0(x), x<v_0\}$,
$c_2$ is the $1$-type $\{U_0(x), v_0<x, v_1<x\}$, etc.

Note  that $\bS(\bQ_2)$ looks like two identical disjoint copies of a coding tree for $\bQ$.
This is because each of the two unary relations, representing the two equivalence classes, appears densely in the linear order.
The ordered structure $\bQ_2$ appears below the two trees as the vertices $v_0,v_1,\dots$.
Unlike Figure \ref{fig.Qtree} for $\bQ$, the  vertices in $\bQ_2$ do not line up below the coding nodes in the trees representing them, since $\bS(\bQ_2)$ has two roots.
However, if we modify our definition of coding tree of $1$-types to  have individual  coding nodes  $c_n$ represent the unary relations satisfied by $v_n$
(rather than $\bS$  having   $|\Gamma|$ many roots),
this has the effect of producing a one-rooted tree with
``$\gamma$-colored'' coding nodes appearing cofinally in the tree, for each $\gamma\in\Gamma$.
This approach then shows the linear order $\bQ_2$ lining up below the coding nodes,  recovers the characterization of the big Ramsey degrees in \cite{Laflamme/NVT/Sauer10}, and
 will aid us in proving \EEAP$^+$ for $\mathcal{P}_2$.
 (See Definition \ref{defn.ctU}
 for this variation of tree of $1$-types, which reproduces the approach in \cite{Laflamme/NVT/Sauer10}.)

Similarly, for any $n\ge 2$,  $\bS(\bQ_n)$ will have  $n$ roots, and  above each root, the $n$ trees will  be copies of each other.
\end{example}

\begin{figure}
\begin{center}
\begin{minipage}{.2\textwidth}
\hspace*{-4cm}
\begin{tikzpicture}[grow'=up,scale=.3]
\tikzstyle{level 1}=[sibling distance=4in]
\tikzstyle{level 2}=[sibling distance=2in]
\tikzstyle{level 3}=[sibling distance=1in]
\tikzstyle{level 4}=[sibling distance=0.5in]
\tikzstyle{level 5}=[sibling distance=0.2in]
\tikzstyle{level 6}=[sibling distance=0.1in]
\tikzstyle{level 7}=[sibling distance=0.07in]

\node {} coordinate (t9)
child{ coordinate (t0) edge from parent[color=black,thick]
child{coordinate (t00) edge from parent[color=black,thick]
child{coordinate (t000) edge from parent[color=black,thick]
child{coordinate (t0000) edge from parent[color=black,thick]
child{coordinate (t00000) edge from parent[color=black,thick]}
}
child{coordinate (t0001) edge from parent[color=black,thick]
child{coordinate (t00010) edge from parent[color=black,thick]}
}
}
}
child{coordinate (t01) edge from parent[color=black,thick]
child{coordinate (t010) edge from parent[color=black,thick]
child{coordinate (t0100) edge from parent[color=black,thick]
child{coordinate (t01000) edge from parent[color=black,thick]}
}
}
}
}
child{ coordinate (t1) edge from parent[color=black,thick]
child{coordinate (t10) edge from parent[color=black,thick]
child{coordinate (t100) edge from parent[color=black,thick]
child{coordinate (t1000) edge from parent[color=black,thick]
child{coordinate (t10000) edge from parent[color=black,thick]}
child{coordinate (t10001) edge from parent[color=black,thick]}
}
}
child{coordinate (t101) edge from parent[color=black,thick]
child{coordinate (t1010) edge from parent[color=black,thick]
child{coordinate (t10100) edge from parent[color=black,thick]}
}
}
}
};
%
\node[label=90:$\scriptstyle{c_1}$,label=10:\hspace*{3mm}$\scriptstyle{v_1<x}$,label=170:$\scriptstyle{x<v_1}$\hspace*{3mm},circle, fill=red,inner sep=0pt, minimum size=5pt] at (t0) {};
\node[label=90:$\scriptstyle{c_2}$,label=150:$\scriptstyle{x<v_2}\hspace{2mm}$,label=20:$\hspace{2mm}\scriptstyle{v_2<x}$,circle, fill=red,inner sep=0pt, minimum size=5pt] at (t10) {};
\node[label=0:$\scriptstyle{c_5}$,circle, fill=red,inner sep=0pt, minimum size=5pt] at (t00010) {};

\node[circle,inner sep=0pt, minimum size=5pt,label=90:$\vdots$] at (t00000) {};
\node[circle,inner sep=0pt, minimum size=5pt,label=90:$\vdots$] at (t00010) {};
\node[circle,inner sep=0pt, minimum size=5pt,label=90:$\vdots$] at (t01000) {};
\node[circle,inner sep=0pt, minimum size=5pt,label=90:$\vdots$] at (t10000) {};
\node[circle,inner sep=0pt, minimum size=5pt,label=90:$\vdots$] at (t10001) {};
\node[circle,inner sep=0pt, minimum size=5pt,label=90:$\vdots$] at (t10100) {};

\node[label=180:$\scriptstyle{x<v_0}\hspace{6mm}$,label=270:$\scriptstyle{U_0(x)}$,label=0:$\hspace{6mm}\scriptstyle{v_0<x}$,circle, fill=white,draw,inner sep=0pt, minimum size=4pt] at (t9) {};
\node[label=90:$\scriptstyle{x<v_2}\hspace{14mm}$,circle, fill=white,draw,inner sep=0pt, minimum size=4pt] at (t00) {};
\node[circle, fill=white,draw,inner sep=0pt, minimum size=4pt] at (t000) {};
\node[circle, fill=white,draw,inner sep=0pt, minimum size=4pt] at (t0000) {};
\node[circle, fill=white,draw,inner sep=0pt, minimum size=4pt] at (t00000) {};
\node[circle, fill=white,draw,inner sep=0pt, minimum size=4pt] at (t0001) {};
\node[label=90:$\hspace{11mm}\scriptstyle{x<v_2}$,circle, fill=white,draw,inner sep=0pt, minimum size=4pt] at (t01) {};
\node[circle, fill=white,draw,inner sep=0pt, minimum size=4pt] at (t010) {};
\node[circle, fill=white,draw,inner sep=0pt, minimum size=4pt] at (t0100) {};
\node[circle, fill=white,draw,inner sep=0pt, minimum size=4pt] at (t01000) {};
\node[label=20:$\scriptstyle{v_1<x}$,circle, fill=white,draw,inner sep=0pt, minimum size=4pt] at (t1) {};
\node[circle, fill=white,draw,inner sep=0pt, minimum size=4pt] at (t100) {};
\node[circle, fill=white,draw,inner sep=0pt, minimum size=4pt] at (t1000) {};
\node[circle, fill=white,draw,inner sep=0pt, minimum size=4pt] at (t10000) {};
\node[circle, fill=white,draw,inner sep=0pt, minimum size=4pt] at (t10001) {};
\node[circle, fill=white,draw,inner sep=0pt, minimum size=4pt] at (t101) {};
\node[circle, fill=white,draw,inner sep=0pt, minimum size=4pt] at (t1010) {};
\node[circle, fill=white,draw,inner sep=0pt, minimum size=4pt] at (t10100) {};
\end{tikzpicture}
\end{minipage}
\begin{minipage}{0.2\textwidth}
\hspace*{-0.5cm}
\begin{tikzpicture}[grow'=up,scale=.3]
\tikzstyle{level 1}=[sibling distance=4in]
\tikzstyle{level 2}=[sibling distance=2in]
\tikzstyle{level 3}=[sibling distance=1in]
\tikzstyle{level 4}=[sibling distance=0.5in]
\tikzstyle{level 5}=[sibling distance=0.2in]
\tikzstyle{level 6}=[sibling distance=0.1in]
\tikzstyle{level 7}=[sibling distance=0.07in]

\node {} coordinate (u9)
child{coordinate (u0) edge from parent[color=black,thick]
child{coordinate (u00) edge from parent[color=black,thick]
child{coordinate (u000) edge from parent[color=black,thick]
child{coordinate (u0000) edge from parent[color=black,thick]
child{coordinate (u00000) edge from parent[color=black,thick]}
}
child{coordinate (u0001) edge from parent[color=black,thick]
child{coordinate (u00010) edge from parent[color=black,thick]}
}
}
}
child{coordinate (u01) edge from parent[color=black,thick]
child{coordinate (u010) edge from parent[color=black,thick]
child{coordinate (u0100) edge from parent[color=black,thick]
child{coordinate (u01000) edge from parent[color=black,thick]}
}
}
}
}
child{coordinate (u1) edge from parent[color=black,thick]
child{coordinate (u10) edge from parent[color=black,thick]
child{coordinate (u100) edge from parent[color=black,thick]
child{coordinate (u1000) edge from parent[color=black,thick]
child{coordinate (u10000) edge from parent[color=black,thick]}
child{coordinate (u10001) edge from parent[color=black,thick]}
}
}
child{coordinate (u101) edge from parent[color=black,thick]
child{coordinate (u1010) edge from parent[color=black,thick]
child{coordinate (u10100) edge from parent[color=black,thick]}
}
}
}
};

\node[label=90:$\scriptstyle{c_0}$, label=180:$\scriptstyle{x<v_0}\hspace{6mm}$,label=0:$\hspace{6mm}\scriptstyle{v_0<x}$,label=270:$\scriptstyle{U_1(x)}$,circle, fill=blue,inner sep=0pt, minimum size=5pt] at (u9) {};
\node[label=0:$\scriptstyle{c_3}$,circle, fill=blue,inner sep=0pt, minimum size=5pt] at (u000) {};
\node[label=180:$\scriptstyle{c_4}$,circle, fill=blue,inner sep=0pt, minimum size=5pt] at (u1000) {};

\node[circle,inner sep=0pt, minimum size=5pt,label=90:$\vdots$] at (u00000) {};
\node[circle,inner sep=0pt, minimum size=5pt,label=90:$\vdots$] at (u00010) {};
\node[circle,inner sep=0pt, minimum size=5pt,label=90:$\vdots$] at (u01000) {};
\node[circle,inner sep=0pt, minimum size=5pt,label=90:$\vdots$] at (u10000) {};
\node[circle,inner sep=0pt, minimum size=5pt,label=90:$\vdots$] at (u10001) {};
\node[circle,inner sep=0pt, minimum size=5pt,label=90:$\vdots$] at (u10100) {};

\node[label=150:$\scriptstyle{x<v_1}\hspace{5mm}$,label=30:$\hspace{5mm}\scriptstyle{v_1<x}$,circle, fill=white,draw,inner sep=0pt, minimum size=4pt] at (u0) {};
\node[label=120:$\scriptstyle{x<v_2}$,circle, fill=white,draw,inner sep=0pt, minimum size=4pt] at (u00) {};
\node[circle, fill=white,draw,inner sep=0pt, minimum size=4pt] at (u0000) {};
\node[circle, fill=white,draw,inner sep=0pt, minimum size=4pt] at (u00000) {};
\node[circle, fill=white,draw,inner sep=0pt, minimum size=4pt] at (u0001) {};
\node[circle, fill=white,draw,inner sep=0pt, minimum size=4pt] at (u00010) {};
\node[label=30:$\scriptstyle{x<v_2}$,circle, fill=white,draw,inner sep=0pt, minimum size=4pt] at (u01) {};
\node[circle, fill=white,draw,inner sep=0pt, minimum size=4pt] at (u010) {};
\node[circle, fill=white,draw,inner sep=0pt, minimum size=4pt] at (u0100) {};
\node[circle, fill=white,draw,inner sep=0pt, minimum size=4pt] at (u01000) {};
\node[label=30:$\scriptstyle{v_1<x}$,circle, fill=white,draw,inner sep=0pt, minimum size=4pt] at (u1) {};
\node[label=150:$\scriptstyle{x<v_2}\hspace{2mm}$,label=30:$\hspace{2mm}\scriptstyle{v_2<x}$,circle, fill=white,draw,inner sep=0pt, minimum size=4pt] at (u10) {};
\node[circle, fill=white,draw,inner sep=0pt, minimum size=4pt] at (u100) {};
\node[circle, fill=white,draw,inner sep=0pt, minimum size=4pt] at (u10000) {};
\node[circle, fill=white,draw,inner sep=0pt, minimum size=4pt] at (u10001) {};
\node[circle, fill=white,draw,inner sep=0pt, minimum size=4pt] at (u101) {};
\node[circle, fill=white,draw,inner sep=0pt, minimum size=4pt] at (u1010) {};
\node[circle, fill=white,draw,inner sep=0pt, minimum size=4pt] at (u10100) {};
\end{tikzpicture}
\end{minipage}

\end{center}
\begin{tikzpicture}
\node[circle, fill=blue,inner sep=0pt, minimum size=5pt, label=$v_3$] (v3) {};
\node[circle, fill=red,inner sep=0pt, minimum size=5pt,label=$v_5$,right=1.6cm of v3] (v5) {};
\node[circle, fill=red,inner sep=0pt, minimum size=5pt,label=$v_1$,right=1.6cm of v5] (v1) {};
\node[circle, fill=blue,inner sep=0pt, minimum size=5pt, label=$v_0$,right=1.6cm of v1] (v0) {};
\node[circle, fill=blue,inner sep=0pt, minimum size=5pt,label=$v_4$,right=1.6cm of v0] (v4) {};
\node[circle, fill=red,inner sep=0pt, minimum size=5pt,label=$v_2$,right=1.6cm of v4] (v2) {};
\end{tikzpicture}
\caption{Coding tree of $1$-types for $(\mathbb{Q}_2,<)$ and the linear order represented by its coding nodes.}
\label{fig.Q2tree}
\end{figure}
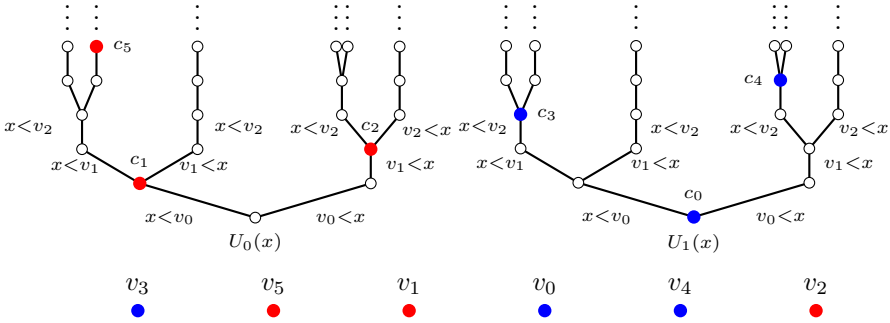


\begin{example}[The coding tree of $1$-types  $\bS(\bQ_{\bQ})$]\label{ex.Q_Q}
Next, we present the tree of $1$-types  for the  \Fraisse\ structure $\bQ_{\bQ}$.
Recall that  this  is the \Fraisse\ limit of the class $\mathcal{CO}$ in the language $\mathcal{L}=\{<,E\}$, where $<$ is a linear order and $E$ is a convexly ordered
equivalence relation, meaning that all equivalence classes are intervals.

Figure \ref{fig.QQtree} shows  the first six levels of a coding tree of $1$-types, $\bS(\bQ_{\bQ})$.
The formulas which are in the $1$-types can be read from the graphic.  For instance,
$c_0$ is the empty type.
$c_1$ is the $1$-type $\{v_0<x, E(x,v_0)\}$,
so since the vertex $v_1$ satisfies this $1$-type, we have
$v_0<v_1$ and $E(v_0,v_1)$ holding.
Similarly,
$c_2$ is the $1$-type
$\{ x<v_0, \neg E(x,v_0), x<v_1, \neg E(x,v_1)\}$,
so $v_2$ satisfies  $v_2<v_0$, and $v_2$ is not equivalent to either of $v_0$ or $v_1$.
 $c_3$ is the $1$-type
$\{v_0<x,E(x,v_0), v_1<x, E(x,v_1), v_2<x,\neg E(x,v_2)\}$, and hence, we see that $v_1<v_3$ and $v_3$ is equivalent to $v_1$ and hence also to $v_0$.
Note that only coding nodes branch.
Moreover,  $c_n$ has splitting  degree two if $c_n$ represents a vertex $v_n$ which is equivalent to $v_i$ for some $i<n$;
 otherwise
$c_n$ has  splitting  degree four.
 For each non-coding node $s$ on the level of a coding node $c_n$, there is only one possible $1$-type extending $s$ over the initial structure on the first $n+1$ vertices of  $\bQ_{\bQ}$.
We will show in Theorem 5.10 in Part II
that  $\bQ_{\bQ}$  satisfies   \EEAP$^+$.

In Figure \ref{fig.QQtree},
below the tree  $\bS(\bQ_{\bQ})$ is the linear order  on the vertices  $v_0,\dots, v_5$ represented  by the coding nodes $c_0,\dots,c_5$;  the lines between the vertices represent that they are in the same equivalence class.  Thus, $v_0,v_1,v_3$ are all in the same equivalence class, $v_2, v_4$ are in a different equivalence class, and $v_5$ is in yet another equivalence class.
\end{example}


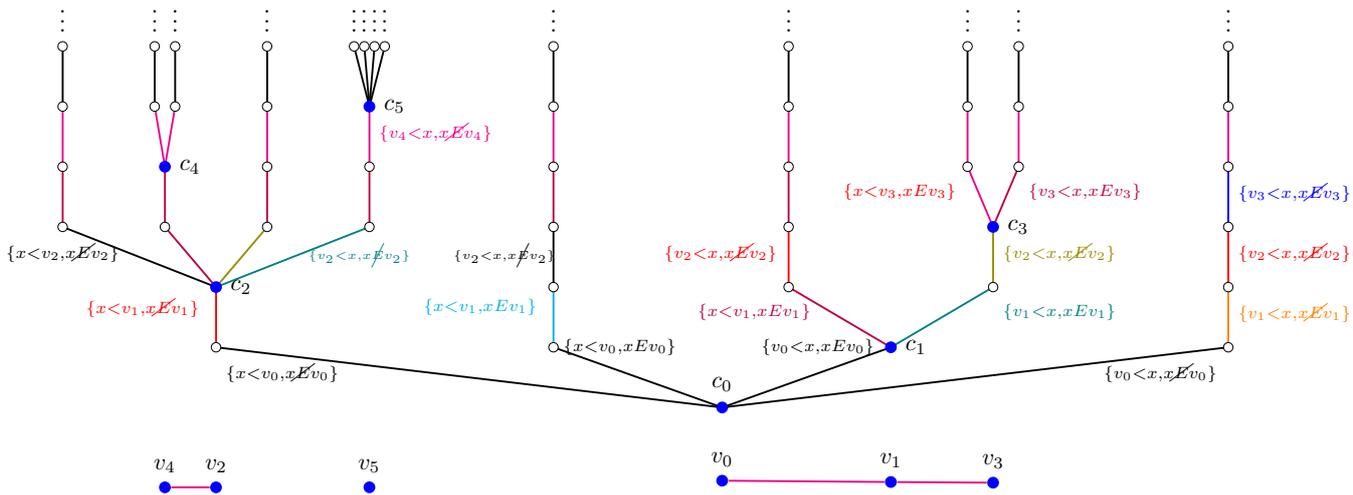
\begin{sidewaysfigure}
\vspace{13cm}
\resizebox{18cm}{!}{
\begin{tikzpicture}[grow'=up,scale=.60]
\tikzstyle{level 1}=[sibling distance=3.3in]
\tikzstyle{level 2}=[sibling distance=2in]
\tikzstyle{level 3}=[sibling distance=1in]
\tikzstyle{level 4}=[sibling distance=0.5in]
\tikzstyle{level 5}=[sibling distance=0.2in]
\tikzstyle{level 6}=[sibling distance=0.1in]
\tikzstyle{level 7}=[sibling distance=0.07in]
\node [label=$c_0$] {} coordinate (t9)
child{ coordinate (t0) edge from parent[color=black,thick]
child{ coordinate (t00) edge from parent[color=red,thick]
child{coordinate (t000) edge from parent[color=black,thick]
child{coordinate (t0000) edge from parent[color=purple,thick]
child{coordinate (t00000) edge from parent[color=magenta,thick]
child{coordinate (t000000) edge from parent[color=black,thick]}
}
}
}
child{coordinate (t001) edge from parent[color=purple,thick]
child{coordinate (t0010) edge from parent[color=purple,thick]
child{coordinate (t00100) edge from parent[color=magenta,thick]
child{coordinate (t001000) edge from parent[color=black,thick]}
}
child{coordinate (t00101) edge from parent[color=magenta,thick]
child{coordinate (t001010) edge from parent[color=black,thick]}
}
}
}
child{coordinate (t002) edge from parent[color=olive,thick]
child{coordinate (t0020) edge from parent[color=purple,thick]
child{coordinate (t00200)
edge from parent[color=magenta,thick]
child{coordinate (t002000)
edge from parent[color=black,thick]}
}
}
}
child{coordinate (t003) edge from parent[color=teal,thick]
child{coordinate (t0030) edge from parent[color=purple,thick]
child{coordinate (t00300) edge from parent[color=magenta,thick]
child{coordinate (t003000) edge from parent[color=black,thick]}
child{coordinate (t003001) edge from parent[color=black,thick]}
child{coordinate (t003002) edge from parent[color=black,thick]}
child{coordinate (t003003) edge from parent[color=black,thick]}
}
}
}
}
}
child{ coordinate (t1) edge from parent[color=black,thick]
child{ coordinate (t10) edge from parent[color=cyan,thick]
child{coordinate (t100) edge from parent[color=black,thick]
child{coordinate (t1000) edge from parent[color=purple,thick]
child{coordinate (t10000) edge from parent[color=magenta,thick]
child{coordinate (t100000) edge from parent[color=black,thick]}
}
}
}
}
}
child{ coordinate (t2) edge from parent[color=black,thick]
child{ coordinate (t20) edge from parent[color=purple,thick]
child{coordinate (t200) edge from parent[color=red,thick]
child{coordinate (t2000) edge from parent[color=purple,thick]
child{coordinate (t20000) edge from parent[color=magenta,thick]
child{coordinate (t200000) edge from parent[color=black,thick]}
}
}
}
}
child{coordinate (t21) edge from parent[color=teal,thick]
child{coordinate (t210) edge from parent[color=olive,thick]
child{coordinate (t2100) edge from parent[color=magenta,thick]
child{coordinate (t21000) edge from parent[color=magenta,thick]
child{coordinate (t210000) edge from parent[color=black,thick]}
}
}
child{coordinate (t2101) edge from parent[color=purple,thick]
child{coordinate (t21010) edge from parent[color=magenta,thick]
child{coordinate (t210100) edge from parent[color=black,thick]}
}
}
}
}
}
child{ coordinate (t3) edge from parent[color=black,thick]
child{ coordinate (t30) edge from parent[color=orange,thick]
child{coordinate (t300) edge from parent[color=red,thick]
child{coordinate (t3000) edge from parent[color=blue,thick]
child{coordinate (t30000) edge from parent[color=magenta,thick]
child{coordinate (t300000) edge from parent[color=black,thick]}
}
}
}
}
}
;

\node[circle, fill=blue,inner sep=0pt, minimum size=5pt] at (t9) {};
\node[label=280:$\scriptstyle{\{x<v_0,x \cancel{E} v_0\}}$] at (t0) {};
\node[circle, fill=blue,inner sep=0pt, minimum size=5pt,label=0:$c_1$] at (t2) {};
\node[label=250:$\scriptstyle{\{v_0<x,x \cancel{E} v_0\}}$] at (t3) {};
\node[circle, fill=blue,inner sep=0pt, minimum size=5pt,label=0:$c_2$] at (t00) {};
\node[circle, fill=blue,inner sep=0pt, minimum size=5pt,label=0:$c_3$] at (t210) {};
\node[circle, fill=blue,inner sep=0pt, minimum size=5pt,label=0:$c_4$] at (t0010) {};
\node[circle, fill=blue,inner sep=0pt, minimum size=5pt,label=0:$c_5$] at (t00300) {};
\node[label=0:$\hspace{-1mm}\scriptstyle{\{x<v_0, x E v_0\}}$] at (t1) {};
\node[label=180:$\hspace{1mm}\scriptstyle{\{v_0<x, x E v_0\}}$] at (t2) {};
\node[label=200:${\color{red}\scriptstyle{\{x<v_1, x \cancel{E} v_1\}}}$] at (t00) {};
\node[label=200:${\color{cyan}\scriptstyle{\{x<v_1, x E v_1\}}}$] at (t10) {};
\node[label=250:${\color{purple}\scriptstyle{\{x<v_1, x E v_1\}}}\hspace{-5mm}$] at (t20) {};
\node[label=280:${\color{teal}\scriptstyle{\{v_1<x, x E v_1\}}}$] at (t21) {};
\node[label=280:${\color{orange}\scriptstyle{\{v_1<x, x \cancel{E} v_1\}}}$] at (t30) {};
\node[label=270:$\scriptstyle{\{x<v_2, x \cancel{E} v_2\}}$] at (t000) {};
\node[label=270:${\color{teal}\scriptscriptstyle{\{v_2<x, x \cancel{E} v_2\}}} \hspace{3mm}$] at (t003) {};
\node[label=250:$\scriptscriptstyle{\{v_2<x, x \cancel{E} v_2\}} \hspace{-2mm}$] at (t100) {};
\node[label=250:${\color{red}\scriptstyle{\{v_2<x, x \cancel{E} v_2\}}}$] at (t200) {};
\node[label=280:${\color{olive}\scriptstyle{\{v_2<x, x \cancel{E} v_2\}}}$] at (t210) {};
\node[label=280:${\color{red}\scriptstyle{\{v_2<x, x \cancel{E} v_2\}}}$] at (t300) {};
\node[label=280:${\color{blue}\scriptstyle{\{v_3<x, x \cancel{E} v_3\}}}$] at (t3000) {};
\node[label=280:${\color{purple}\scriptstyle{\{v_3<x, x E v_3\}}}$] at (t2101) {};
\node[label=250:${\color{red}\scriptstyle{\{x<v_3, x E v_3\}}}$] at (t2100) {};
\node[label=280:${\color{magenta}\scriptstyle{\{v_4<x, x \cancel{E} v_4\}}}$] at (t00300) {};

\node[circle,inner sep=0pt, minimum size=5pt,label=90:$\vdots$] at (t000000) {};
\node[circle,inner sep=0pt, minimum size=5pt,label=90:$\vdots$] at (t001000) {};
\node[circle,inner sep=0pt, minimum size=5pt,label=90:$\vdots$] at (t001010) {};
\node[circle,inner sep=0pt, minimum size=5pt,label=90:$\vdots$] at (t002000) {};
\node[circle,inner sep=0pt, minimum size=5pt,label=90:$\vdots$] at (t003000) {};
\node[circle,inner sep=0pt, minimum size=5pt,label=90:$\vdots$] at (t003001) {};
\node[circle,inner sep=0pt, minimum size=5pt,label=90:$\vdots$] at (t003002) {};
\node[circle,inner sep=0pt, minimum size=5pt,label=90:$\vdots$] at (t003003) {};
\node[circle,inner sep=0pt, minimum size=5pt,label=90:$\vdots$] at (t100000) {};
\node[circle,inner sep=0pt, minimum size=5pt,label=90:$\vdots$] at (t200000) {};
\node[circle,inner sep=0pt, minimum size=5pt,label=90:$\vdots$] at (t210000) {};
\node[circle,inner sep=0pt, minimum size=5pt,label=90:$\vdots$] at (t210100) {};
\node[circle,inner sep=0pt, minimum size=5pt,label=90:$\vdots$] at (t300000) {};

\node[circle, fill=blue,inner sep=0pt, minimum size=5pt, label=$v_4$,below=3.8cm of t001] (v4) {};
\node[circle, fill=blue,inner sep=0pt, minimum size=5pt, label=$v_2$,below=2.9cm of t00] (v2) {};
\node[circle, fill=blue,inner sep=0pt, minimum size=5pt, label=$v_5$,below=3.799cm of t003] (v5) {};
\node[circle, fill=blue,inner sep=0pt, minimum size=5pt, label=$v_0$,below=1.0cm of t9] (v0) {};
\node[circle, fill=blue,inner sep=0pt, minimum size=5pt, label=$v_1$,below=1.92cm of t2] (v1) {};
\node[circle, fill=blue,inner sep=0pt, minimum size=5pt, label=$v_3$,below=2.83cm of t21] (v3) {};

\node[circle, fill=white,draw,inner sep=0pt, minimum size=4pt] at (t0) {};
\node[circle, fill=white,draw,inner sep=0pt, minimum size=4pt] at (t000) {};
\node[circle, fill=white,draw,inner sep=0pt, minimum size=4pt] at (t0000) {};
\node[circle, fill=white,draw,inner sep=0pt, minimum size=4pt] at (t00000) {};
\node[circle, fill=white,draw,inner sep=0pt, minimum size=4pt] at (t000000) {};
\node[circle, fill=white,draw,inner sep=0pt, minimum size=4pt] at (t001) {};
\node[circle, fill=white,draw,inner sep=0pt, minimum size=4pt] at (t00100) {};
\node[circle, fill=white,draw,inner sep=0pt, minimum size=4pt] at (t001000) {};
\node[circle, fill=white,draw,inner sep=0pt, minimum size=4pt] at (t00101) {};
\node[circle, fill=white,draw,inner sep=0pt, minimum size=4pt] at (t001010) {};
\node[circle, fill=white,draw,inner sep=0pt, minimum size=4pt] at (t002) {};
\node[circle, fill=white,draw,inner sep=0pt, minimum size=4pt] at (t0020) {};
\node[circle, fill=white,draw,inner sep=0pt, minimum size=4pt] at (t00200) {};
\node[circle, fill=white,draw,inner sep=0pt, minimum size=4pt] at (t002000) {};
\node[circle, fill=white,draw,inner sep=0pt, minimum size=4pt] at (t003) {};
\node[circle, fill=white,draw,inner sep=0pt, minimum size=4pt] at (t0030) {};
\node[circle, fill=white,draw,inner sep=0pt, minimum size=4pt] at (t003000) {};
\node[circle, fill=white,draw,inner sep=0pt, minimum size=4pt] at (t003001) {};
\node[circle, fill=white,draw,inner sep=0pt, minimum size=4pt] at (t003002) {};
\node[circle, fill=white,draw,inner sep=0pt, minimum size=4pt] at (t003003) {};
\node[circle, fill=white,draw,inner sep=0pt, minimum size=4pt] at (t1) {};
\node[circle, fill=white,draw,inner sep=0pt, minimum size=4pt] at (t10) {};
\node[circle, fill=white,draw,inner sep=0pt, minimum size=4pt] at (t100) {};
\node[circle, fill=white,draw,inner sep=0pt, minimum size=4pt] at (t1000) {};
\node[circle, fill=white,draw,inner sep=0pt, minimum size=4pt] at (t10000) {};
\node[circle, fill=white,draw,inner sep=0pt, minimum size=4pt] at (t100000) {};
\node[circle, fill=white,draw,inner sep=0pt, minimum size=4pt] at (t20) {};
\node[circle, fill=white,draw,inner sep=0pt, minimum size=4pt] at (t200) {};
\node[circle, fill=white,draw,inner sep=0pt, minimum size=4pt] at (t2000) {};
\node[circle, fill=white,draw,inner sep=0pt, minimum size=4pt] at (t20000) {};
\node[circle, fill=white,draw,inner sep=0pt, minimum size=4pt] at (t200000) {};
\node[circle, fill=white,draw,inner sep=0pt, minimum size=4pt] at (t21) {};
\node[circle, fill=white,draw,inner sep=0pt, minimum size=4pt] at (t2100) {};
\node[circle, fill=white,draw,inner sep=0pt, minimum size=4pt] at (t21000) {};
\node[circle, fill=white,draw,inner sep=0pt, minimum size=4pt] at (t210000) {};
\node[circle, fill=white,draw,inner sep=0pt, minimum size=4pt] at (t2101) {};
\node[circle, fill=white,draw,inner sep=0pt, minimum size=4pt] at (t21010) {};
\node[circle, fill=white,draw,inner sep=0pt, minimum size=4pt] at (t210100) {};
\node[circle, fill=white,draw,inner sep=0pt, minimum size=4pt] at (t3) {};
\node[circle, fill=white,draw,inner sep=0pt, minimum size=4pt] at (t30) {};
\node[circle, fill=white,draw,inner sep=0pt, minimum size=4pt] at (t300) {};
\node[circle, fill=white,draw,inner sep=0pt, minimum size=4pt] at (t3000) {};
\node[circle, fill=white,draw,inner sep=0pt, minimum size=4pt] at (t30000) {};
\node[circle, fill=white,draw,inner sep=0pt, minimum size=4pt] at (t300000) {};

\draw[magenta,thick] (v4)--(v2);
\draw[magenta,thick] (v0)--(v1);
\draw[magenta,thick] (v1)--(v3);
\end{tikzpicture}}
\caption{Coding tree of $1$-types for $\bQ_{\bQ}$ and linear order with convex equivalence relations represented by its coding nodes.}
\label{fig.QQtree}
\end{sidewaysfigure}


Next, we present graphics for coding trees of $1$-types for some free amalgamation classes.
The tree of $1$-types for the Rado graph is simply a binary tree in which the coding nodes are dense and every node $s$  at the level of the $n$-th coding node splits into two immediate successors, representing the two possible extensions of $s$ to
the $1$-types
$s\cup\{E(x,v_n)\}$ and $s\cup\{\neg E(x,v_n)\}$.
This follows immediately from the Extension Property for the Rado graph.
As this is simple to visualize, and as a graphic has already appeared in \cite{DobrinenIfCoLog20}, we move on to bipartite graphs.

\begin{example}[The coding  tree of $1$-types
for the generic bipartite graph]\label{ex.BPG}
Figure \ref{fig.genbiptree}.\ presents a  coding tree of $1$-types for the  generic
bipartite graph.
The unary relations $U_0$ and $U_1$, which keep track of which partition each vertex is in, are represented by ``red'' and ``blue'', respectively.
We have chosen to enumerate  this  structure so that
 odd indexed vertices are in one of the partitions, and even indexed vertices  are in the other,
for purely aesthetic reasons.
The edge relation is represented as extension to the right, and non-edge is represented by extending left.
On the left is the bipartite graph being represented by the coding nodes in the two-rooted tree of $1$-types.
For instance, $c_0$ is the $1$-type $\{U_0(x)\}$, so $v_0$ is a vertex in the  collection of ``red''  vertices.
For another example,
$c_3$ is the $1$-type $\{U_1(x), E(x,v_0), \neg E(x,v_1),E(x,v_2)\}$.
Thus, $v_3$ is in the collection of ``blue''  vertices and has edges exactly with $v_0$ and $v_2$.
It is straightforward to check that the classes of $n$-partite graphs satisfy SFAP.
\end{example}


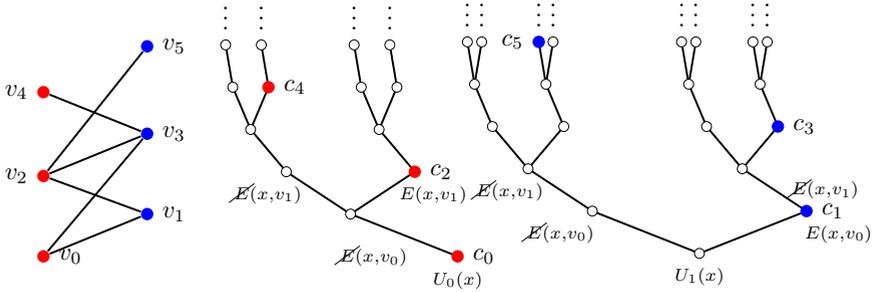
\begin{figure}
\begin{center}
\resizebox{4.7cm}{!}{
\begin{minipage}{.2\textwidth}
\hspace*{-3.5cm}
\begin{tikzpicture}[grow'=up,scale=.4]
\tikzstyle{level 1}=[sibling distance=3in]
\tikzstyle{level 2}=[sibling distance=1.8in]
\tikzstyle{level 3}=[sibling distance=1in]
\tikzstyle{level 4}=[sibling distance=0.5in]
\tikzstyle{level 5}=[sibling distance=0.2in]
\tikzstyle{level 6}=[sibling distance=0.1in]
\tikzstyle{level 7}=[sibling distance=0.07in]

\node {} coordinate (t9)
child{ coordinate (t0) edge from parent[color=black,thick]
child{coordinate (t00) edge from parent[color=black,thick]
child{coordinate (t000) edge from parent[color=black,thick]
child{coordinate (t0000) edge from parent[color=black,thick]
child{coordinate (t00000) edge from parent[color=black,thick]}
child{coordinate (t00001) edge from parent[color=white,thick]}
}
child{coordinate (t0001) edge from parent[color=black,thick]
child{coordinate (t00010) edge from parent[color=black,thick]}
child{coordinate (t00011) edge from parent[color=white,thick]}
}
}
child{coordinate (t001) edge from parent[color=white,thick]}
}
child{coordinate (t01) edge from parent[color=black,thick]
child{coordinate (t010) edge from parent[color=black,thick]
child{coordinate (t0100) edge from parent[color=black,thick]
child{coordinate (t01000) edge from parent[color=black,thick]}
child{coordinate (t01001) edge from parent[color=white,thick]}
}
child{coordinate (t0101) edge from parent[color=black,thick]
child{coordinate (t01010) edge from parent[color=black,thick]}
child{coordinate (t01011) edge from parent[color=white,thick]}
}
}
child{coordinate (t011) edge from parent[color=white,thick]}
}
}
child{coordinate (t1) edge from parent[color=white,thick]}
;

\node[label=180:$\scriptstyle{\cancel{E}(x,v_0)}\hspace{0.5cm}$,label=270:$\scriptstyle{U_0(x)}$,label=0:$c_0$,circle, fill=red,inner sep=0pt, minimum size=5pt] at (t9) {};
\node[label=0:$c_2$,circle, fill=red,inner sep=0pt, minimum size=5pt] at (t01) {};
\node[label=0:$c_4$,circle, fill=red,inner sep=0pt, minimum size=5pt] at (t0001) {};

\node[circle,inner sep=0pt, minimum size=5pt,label=90:$\vdots$] at (t00000) {};
\node[circle,inner sep=0pt, minimum size=5pt,label=90:$\vdots$] at (t00010) {};
\node[circle,inner sep=0pt, minimum size=5pt,label=90:$\vdots$] at (t01000) {};
\node[circle,inner sep=0pt, minimum size=5pt,label=90:$\vdots$] at (t01010) {};

\node[label=20:$\hspace{0.5cm}\scriptstyle{E(x,v_1)}$,label=160:$\scriptstyle{\cancel{E}(x,v_1)}\hspace{0.5cm}$,circle, fill=white,draw,inner sep=0pt, minimum size=4pt] at (t0) {};
\node[circle, fill=white,draw,inner sep=0pt, minimum size=4pt] at (t00) {};
\node[circle, fill=white,draw,inner sep=0pt, minimum size=4pt] at (t000) {};
\node[circle, fill=white,draw,inner sep=0pt, minimum size=4pt] at (t0000) {};
\node[circle, fill=white,draw,inner sep=0pt, minimum size=4pt] at (t00000) {};
\node[circle, fill=white,draw,inner sep=0pt, minimum size=4pt] at (t00010) {};
\node[circle, fill=white,draw,inner sep=0pt, minimum size=4pt] at (t010) {};
\node[circle, fill=white,draw,inner sep=0pt, minimum size=4pt] at (t0100) {};
\node[circle, fill=white,draw,inner sep=0pt, minimum size=4pt] at (t01000) {};
\node[circle, fill=white,draw,inner sep=0pt, minimum size=4pt] at (t0101) {};
\node[circle, fill=white,draw,inner sep=0pt, minimum size=4pt] at (t01010) {};

\node[label=0:$v_0$,circle, fill=red,inner sep=0pt, minimum size=5pt,left=5.8cm of t9] (v0) {};
\node[label=0:$v_1$,circle, fill=blue,inner sep=0pt, minimum size=5pt,left=2.8cm of t0] (v1) {};
\node[label=180:$v_2$,circle, fill=red,inner sep=0pt, minimum size=5pt,above=0.95cm of v0] (v2) {};
\node[label=0:$v_3$,circle, fill=blue,inner sep=0pt, minimum size=5pt,above=0.95cm of v1] (v3) {};
\node[label=180:$v_4$,circle, fill=red,inner sep=0pt, minimum size=5pt,above=1.0cm of v2] (v4) {};
\node[label=0:$v_5$,circle, fill=blue,inner sep=0pt, minimum size=5pt,above=1.05cm of v3] (v5) {};

\draw[black,thick] (v0) to (v1);
\draw[black,thick] (v0) to (v3);
\draw[black,thick] (v1) to (v2);
\draw[black,thick] (v2) to (v3);
\draw[black,thick] (v2) to (v5);
\draw[black,thick] (v3) to (v4);
\end{tikzpicture}
\end{minipage}
\begin{minipage}{.2\textwidth}
\hspace*{0.5cm}
\vspace*{0.1cm}
\begin{tikzpicture}[grow'=up,scale=.4]
\tikzstyle{level 1}=[sibling distance=3in]
\tikzstyle{level 2}=[sibling distance=1.8in]
\tikzstyle{level 3}=[sibling distance=1in]
\tikzstyle{level 4}=[sibling distance=0.5in]
\tikzstyle{level 5}=[sibling distance=0.2in]
\tikzstyle{level 6}=[sibling distance=0.1in]
\tikzstyle{level 7}=[sibling distance=0.07in]

\node {} coordinate (u9)
child{ coordinate (u0) edge from parent[color=black,thick]
child{coordinate (u00) edge from parent[color=black,thick]
child{coordinate (u000) edge from parent[color=black,thick]
child{coordinate (u0000) edge from parent[color=black,thick]
child{coordinate (u00000) edge from parent[color=black,thick]}
child{coordinate (u00001) edge from parent[color=black,thick]}
}
child{coordinate (u0001) edge from parent[color=white,thick]}
}
child{coordinate (u001) edge from parent[color=black,thick]
child{coordinate (u0010) edge from parent[color=black,thick]
child{coordinate (u00100) edge from parent[color=black,thick]}
child{coordinate (u00101) edge from parent[color=black,thick]}
}
child{coordinate (u0011) edge from parent[color=white,thick]}
}
}
child{coordinate (u01) edge from parent[color=white,thick]}
}
child{coordinate (u1) edge from parent[color=black,thick]
child{coordinate (u10) edge from parent[color=black,thick]
child{coordinate (u100) edge from parent[color=black,thick]
child{coordinate (u1000) edge from parent[color=black,thick]
child{coordinate (u10000) edge from parent[color=black,thick]}
child{coordinate (u10001) edge from parent[color=black,thick]}
}
child{coordinate (u1001) edge from parent[color=white,thick]}
}
child{coordinate (u101) edge from parent[color=black,thick]
child{coordinate (u1010) edge from parent[color=black,thick]
child{coordinate (u10100) edge from parent[color=black,thick]}
child{coordinate (u10101) edge from parent[color=black,thick]}
}
child{coordinate (u1011) edge from parent[color=white,thick]}
}
}
child{coordinate (u11) edge from parent[color=white,thick]}
}
;

\node[label=0:$c_1$,circle, fill=blue,inner sep=0pt, minimum size=5pt] at (u1) {};
\node[label=0:$c_3$,circle, fill=blue,inner sep=0pt, minimum size=5pt] at (u101) {};
\node[label=180:$c_5$,circle, fill=blue,inner sep=0pt, minimum size=5pt] at (u00100) {};

\node[circle,inner sep=0pt, minimum size=5pt,label=90:$\vdots$] at (u00000) {};
\node[circle,inner sep=0pt, minimum size=5pt,label=90:$\vdots$] at (u00001) {};
\node[circle,inner sep=0pt, minimum size=5pt,label=90:$\vdots$] at (u00100) {};
\node[circle,inner sep=0pt, minimum size=5pt,label=90:$\vdots$] at (u00101) {};
\node[circle,inner sep=0pt, minimum size=5pt,label=90:$\vdots$] at (u10000) {};
\node[circle,inner sep=0pt, minimum size=5pt,label=90:$\vdots$] at (u10001) {};
\node[circle,inner sep=0pt, minimum size=5pt,label=90:$\vdots$] at (u10100) {};
\node[circle,inner sep=0pt, minimum size=5pt,label=90:$\vdots$] at (u10101) {};

\node[label=270:$\scriptstyle{U_1(x)}$,label=20:$\hspace{1.3cm}\scriptstyle{E(x,v_0)}$,label=160:$\scriptstyle{\cancel{E}(x,v_0)}\hspace{1.3cm}$,circle, fill=white,draw,inner sep=0pt, minimum size=4pt] at (u9) {};
\node[label=160:$\scriptstyle{\cancel{E}(x,v_1)}\hspace{0.5cm}$,circle, fill=white,draw,inner sep=0pt, minimum size=4pt] at (u0) {};
\node[circle, fill=white,draw,inner sep=0pt, minimum size=4pt] at (u00) {};
\node[circle, fill=white,draw,inner sep=0pt, minimum size=4pt] at (u000) {};
\node[circle, fill=white,draw,inner sep=0pt, minimum size=4pt] at (u001) {};
\node[circle, fill=white,draw,inner sep=0pt, minimum size=4pt] at (u0000) {};
\node[circle, fill=white,draw,inner sep=0pt, minimum size=4pt] at (u00000) {};
\node[circle, fill=white,draw,inner sep=0pt, minimum size=4pt] at (u0010) {};
\node[circle, fill=white,draw,inner sep=0pt, minimum size=4pt] at (u00001) {};
\node[circle, fill=white,draw,inner sep=0pt, minimum size=4pt] at (u00101) {};
\node[label=-20:$\hspace{5mm}\scriptstyle{\cancel{E}(x,v_1)}$,circle, fill=white,draw,inner sep=0pt, minimum size=4pt] at (u10) {};
\node[circle, fill=white,draw,inner sep=0pt, minimum size=4pt] at (u100) {};
\node[circle, fill=white,draw,inner sep=0pt, minimum size=4pt] at (u1000) {};
\node[circle, fill=white,draw,inner sep=0pt, minimum size=4pt] at (u10000) {};
\node[circle, fill=white,draw,inner sep=0pt, minimum size=4pt] at (u10001) {};
\node[circle, fill=white,draw,inner sep=0pt, minimum size=4pt] at (u1010) {};
\node[circle, fill=white,draw,inner sep=0pt, minimum size=4pt] at (u10100) {};
\node[circle, fill=white,draw,inner sep=0pt, minimum size=4pt] at (u10101) {};
\end{tikzpicture}
\end{minipage}}
\end{center}
\caption{Coding tree of $1$-types for the generic bipartite graph\hspace*{2cm}}\label{fig.genbiptree}
\end{figure}


Lastly, we consider free amalgamation classes with relations of higher arity.
The prototypical example of this is the generic $3$-uniform hypergraph, and discussing it should provide the reader with reasonable intuition about coding trees for higher arities.

\begin{example}[The coding tree of $1$-types for the generic $3$-uniform hypergraph]\label{ex.3reghypergraph}
Figure \ref{fig.3hyptree}.\ presents
the coding tree of $1$-types for the generic $3$-uniform hypergraph.
This tree has the property that every node at the same level branches into the same number of immediate successors, as there are no forbidden substructures.
On the left of Figure \ref{fig.3hyptree}.\ is a picture of the hypergraph being built, where $v_n$ is the vertex satisfying the $1$-type of the coding node $c_n$ over the initial segment of the structure restricted to  $\{v_i:i<n\}$.

Since hyperedges involve three vertices, $c_0$ and $c_1$ are both the empty $1$-types.
Technically, these nodes are the same, but we draw them distinctly in Figure 4.\ to aid the drawing of the hypergraph on the left.
Letting $R$ denote the $3$-hyperedge relation,
$c_1$ branches into  two  $1$-types  over $\{v_0,v_1\}$:
$\{\neg R(x,v_0,v_1)\}$ and $\{R(x,v_0,v_1)\}$.
Since $c_2=\{R(x,v_0,v_1)\}$, it follows that $R(v_0,v_1,v_2)$ holds in the hypergraph represented on the left of the tree; this hyperedge is represented by the oval containing these three vertices.

Both nodes on the level of $c_2$ branch into four immediate successors.
This is because
for each node $s$ at the level of $c_2$, the immediate successors of $s$ range over the possibilities of adding a new formula $R(x,\cdot,\cdot)$ or $\neg R(x,\cdot,\cdot)$ containing  the parameter $v_2$
and a choice of either $v_0$ or $v_1$ as the second parameter.
In particular,  the immediate successors of $c_2$ are the $1$-types consisting of  $\{R(x,v_0,v_1)\}$  unioned with  one of the following:
\begin{enumerate}
\item
$\{\neg R(x,v_0,v_2),\neg R(x,v_1,v_2)\}$;
\item
$\{\neg R(x,v_0,v_2),R(x,v_1,v_2)\}$;
\item
$\{ R(x,v_0,v_2),\neg R(x,v_1,v_2)\}$;
\item
$\{ R(x,v_0,v_2),R(x,v_1,v_2)\}$.
\end{enumerate}
Likewise, the immediate successors of the other node $s=\{\neg R(x,v_0,v_1)\}$ in level two of the tree consists of the extensions of $s$ by one of the four above cases.
In general, each node on the level of $c_n$ branches into $2^n$ many immediate successors.
This is because the new formulas in any immediate successor have the choice of $R(x,p,v_n)$ or its negation, where $p\in\{v_i:i<n\}$.
However,
the \Fraisse\ class of  finite $3$-uniform hypergraphs  satisfies \SFAP\ (by
 Proposition 5.2 of Part II), and
  Theorem \ref{thm.SFAPimpliesDCT} will  provide  a  skew subtree coding the generic $3$-hypergraph in which the branching degree is two (that is, a diagonal subtree).

The coding node $c_3$ is the $1$-type $\{\neg R(x,v_0,v_1),  R(x,v_0,v_2),\neg R(x,v_1,v_2)\}$.
Thus, the hypergraph being built on the left has the hyperedge $R(v_0,v_2,v_3)$.
The coding node
$c_4$ is the $1$-type  consisting of $R(x,v_0,v_1),
R(x,v_1,v_2),
R(x,v_2,v_3)$ along with
$\neg R(x,p_0,p_1)$ where $p_0,p_1$ are parameters in $\{v_0,\dots,v_3\}$.
This codes the new hyperedges
$R(v_0,v_1,v_4)$,  $R(v_1,v_2,v_4)$ and $R(v_2,v_3,v_4)$.
\end{example}


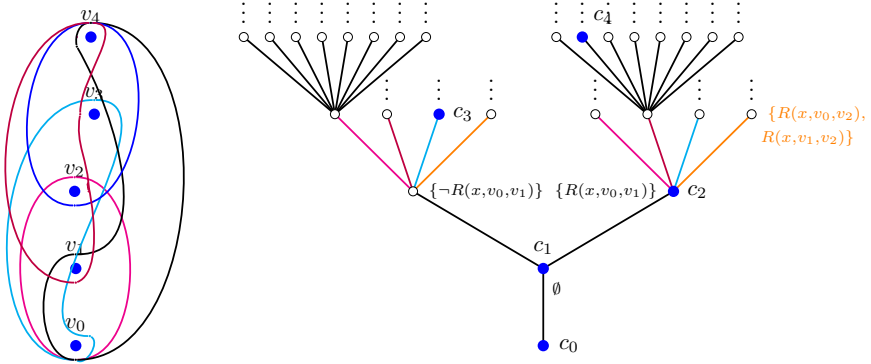
\begin{figure}
\resizebox{12cm}{!}{
\begin{tikzpicture}[grow'=up,scale=.8]
\tikzstyle{level 1}=[sibling distance=3.3in]
\tikzstyle{level 2}=[sibling distance=2in]
\tikzstyle{level 3}=[sibling distance=.4in]
\tikzstyle{level 4}=[sibling distance=0.2in]
\tikzstyle{level 5}=[sibling distance=0.125in]
\node [label=0:$c_0$] {} coordinate (t9)
child{ coordinate (t0) edge from parent[color=black,thick]
child{ coordinate (t00) edge from parent[color=black,thick]
child{ coordinate (t000) edge from parent[color=magenta,thick]
child{coordinate (t0000) edge from parent[color=black,thick]}
child{coordinate (t0001) edge from parent[color=black,thick]}
child{coordinate (t0002) edge from parent[color=black,thick]}
child{coordinate (t0003) edge from parent[color=black,thick]}
child{coordinate (t0004) edge from parent[color=black,thick]}
child{coordinate (t0005) edge from parent[color=black,thick]}
child{coordinate (t0006) edge from parent[color=black,thick]}
child{coordinate (t0007) edge from parent[color=black,thick]}
}
child{coordinate (t001) edge from parent[color=purple,thick]}
child{coordinate (t002) edge from parent[color=cyan,thick]}
child{coordinate (t003) edge from parent[color=orange,thick]}
}
child{ coordinate (t01) edge from parent[color=black,thick]
child{ coordinate (t010) edge from parent[color=magenta,thick]}
child{ coordinate (t011) edge from parent[color=purple,thick]
child{ coordinate (t0110) edge from parent[color=black,thick]}
child{ coordinate (t0111) edge from parent[color=black,thick]}
child{ coordinate (t0112) edge from parent[color=black,thick]}
child{ coordinate (t0113) edge from parent[color=black,thick]}
child{ coordinate (t0114) edge from parent[color=black,thick]}
child{ coordinate (t0115) edge from parent[color=black,thick]}
child{ coordinate (t0116) edge from parent[color=black,thick]}
child{ coordinate (t0117) edge from parent[color=black,thick]}
}
child{ coordinate (t012) edge from parent[color=cyan,thick]}
child{ coordinate (t013) edge from parent[color=orange,thick]}
}
}
;

\node[circle, fill=blue,inner sep=0pt, minimum size=5pt] at (t9) {};
\node[circle, fill=blue,inner sep=0pt, minimum size=5pt,label=90:$c_1$,label=280:$\scriptstyle{\emptyset}$] at (t0) {};
\node[circle, fill=blue,inner sep=0pt, minimum size=5pt,label=0:$c_2$] at (t01) {};
\node[circle, fill=blue,inner sep=0pt, minimum size=5pt,label=0:$c_3$] at (t002) {};
\node[circle, fill=blue,inner sep=0pt, minimum size=5pt,label=60:$c_4$] at (t0111) {};
\node[label=0:$\scriptstyle{\{\neg R(x,v_0,v_1)\}}$] at (t00) {};
\node[label=180:$\scriptstyle{\{R(x,v_0,v_1)\}}$] at (t01) {};
\node[label=280:${\color{orange}\scriptstyle{R(x,v_1,v_2)\}}}$,label=0:${\color{orange}\scriptstyle{\{R(x,v_0,v_2),}}$] at (t013) {};

\node[circle, fill=white,draw,inner sep=0pt, minimum size=4pt] at (t00) {};
\node[circle, fill=white,draw,inner sep=0pt, minimum size=4pt] at (t000) {};
\node[circle, fill=white,draw,inner sep=0pt, minimum size=4pt] at (t001) {};
\node[circle, fill=white,draw,inner sep=0pt, minimum size=4pt] at (t003) {};
\node[circle, fill=white,draw,inner sep=0pt, minimum size=4pt] at (t0000) {};
\node[circle, fill=white,draw,inner sep=0pt, minimum size=4pt] at (t0001) {};
\node[circle, fill=white,draw,inner sep=0pt, minimum size=4pt] at (t0002) {};
\node[circle, fill=white,draw,inner sep=0pt, minimum size=4pt] at (t0003) {};
\node[circle, fill=white,draw,inner sep=0pt, minimum size=4pt] at (t0004) {};
\node[circle, fill=white,draw,inner sep=0pt, minimum size=4pt] at (t0005) {};
\node[circle, fill=white,draw,inner sep=0pt, minimum size=4pt] at (t0006) {};
\node[circle, fill=white,draw,inner sep=0pt, minimum size=4pt] at (t0007) {};
\node[circle, fill=white,draw,inner sep=0pt, minimum size=4pt] at (t010) {};
\node[circle, fill=white,draw,inner sep=0pt, minimum size=4pt] at (t011) {};
\node[circle, fill=white,draw,inner sep=0pt, minimum size=4pt] at (t012) {};
\node[circle, fill=white,draw,inner sep=0pt, minimum size=4pt] at (t013) {};
\node[circle, fill=white,draw,inner sep=0pt, minimum size=4pt] at (t0110) {};
\node[circle, fill=white,draw,inner sep=0pt, minimum size=4pt] at (t0112) {};
\node[circle, fill=white,draw,inner sep=0pt, minimum size=4pt] at (t0113) {};
\node[circle, fill=white,draw,inner sep=0pt, minimum size=4pt] at (t0114) {};
\node[circle, fill=white,draw,inner sep=0pt, minimum size=4pt] at (t0115) {};
\node[circle, fill=white,draw,inner sep=0pt, minimum size=4pt] at (t0116) {};
\node[circle, fill=white,draw,inner sep=0pt, minimum size=4pt] at (t0117) {};

\node[circle, fill=blue,inner sep=0pt, minimum size=5pt, label=$v_0$,left=7.2cm of t9] (v0) {};
\node[circle, fill=blue,inner sep=0pt, minimum size=5pt, label=$v_1$,left=7.2cm of t0] (v1) {};
\node[circle, fill=blue,inner sep=0pt, minimum size=5pt, label=$v_2$,left=9.25cm of t01] (v2) {};
\node[circle, fill=blue,inner sep=0pt, minimum size=5pt, label=$v_3$,left=3.66cm of t000] (v3) {};
\node[circle, fill=blue,inner sep=0pt, minimum size=5pt, label=$v_4$,left=7.98cm of t0112] (v4) {};

\node[circle,inner sep=0pt, minimum size=5pt,label=90:$\vdots$] at (t0000) {};
\node[circle,inner sep=0pt, minimum size=5pt,label=90:$\vdots$] at (t0001) {};
\node[circle,inner sep=0pt, minimum size=5pt,label=90:$\vdots$] at (t0002) {};
\node[circle,inner sep=0pt, minimum size=5pt,label=90:$\vdots$] at (t0003) {};
\node[circle,inner sep=0pt, minimum size=5pt,label=90:$\vdots$] at (t0004) {};
\node[circle,inner sep=0pt, minimum size=5pt,label=90:$\vdots$] at (t0005) {};
\node[circle,inner sep=0pt, minimum size=5pt,label=90:$\vdots$] at (t0006) {};
\node[circle,inner sep=0pt, minimum size=5pt,label=90:$\vdots$] at (t0007) {};
\node[circle,inner sep=0pt, minimum size=5pt,label=90:$\vdots$] at (t001) {};
\node[circle,inner sep=0pt, minimum size=5pt,label=90:$\vdots$] at (t002) {};
\node[circle,inner sep=0pt, minimum size=5pt,label=90:$\vdots$] at (t003) {};
\node[circle,inner sep=0pt, minimum size=5pt,label=90:$\vdots$] at (t010) {};
\node[circle,inner sep=0pt, minimum size=5pt,label=90:$\vdots$] at (t012) {};
\node[circle,inner sep=0pt, minimum size=5pt,label=90:$\vdots$] at (t013) {};
\node[circle,inner sep=0pt, minimum size=5pt,label=90:$\vdots$] at (t0110) {};
\node[circle,inner sep=0pt, minimum size=5pt,label=90:$\vdots$] at (t0111) {};
\node[circle,inner sep=0pt, minimum size=5pt,label=90:$\vdots$] at (t0112) {};
{};
\node[circle,inner sep=0pt, minimum size=5pt,label=90:$\vdots$] at (t0113) {};
{};
\node[circle,inner sep=0pt, minimum size=5pt,label=90:$\vdots$] at (t0114) {};
{};
\node[circle,inner sep=0pt, minimum size=5pt,label=90:$\vdots$] at (t0115) {};
{};
\node[circle,inner sep=0pt, minimum size=5pt,label=90:$\vdots$] at (t0116) {};
{};
\node[circle,inner sep=0pt, minimum size=5pt,label=90:$\vdots$] at (t0117) {};

\node[circle,fill=black,inner sep=0pt, minimum size=1pt,below=.1cm of v0] (u0) {};
\node[circle,fill=purple,inner sep=0pt, minimum size=1pt,below=.1cm of v1] (u1) {};
\node[circle,fill=magenta,inner sep=0pt, minimum size=1pt,above=.1cm of v2] (a2) {};
\node[circle,fill=black,inner sep=0pt, minimum size=1pt,above=.1cm of v1] (a1) {};
\node[circle,fill=blue,inner sep=0pt, minimum size=1pt,below=.1cm of v2] (u2) {};
\node[circle,fill=purple,inner sep=0pt, minimum size=1pt,right=.1cm of v2] (r2) {};
\node[circle,fill=cyan,inner sep=0pt, minimum size=1pt,above=.1cm of v3] (a3) {};
\node[circle,fill=purple,inner sep=0pt, minimum size=1pt,left=.1cm of v3] (l3) {};
\node[circle,fill=white,inner sep=0pt, minimum size=1pt,left=.1cm of v1] (l1) {};
\node[circle,fill=white,inner sep=0pt, minimum size=1pt,right=.1cm of v0] (r0) {};
\node[circle,fill=cyan,inner sep=0pt, minimum size=1pt,above=.1cm of r0] (ar0) {};
\node[circle,fill=black,inner sep=0pt, minimum size=1pt,above=.1cm of v4] (a4) {};
\node[circle,fill=magenta,inner sep=0pt, minimum size=1pt,left=.1cm of v4] (l4) {};
\node[circle,fill=black,inner sep=0pt, minimum size=1pt,below=.1cm of l4] (ul4) {};

\draw[magenta,thick] (u0) to [out=180,in=180] (a2);
\draw[magenta,thick] (u0) to [in=0,out=0] (a2);

\draw[cyan,thick] (u0) to [in=180,out=180] (a3);
\draw[cyan,thick] (a3) to [in=180,out=0] (ar0);
\draw[cyan,thick] (ar0) to [in=-30,out=-30] (u0);

\draw[black,thick] (u0) to [in=180,out=180] (a1);
\draw[black,thick] (a4) to [in=0,out=0] (u0);
\draw[black,thick] (a1) to [in=300,out=0] (ul4);
\draw[black,thick] (ul4) to [in=180,out=100] (a4);

\draw[blue,thick] (u2) to [out=180,in=180] (a4);
\draw[blue,thick] (a4) to [in=0,out=0] (u2);

\draw[purple,thick] (a4) to [in=180,out=180] (u1);
\draw[purple,thick] (u1) to [in=100,out=20] (r2);
\draw[purple,thick] (r2) to [in=260,out=90] (l3);
\draw[purple,thick] (l3) to [in=0,out=90] (a4);
\end{tikzpicture}}
\caption{Coding tree of $1$-types for the generic $3$-uniform hypergraph.\hspace*{8cm}}\label{fig.3hyptree}
\end{figure}


\subsection{Passing types and  similarity}\label{subsec.3.2}

As before, let
$\bK$ be an enumerated \Fraisse\ structure and $\bS:=\bS(\bK)$ be the corresponding coding  tree of $1$-types.
We begin by defining the notion of a subtree of $\bS$.
As is standard in Ramsey theory on infinite trees (see Chapter 6 of \cite{TodorcevicBK10}),
a subtree is not necessarily closed under initial segments,
but rather it is closed under
those portions of initial segments that have certain prescribed lengths.

\begin{defn}[Subtree]\label{defn.subtree}
Let $T$ be a subset of $\bS$, and let $L$ be the set of lengths of
coding nodes in $T$
and lengths of meets
of two incomparable nodes (not necessarily coding nodes) in $T$. Then $T$
is a {\em subtree} of $\bS$
if $T$ is closed under meets and
closed under initial segments with lengths in
$L$, by which we mean
that  whenever   $\ell\in L$ and $t\in T$ with $\ell\le |t|$, then $t\re \ell$ is also a member of $T$.
\end{defn}

We
now describe the
 natural correspondence from subtrees of $\bS$ to  substructures of $\bK$.
 The following notation will aid in the translation.

 \begin{notation}\label{notn.KreA}
Given a subtree $A\sse\bS$, let
 $\lgl c^A_n:n< N\rgl$ denote the  enumeration of
 the coding nodes of $A$  in order of increasing length, where $N\le \om$ is the number of coding nodes in $A$.
 Let
 \begin{equation}
 \mathrm{N}^A:=\{i\in \om:   \exists m\, (c_i=c^A_m)\},
 \end{equation}
 the set of indices $i$ such that $c_i$ is a coding node in $A$.
 For $n<N$,
 let
\begin{equation}
 \mathrm{N}^A_n:=\{i\in   \mathrm{N}^A:   \exists m<n\, (c_i=c^A_m)\},
\end{equation}
the set of indices  of the first  $n$
coding nodes  in $A$.
  Recall that $\om$ is the set of vertices for $\bK$,
 and that we  often  use   $v_i$ to denote $i$,  the $i$-th vertex of $\bK$.
Thus,  $\mathrm{N}^A$ is precisely the set of vertices of $\bK$
 represented by the coding nodes in $A$.
  Let $\bK\re A$ denote the
 substructure of $\bK$
on  universe  $ \mathrm{N}^A$.
 We call this the
 {\em substructure of $\bK$
 represented by the coding nodes in $A$}, or simply {\em  the substructure represented by $A$}.
\end{notation}

The next definition extends the notion   of {\em passing number}  developed in  \cite{Laflamme/Sauer/Vuksanovic06} and \cite{Sauer06} to
code
binary relations
using
 pairs of nodes
in regular splitting trees.
 Here, we extend this notion to relations of any arity.

Recall from
the discussion after
Definition \ref{defn.treecodeK}
that for $s\in \bS$,
 $s(0)$  denotes the set of formulas in $s$
 without parameters; and for
for  $1\le i<|s|$,
 $s(i)$ denotes  the set of those formulas in $s\re \bK_i$
 in which $v_{i - 1}$ appears.

\begin{defn}[Passing Type]\label{defn.passingtype}
Given $s,t\in \bS$  with $|s|<|t|$,
we call
  $t(|s|)$
the {\em passing type of $t$  at $s$}.
We also call $t(|s|)$
the {\em passing type of $t$ at $c_n$},
where $n+1 = |s|$,
as $|c_n|=n+1$.

Let $A$ be a subtree of $\bS$, $t$ be a node in $\bS$,
and
 $c_n$ be a coding node in $\bS$ such that  $|c_n|<|t|$.
We write $t(c_n;A)$ to denote the set of those formulas in
$t(|c_n|)$ in which all  parameters  are  from among
$\{v_i: i\in  \mathrm{N}^A_m\cup\{n\}\}$,
where $m$ is least such that $|c^A_m|\ge |c_n|$.
We call $t(c_n;A)$ the {\em passing type of $t$ at $c_n$ over $A$}.

Given  a  coding node $c_n^A$ in $A$, we write
 $t(n;A)$ to denote  $t(c^A_n;A)$,
  and call this
the {\em passing type of $t$ at $n$ over $A$}.
\end{defn}

Note that passing types are  partial types which   do  not include any unary relation symbols.
Thus,  one can have realizations of the same passing type by elements which differ on the unary relations.
Further, note that the passing type of $t$ at $s$  only takes into consideration the length of $s$, not $s$ itself.
Writing the ``passing type of $t$ at $s$'' rather than ``passing type of $t$ at $|s|$'' continues the convention set forth in  \cite{Laflamme/Sauer/Vuksanovic06}, \cite{Sauer06}, and continued in  all papers following  on these two.

\begin{rem}\label{rem.pnspecial case}
In the  case where the language $\mathcal{L}$ only has binary relation symbols,
passing type  reduces to  the concept of passing number,  first defined and  used in  \cite{Laflamme/Sauer/Vuksanovic06} and
\cite{Sauer06}  and later used  in \cite{DobrinenJML20}, \cite{DobrinenH_k19}, \cite{DobrinenRado19}, \cite{Zucker20}.
This is  because
for binary relational structures,
the tree $\bS$ has a bounded degree of  branching.
In the special case
of the Rado graph,
where the language has exactly one binary relation, say $E$,  the tree $\bS$ is regular $2$-branching and may be correlated with the tree of
finite
sequences of $0$'s and $1$'s;
then
the passing number  $0$ of $t$ at  $s$  corresponds to the passing type
generated by
$\{\neg R(x,v_{|s|})\}$, and
 the passing number  $1$ of $t$ at  $s$  corresponds to the passing type
 generated by
 $\{R(x,v_{|s|})\}$.

In the case of the rationals,
the  coding tree  of $1$-types $\bS$ for $\bQ$  provides a minimalistic  way to view the work of Devlin in
 \cite{DevlinThesis},
 as $\bS$ branches exactly at coding nodes and nowhere else.
 In our set-up, any antichain of coding nodes is automatically a so-called diagonal antichain,
as defined in Subsection \ref{subsec.3.3}.
 This differs from the  previous approaches to big Ramsey degrees of $\bQ$
 in \cite{LavUnp} and  \cite{DevlinThesis} (see also \cite{TodorcevicBK10}),
  which  use  the binary branching tree, Milliken's theorem, and the method of  envelopes.
\end{rem}

We will need to be able to compare structures represented by  different sets of coding nodes in $\bS$.
The next notion provides a way to do so.

Recall that $x$ is  the variable used in all $1$-types in $\bS$. Given subsets $X$ and $Y$ of $\om$ and map $f: X \to Y$, let
$f^*: X \cup \{x\} \to Y \cup \{x\}$ be the extension of $f$ given by $f^*(x) = x$.

\begin{defn}[Similarity of Passing Types over Subsets]\label{defn.prespt}
Let $A$ and $B$ be  subsets of $\bS$,  and let $m,n\in\om$
 be such that
 $\mathrm{N}^A\cap m$ has the same number of elements
  as $\mathrm{N}^B\cap n$, say $p$.
  Let $f$ be the increasing  bijection  from $\mathrm{N}^A_p$ to $\mathrm{N}^B_p$.
Suppose  $s,t\in\bS$  are such that $|c_m|<|s|$ and $|c_n|<|t|$.
We write
\begin{equation}
s(c_m;A)\sim t(c_n;B)
\end{equation}
when,
given any
relation symbol $R\in \mathcal{L}$ of arity $k$
and  $k$-tuple $(z_0,\dots, z_{k - 1})$, where all $z_i$ are from among $\{v_i:i\in \mathrm{N}^A_p\}\cup\{x\}$ and at least one $z_i$ is the variable $x$,
we have that
  $R(z_0,\dots, z_{k-1})$
   is in $s(c_m;A)$ if and only if
  $R(f^*(z_0),\dots,f^*(z_{k-1}))$ is in $t(c_n;B)$.
   When $s(c_m;A)\sim t(c_n;B)$ holds,
 we say that  the passing type of  $s$ at $c_m$ over $A$ is {\em similar} to the passing type of $t$ at $c_n$ over $B$.

If $A$ and $B$ each have  at least $n+1$ coding nodes,
then
for  $s,t\in\bS$ with
 $|c^A_n|<|s|$  and
 $|c^B_n|<|t|$,
define
\begin{equation}
s(n;A)\sim t(n;B)
\end{equation}
  to mean that $s(c^A_n;A)\sim t(c^B_n;B)$.
 When $s(n;A)\sim t(n;B)$,
 we say that {\em $s$ over $A$ and $t$ over $B$ have similar passing types at the $n$-th coding node},
 or that {\em the passing type of  $s$ at $n$ over $A$ is similar to the passing type of $t$ at $n$ over $B$}.
\end{defn}

It is clear that for fixed $n$, $\sim$ is an equivalence relation
on passing types over subsets of $\bS$.

The following fact is the essence of why we are interested in similarity of passing types:
They tell us exactly when two structures represented by coding nodes are isomorphic
 as substructures of the enumerated structure $\bK$; that is, when there exists an
 $\mathcal{L}$-isomorphism between the structures that preserves the order relation on their underlying sets
 inherited from $\om$.

\begin{fact}\label{fact.simsamestructure}
Let $A$ and $B$ be subsets of $\bS$  and $n<\om$ such that  $A$ and $B$ each have  $n+1$ many coding nodes.
Then
the substructures $\bK\re A$ and $\bK\re B$ are isomorphic, as ordered substructures of $\bK$,
if and only if
\begin{enumerate}
\item
For each $i\le n$, the $1$-types
$c^A_i$   and $c^B_i$ contain
the same parameter-free formulas;
and
\item
For all $i<j\le n$, $c^A_{j}(i;A)\sim c^B_j(i;B)$.
\end{enumerate}
\end{fact}

We now extend the similarity relation on passing types over subsets of $\bS$ to a relation on subtrees of $\bS$ that preserves
tree structure. For this, we first define a
(strict)
linear order $\prec$ on $\bS$:
We may assume there is a linear ordering on the relation symbols
and negated relation symbols in $\mathcal{L}$,
with the convention that
all the
negated relation symbols  appear in the linear order before the  relation symbols.
(We make this convention to support the intuition that ``moving left'' from a node in a tree indicates that a relation does not hold, while ``moving right'' suggests that it does; the convention is not necessary for our results.)
 Extend the usual linear order $<$ on $\om$, the underlying set of $\bK$,
to the set $\{x\} \cup \om$ by setting $x<n$ for each $n \in \om$.
Let $(\{x\} \cup \om)^{<\om}$,  the set of finite sequences from $\{x\} \cup \om$, have the induced lexicographic order.
Then the induced lexicographic order on the set
$$
(\{ R : R\in\mathcal{L}\} \cup \{\neg R : R \in \mathcal{L}\}) \times (\{x\} \cup \om)^{<\om}
$$
is a linear order on the set of atomic and negated atomic formulas of $\mathcal{L}$ that have one free variable $x$ and parameters
from $\omega$.  Since any node of $\bS$ is completely determined by such atomic and negated atomic formulas, this
lexicographic order gives rise to a linear order on $\bS$, which we denote $\prec$.
 Observe that by definition of the lexicographic ordering, we have:
 If
 $s \subsetneq t$,
then $s\prec t$; and
for any incomparable   $s,t \in\bS$, if
 $|s\wedge t|=n$,
 then
$s\prec t$ if and only if $s\re(n+1)\prec t\re (n+1)$.
This order $\prec$
generalizes
the lexicographic order  for the case of binary relational structures in \cite{Sauer06}, \cite{Laflamme/Sauer/Vuksanovic06}, \cite{DobrinenJML20}, \cite{DobrinenH_k19}, and \cite{Zucker20}.

\begin{defn}[Similarity Map]\label{def.ssmap}
Let  $S$ and $T$  be meet-closed subsets
of $\bS$.
A function $f:S\ra T$ is a {\em similarity map} of $S$ to $T$ if for all nodes
$s, t \in \bS$,
the following hold:
\begin{enumerate}
\item
$f$ is a bijection which preserves $\prec$:
$s\prec t$ if and only if $f(s)\prec f(t)$.

\item
$f$ preserves meets, and hence splitting nodes:
$f(s\wedge t)=f(s)\wedge f(t)$.

\item
$f$ preserves relative lengths:
$|s|<|t|$ if and only if
$|f(s)|<|f(t)|$.

\item
$f$ preserves initial segments:
$s\sse t$ if and only if $f(s)\sse f(t)$.

\item
$f$ preserves  coding  nodes and their
parameter-free formulas:
  Given a coding node $c_n^S\in S$, $f(c_n^S)=c^T_n$;
moreover,
for  $\gamma\in\Gamma$,
$\gamma(v^S_n)$ holds in $\bK$
  if and only if
 $\gamma(v^T_n)$ holds in $\bK$,
 where $v^S_n$ and $v^T_n$ are the vertices of $\bK$ represented by coding nodes
 $c^S_n$ and $c^T_n$, respectively.

\item
$f$ {\em preserves relative passing types} at coding nodes:
$s(n;S)\sim f(s)(n;T)$, for each $n$ such that $|c^S_n|<|s|$.
\end{enumerate}

When there is a similarity map between $S$ and $T$, we say that $S$ and $T$  are {\em similar} and  we write $S\sim T$.
Given a subtree $S$ of $\bS$,
we let $\Sim(S)$ denote the collection of all subtrees  $T$of $\bS$ which are similar to $S$.
If $T'\sse T$ and $f$ is a  similarity map of $S$ to $T'$, then we say that  $f$ is a {\em similarity embedding} of $S$ into $T$.
\end{defn}

\begin{rem}
It follows from  (2)  that
 $s$ is a splitting node in $S$ if and only if  $f(s)$ is a splitting node in $T$.
 Moreover, if $s$ is a splitting node in $S$,
 then $s$ has the same number of immediate successors in $S$ as $f(s)$ has in $T$.
Similarity is an equivalence relation
on the subtrees of $\bS$,
since the
identity map is a similarity map, the
inverse of a  similarity map is a similarity map, and the composition of two similarity maps is a  similarity map.

Our notion of {\em similarity}   extends  the notion of {\em strong similarity} in \cite{Sauer06} and  \cite{Laflamme/Sauer/Vuksanovic06} for trees without coding nodes, and in \cite{DobrinenJML20} and \cite{DobrinenH_k19} for trees with coding nodes.
We drop the word {\em strong} to make the terminology more efficient, since there is only one notion of similarity  used in this paper.
\end{rem}

Given two
substructures  $\bF,\bG$ of $\bK$, we write $\bF\cong^{\om} \bG$
when there exists an $\mathcal{L}$-isomorphism between $\bF$ and $\bG$ that preserves
 the linear order on
their universes inherited from $\om$.
Note that for any subtrees
$S,T$ of $\bS$,
$S\sim T$ implies that $\bK\re S\cong^{\om}\bK\re T$.


\section{Diagonal coding trees and \texorpdfstring{SDAP$^+$}{SDAP+}}

In this section, we introduce concepts utilized in the proof of indivisibility in this paper and in the simple characterization of big Ramsey degrees in \cite{CDPII}.

\subsection{Diagonal coding trees}\label{subsec.3.3}

Our approach to proving indivisibility and to finding exact big Ramsey degrees for structures with unary and  binary relations
 starts with the kinds of trees  that will
actually produce  indivisibility as well as the  exact degrees, upon taking a subcopy of $\bK$ represented by an antichain of coding nodes in such trees.
Namely, we will work with {\em diagonal coding trees}.

First, the following modification of
 Definition \ref{defn.treecodeK}  of $\bS(\bK)$  will be useful especially for \Fraisse\ classes which have both non-trivial unary relations and a linear order or some  similar
relation, such as the betweenness relation.
Recall    that $\Gamma$ denotes the set of complete
$1$-types having only parameter-free formulas;
in particular, the only relation symbols that can occur in any $\gamma \in \Gamma$ will be unary.

\begin{defn}[The Unary-Colored  Coding Tree of $1$-Types,
$\bU(\bK)$]\label{defn.ctU}
Let $\mathcal{K}$ be a \Fraisse\ class in language $\mathcal{L}$ and $\bK$ an enumerated \Fraisse\ structure for $\mathcal{K}$.
For $n < \om$,
let $c_n$ denote the $1$-type
of $v_n$
over $\bK_n$
(exactly as in the definition of $\bS(\bK)$).
Let $\mathcal{L}^-$ denote
the collection of all relation symbols in $\mathcal{L}$ of arity greater than one,
and let $\bK^-$ denote the reduct of $\bK$ to $\mathcal{L}^-$
and $\bK_n^-$ the reduct of $\bK_n$ to $\mathcal{L}^-$.

For $n<\om$, define the {\em $n$-th level, $\bU(n)$}, to  be
the collection
of all $1$-types
$s$ over $\bK^-_n$
in the language $\mathcal{L}^-$
such that
for some $i\ge n$,
$v_i$ satisfies $s$.
Define $\bU$
 to be
$\bigcup_{n<\om}\bU(n)$.
The tree-ordering on $\bU$ is simply inclusion.
The {\em unary-colored coding tree of $1$-types}
is
the tree $\bU$ along with the  function $c:\om\ra \bU$ such that $c(n)=c_n$.
Thus,
$c_n$ is the $1$-type
(in the language $\mathcal{L}^-$) of $v_n$
 in $\bU(n)$  along with the additional ``unary color''
$\gamma\in\Gamma$ such that
$\gamma(v_n)$ holds in $\bK$.
 \end{defn}

Note that $\bK^-$ is not necessarily
a \Fraisse\ structure, as the collection of reducts of members of $\mathcal{K}$ to $\mathcal{L}^-$ need not be a
\Fraisse\ class.
This poses no problem to our uses of $\bU$ or to the results.

\begin{rem}\label{rem.bU}
 In the case that $\mathcal{K}$ has no  unary relations, $\bU$ is  the same as $\bS$.
Otherwise, the difference between $\bU$ and $\bS$
is that all non-coding nodes in $\bU$  are
complete  $1$-types over initial segments of $\bK^-$
in the language $\mathcal{L}^-$,
while all
nodes in $\bS$, coding or non-coding,
are complete $1$-types over initial segments of $\bK$
in the language $\mathcal{L}$.
In particular,  $\bS(0)$ equals $\Gamma$, while
$\bU(0)$
 has  exactly one node, $c_0$.

Definition
\ref{defn.passingtype} of passing type applies  to   $\bU$, as the notion of  passing type involves no unary relations.
Definition
 \ref{defn.prespt}
 of  similarity of passing types
 and Definition \ref{def.ssmap} of similarity maps both apply
to  $\bU$,
 since the notion of coding nodes is the same in both $\bS$ and $\bU$.
Working inside  $\bU$ instead of $\bS$ makes the upper bound arguments   for \Fraisse\ classes with both
a linear order
and unary relations simpler, lining up with the previous approach for big Ramsey degrees of $\bQ_n$ in \cite{Laflamme/NVT/Sauer10}.
This set-up will allow us to do one
uniform
forcing proof in the next section for all classes satisfying \EEAP$^+$.
For  classes with \SFAP, the exact  bound proofs will return to the $\bS$ setting.

Lastly, we point out that the tree $\bU$ extends the approach used by Zucker in \cite{Zucker20} for  certain  free amalgamation  classes
with binary and unary relations.
\end{rem}

The following definition of {\em diagonal}, motivated by Definition 3.2 in \cite{Laflamme/Sauer/Vuksanovic06},
can be found in \cite{DobrinenJML20} and \cite{DobrinenH_k19}.

\begin{defn}[Diagonal tree]\label{def.diagskew}
We call a subtree $T\sse \bS$
or $T\sse\bU$
{\em diagonal}
if each level of $T$ has at most one splitting node,
each splitting node in $T$ has degree two (exactly two immediate successors), and
coding node levels in $T$ have no splitting nodes.
\end{defn}



\begin{notation}\label{notn.cong<}
Given a diagonal subtree $T$
(of $\bS$ or $\bU$)
with coding nodes,
we let
$\lgl c^T_n: n<N\rgl$,
where $N\le\om$,
  denote the enumeration of the coding nodes in $T$ in order of increasing length.
Let $\ell^T_n$ denote  $|c^T_n|$, the {\em length} of $c^T_n$.
We shall call a node in $T$ a {\em critical node} if it is either a splitting node or a coding node in $T$.
Let
\begin{equation}
\widehat{T}=\{t\re n:t\in T\mathrm{\ and \ } n\le |t|\}.
\end{equation}
Given
$s\in T$ that is not a splitting node in $T$,
we let $s^+$ denote the immediate successor of $s$ in $\widehat{T}$.
Given any
$\ell$,
we let $T\re\ell$ denote the set of those nodes in $\widehat{T}$ with length $\ell$,
and we let
  $T\rl \ell$
  denote the
  union of the
  set of nodes in $T$ of length less than  $\ell$
  with the set $T\re\ell$.
\end{notation}

Extending
Notation \ref{notn.KreA} to subtrees $T$ of either $\bS$ or $\bU$, we write $\bK\re T$ to denote the substructure of
$\bK$ on $\mathrm{N}^T$, the set of vertices of $\bK$ represented by the coding nodes in $T$.

\begin{defn}[Diagonal Coding Subtree]\label{defn.sct}
A subtree $T\sse\bU$ is called a {\em diagonal coding subtree} if $T$ is diagonal and  satisfies the following properties:
  \begin{enumerate}
 \item
   $\bK\re T\cong\bK$.
  \item
  For each $n<\om$, the collection of $1$-types
  in
  $T\re (\ell^T_n+1)$ over $\bK\re (T\rl \ell^T_n)$
   is in
    one-to-one correspondence with the collection of
   $1$-types in $\bU(n+1)$.
   \item[(3)]
Given $m<n$
and letting
  $A:=T\rl(\ell^T_m-1)$,
if $c^T_n\contains c^T_m$
 then
 $$
  (c^{T}_n)^+(c^{T}_n; A)
  \sim
  (c^{T}_m)^+(c^{T}_m; A).
$$
\end{enumerate}
Likewise, a subtree $T\sse\bS$ is a {\em diagonal coding subtree} if the above hold with $\bU$ replaced by $\bS$.
\end{defn}

\begin{rem}
Requirement (3) aids in the proofs in the next section and can be  met by
the \Fraisse\ limit of
any \Fraisse\ class satisfying \EEAP.
Note that if $T\sse\bU$
(or $T\sse\bS$)
satisfies (3),
then any subtree $S$ of $T$ satisfying $S\sim T$ automatically satisfies (3).
\end{rem}

Now we are prepared to define the Diagonal Coding Tree Property, which is an assumption in Definition \ref{defn_EEAP_newplus}
of \EEAP$^+$.
We say that a tree  $T$ is  {\em perfect}  if $T$ has  no terminal nodes, and  each node in  $T$ has  at least  two  incomparable extensions in $T$.

Recall  our assumption  that any \Fraisse\ class $\mathcal{K}$ that we consider has
 at least one non-unary relation symbol in its language.
We make this assumption because
if $\mathcal{K}$  has  only unary relation symbols in its language, then $\bS$ is a disjoint union of finitely many infinite branches.
In this case,  finitely many applications of  Ramsey's Theorem  will yield finite big Ramsey degrees.

We point out that whenever $\mathcal{K}$ satisfies \SFAP,
the trees 
  $\bS$ and  $\bU$ are perfect. 
However, there  are \Fraisse\ classes
in binary relational languages
 that satisfy \EEAP, and yet for which the trees $\bS$ and  $\bU$ are not perfect; for example, certain \Fraisse\ classes of ultrametric spaces.
In such cases,
Theorem  \ref{thm.matrixHL}
does not apply, as
the forcing posets
used in its  proof
are atomic.
Thus, one of the requirements for \EEAP$^+$ is that there is a perfect subtree of $\bU$ which codes a copy of $\bK$, whenever $\mathcal{L}$  has relation symbols of arity greater than one.
This is an ingredient in the next property.

\begin{defn}[Diagonal Coding Tree Property]\label{defn.DCTP}
A \Fraisse\ class $\mathcal{K}$ in language $\mathcal{L}$  satisfies the {\em Diagonal Coding Tree Property}
if  given any enumerated \Fraisse\ structure $\bK$
for $\mathcal{K}$,
there is a diagonal coding
subtree $T$ of either $\bS$ or $\bU$
such that $T$ is perfect.
\end{defn}

From here
through most of Section
\ref{sec.FRT},
we will simply work in $\bU$ to avoid duplicating arguments, noting that
for \Fraisse\ classes with \SFAP, or
without \SFAP\ but with \Fraisse\ limits having
\EEAP$^+$ and
in a language with no unary relation symbols,
the following can all be done inside $\bS$.

We now define the  space of coding subtrees of $\bU$ with which we shall be working.

\begin{defn}[The Space of  Diagonal Coding Trees of $1$-Types, $\mathcal{T}$]\label{def.subtree}
Let $\bK$ be any enumerated \Fraisse\  structure
and let $\bT$ be a fixed diagonal coding subtree of $\bU$.
Then the space of coding trees
 $\mathcal{T}(\bT)$ consists of all   subtrees $T$ of $\bT$ such that
 $T\sim\bT$.
Members of $\mathcal{T}(\bT)$ are called simply {\em coding trees}, where diagonal is understood to be  implied.
We shall usually simply write $\mathcal{T}$ when $\bT$ is clear
from context.
For $T\in \mathcal{T}$,  we write
$S\le T$ to mean that  $S$ is a  subtree of $T$ and $S$ is a member of $\mathcal{T}$.
\end{defn}

\begin{rem}
Given  $\bT$ satisfying   (1)--(3) in Definition \ref{defn.sct}, if $T\sse \bT$  satisfies $T\sim \bT$,  then $T$ also satisfies (1)--(3).
Any tree $T$ satisfying  (1) and (2)  has no terminal nodes and has coding nodes dense in $T$.
Condition (2) implies that
the \Fraisse\ structure $\bJ:=\bK\re T$  represented by $T$ has the following  property:
For
any $i-1<j<k$ in $\mathrm{J}$ satisfying
$\bJ\re (i\cup\{j\}) \cong \bJ\re (i\cup\{k\})$,
it holds
 that
$\type(j/\bK_i)=\type(k /\bK_i)$;
equivalently,  that whenever
two vertices in $\mathrm{J}$ are in the same orbit  over
$\bJ_i$ in $\bJ$, they
are in the same orbit over $\bK_i$
in $\bK$.
\end{rem}

The
first use
of diagonal subtrees of the infinite binary tree in characterizing exact big Ramsey degrees
 was
for
 the rationals in \cite{DevlinThesis}.
Diagonal subtrees of the infinite binary tree  turned out to be at the heart of
  characterizing the exact big Ramsey degrees of
the Rado graph  as well as of the generic directed graph and the  generic tournament in
\cite{Sauer06} and
\cite{Laflamme/Sauer/Vuksanovic06}.
More generally,
diagonal subtrees of  boundedly branching trees turned out to be central to the characterization of
big Ramsey degrees of unconstrained
 structures with finitely many binary relations
 in
\cite{Sauer06} and
\cite{Laflamme/Sauer/Vuksanovic06}.
More recently,
  characterizations of the  big Ramsey degrees
for    triangle-free graphs were found to
   involve
 diagonal  subtrees (\cite{DobrinenJML20},\cite{DobrinenH_3ExactDegrees20}), and similarly, for free amalgamation classes with finitely
 many binary relations and finitely many
 finite
 forbidden irreducible substructures  on three or more vertices (\cite{Balko7},\cite{DobrinenH_k19},\cite{Zucker20}).
 However, in these cases,
  properties additional to being diagonal are essential to characterizing their big Ramsey degrees;
  hence, their big Ramsey degrees do not have a ``simple'' characterization solely in terms of similarity types of antichains of coding nodes in diagonal coding trees.
We will prove that, similarly to the rationals and the Rado graph,
 all
unary and binary relational \Fraisse\  classes  with \Fraisse\ structure satisfying  \EEAP$^+$ have big Ramsey degrees which are  characterized  simply by similarity types of
antichains of coding nodes in diagonal  coding trees, along with the passing types of their coding nodes.

Recalling from Notation \ref{notn.cong<}
that
  $t\in T$ is called  a
{\em critical node}
if $t$ is either a splitting node or a coding node in $T$,
any two critical nodes in a diagonal coding tree
 have different lengths, and thus,
the levels of $T$ are designated by the lengths of the critical  nodes in $T$.
(This follows from the definition of {\em diagonal}.)
 If $\lgl d^T_m:m<\om\rgl$  enumerates  the critical nodes in $T$ in order of strictly increasing length,
then we let
$T(m)$ denote the
collection of those nodes in $T$ with length $|d^T_m|$, which we call the {\em $m$-th level} of $T$.

Given a substructure $\bJ$ of $\bK$,
we let $\bU\re\bJ$ denote the subtree of $\bU$
induced by the meet-closure of the coding nodes $\{c_n:n\in \mathrm{J}\}$.
 We call $\bU\re \bJ$ the {\em subtree of $\bU$ induced by $\bJ$}.
 If $\bJ=\bK\re T$ for some  $T\in\mathcal{T}$,
 then $\bU\re \bJ=T$, as $T$ being  diagonal ensures that the  coding nodes in $\bU\re \bJ$ are exactly those in  $T$.

The final work in this subsection is to prove that 
  \Fraisse\ classes  satisfying SFAP,  as well as their ordered expansions,
have \Fraisse\ limits satisfying the Diagonal Coding Tree Property.
The following notation will be used in  the rest of this subsection.
Given  $j<\om$, sets  vertices $\{v_{m_i}:i<j\}$ and
$\{v_{n_i}:i<j\}$, and
$1$-types $s,t\in\bS$ such that $|s|>m_{j-1}$ and $|t|>n_{j-1}$, we will  write
\begin{equation}
s\re (\bK\re\{v_{m_i}:i<j\})\sim t\re (\bK\re\{v_{n_i}:i<j\})
\end{equation}
exactly when, for each
$i<j$,
$s(c_{m_i}; \{c_{m_k}:k<i\})\sim
t(c_{n_i}; \{c_{n_k}:k<i\})$.

\begin{thm}\label{thm.SFAPimpliesDCT}
\SFAP\ implies
the Diagonal Coding Tree Property. 
\end{thm}

\begin{proof}
Suppose $\mathcal{K}$ is a \Fraisse\ class satisfying \SFAP.
Let  $\bK$ be any enumerated \Fraisse\ structure for $\mathcal{K}$, and let $\bS$ be the coding tree of $1$-types over finite initial segments of $\bK$.
Recall that  $c_n$ denotes the $n$-th coding node of $\bS$, that is, the  $1$-type of
the $n$-th vertex of $\bK$ over $\bK_n$.
If there are any  unary relations in the language $\mathcal{L}$ for $\mathcal{K}$, then $\bS(0)$ will have more than one node.
Recall our convention that the ``leftmost'' or $\prec$-least
node in $\bS(n)$  is the $1$-type over $\bK_n$ in which no relations of arity greater than one are satisfied.

We start constructing
a diagonal coding subtree
$\bT$ by letting the minimal level of $\bT$ equal $\bS(0)$.
Take a level set $X$ of $\bS$ satisfying
(a)
for each $t\in \bS(0)$,
the number of nodes in $X$ extending $t$
is the same as the number of nodes in $\bS(1)$ extending  $t$, and
(b)
 the subtree $U_0$ generated by the meet-closure  of $X$ is diagonal.
 We may assume, for convenience, that the $\prec$-order of the  splitting nodes in $U_0$   is the same as the ordering by their lengths.

Let $x_*$ denote the $\prec$-least member of $X$ extending $c_0$.
(If there are no  unary relation symbols in the language, then $x_*$ is the
``leftmost'' or $\prec$-least
node in $X$.)
Let $c^{\bT}_0$ denote the coding node of least length extending
 $x_*$.
Extend the rest of the nodes in $X$  to the length of $c^{\bT}_0$ and call this set of nodes, along with $c^{\bT}_0$,  $Y$;  define $\bT\re |c^{\bT}_0|=Y$.
Then  take one immediate successor in $\bS$  of each member of $Y$ so that  there is  a one-to-one correspondence between the $1$-types in $Y$ over   $\bK\re\{v^T_0\}$,
  where $v_0^{\bT}$ is the vertex in $\bK$ represented by
 $c^{\bT}_0$, and the $1$-types in $\bS(1)$:
Letting $p=|\bS(1)|$,
list the nodes in $\bS(1)$ and $Y$ in $\prec$-increasing order as
$\lgl s_i:i<p\rgl$ and $\lgl y_i:i<p\rgl$, respectively.
Take $z_i$  to be an
 immediate successor of $y_i$ in $\bS$  such that
   $z_i\re (\bK\re\{v_0^{\bT}\})\sim s_i$.
 Such  $z_i$ exist by \SFAP.
 Let $\bT\re (|c^{\bT}_0|+1|)=\{z_i:i<p\}$.
 This constructs $\bT$ up to  length  $|c^{\bT}_0|+1$.

The rest of $\bT$ is constructed similarly:
Suppose  $n\ge 1$ and $\bT$ has been constructed up to the immediate  successors of its $(n-1)$-st coding node, $c^{\bT}_{n-1}$.
Take $W$ to be the set of nodes in $\bT$ of length
 $|c^{\bT}_{n-1}|+1$.
This set $W$ has the same size as $\bS(n)$;
let $\varphi: W\ra \bS(n)$ be the $\prec$-preserving bijection.
Take a level set $X$ of nodes in $\bS$ extending $W$ so that (a) for each $w\in W$, the number of nodes in $X$ extending $w$ is the same as the number of nodes in $\bS(n+1)$ extending $\varphi(w)$, and (b)
the tree $U$ generated by the meet-closure of $X$ is diagonal.
Again, we may assume that the splitting nodes in $U$  increase in length as their $\prec$-order increases.

Note that $X$ and  $\bS(n+1)$ have the same cardinality.
Let $p=|\bS(n+1)|$ and enumerate $X$ in $\prec$-increasing order as $\lgl x_i:i<p\rgl$.
Let  $i_*$ be the index so that
$x_{i_*}$ is the $\prec$-least member of $X$ extending
$\varphi(c_{n})$.
Let $c^{\bT}_{n}$ denote the coding node in $\bS$ of shortest length
extending $x_{i_*}$.
For each  $i\in p\setminus \{i_{*}\}$, take one  $y_i\in \bS$ of length $|c^{\bT}_{n}|$ extending $x_i$.
Let $y_{i_*}=c^{\bT}_{n}$,
$Y=\{y_i:i < p\}$, and
 $\bT\re |c^{\bT}_{n}|=Y$.
Let $\lgl s_i:i<p\rgl$  enumerate the nodes in
$\bS(n+1)$ in   $\prec$-increasing order.
Then for each $i<p$,
let $z_i$ be an immediate successor of $y_i$  in $\bS$ satisfying
\begin{equation}
z_i\re (\bK\re\{v_m^{\bT}:m\le n\})\sim
s_i\re \bK_{n+1},
\end{equation}
where $v_m^{\bT}$ is the vertex of $\bK$ represented by the coding node $c_m^{\bT}$.
Again, such $z_i$ exist by \SFAP.
Let $\bT\re (|c^{\bT}_{n}|+1)=\{z_i:i<p\}$.

In this manner, we construct a subtree $\bT$ of $\bS$.
It is straightforward to check that this construction satisfies  (1) and (2) of Definition \ref{defn.sct}
of diagonal coding tree.
Using  \SFAP,  we may construct $\bT$ so that
 property (3) holds.
As long as the language for $\mathcal{K}$ contains at least one relation symbol, $\bT$ will be  a perfect  tree.
Thus,
any \Fraisse\ limit for
$\mathcal{K}$ satisfies the Diagonal Coding Tree Property.
\end{proof}

Next, we consider  ordered \SFAP\ classes.

\begin{lem}\label{lem.SFAPplusoderimpliesDCT}
Suppose $\mathcal{K}$ is a \Fraisse\ class satisfying \SFAP\
 and let  $\mathcal{K}^{<}$
denote  the \Fraisse\ class
 of  ordered
 expansions of members of
 $\mathcal{K}$.
 Then
 the \Fraisse\ limit $\bK^<$ of
  $\mathcal{K}^{<}$ satisfies the Diagonal Coding Tree Property.
\end{lem}

\begin{proof}
Let $\mathcal{L}$ denote the language for $\mathcal{K}$,
and let $\mathcal{L}^*$ be the expansion $\mathcal{L}\cup\{<\}$, the language of $\mathcal{K}^<$.
Let $\bK^{<}$ denote
an enumerated structure for $\mathcal{K}^{<}$,
and let $\bK$ denote the  reduct of  $\bK^{<}$ to $\mathcal{L}$; thus, $\bK$ is an enumerated \Fraisse\ structure for $\mathcal{K}$.
The universes of $\bK$ and $\bK^{<}$ are  $\om$, which  we shall denote as
 $\lgl v_n:n<\om\rgl$.
Let
$\bU$ denote  the coding tree of $1$-types induced by $\bK$, and
$\bU^{<}$ denote  the coding tree of $1$-types induced by $\bK^{<}$.
(As in the case for the $n$-partite graphs, if $\mathcal{K}$ has unary relations  and these unary relations do not occur densely in $\bU$, then work in $\bS$ and $\bS^<$ instead.)
As usual, we let $c_n$ denote the $n$-th coding node in $\bK$, and we will let $c_n^<$ denote the $n$-th coding node in $\bK^<$.
(Normally, if $<$ is in the language of a \Fraisse\ class $\mathcal{K}$, then we will simply write $\bK$ for its enumerated \Fraisse\ structure and $\bU$ for its induced coding tree of $1$-types, but here it will aid the reader to consider the juxtaposition of $\bU$ and $\bU^<$.)
Notice that  $\mathcal{K}^<$ satisfies  \EEAP:
 This holds because
  \SFAP\ implies \EEAP,
 $\mathcal{LO}$ satisfies \EEAP, and
 \EEAP\ is preserved under free superposition.
So it only remains to show that there is a diagonal coding tree for $\bK^<$.

Note that since $\mathcal{L}$ has  at least one non-unary relation symbol and since
$\mathcal{K}$ satisfies \SFAP,
the tree $\bU$  is perfect.
The branching of $\bU$ and $\bU^<$ are related in the following way:
Each node $t\in\bU(0)^<$ has twice as many immediate successors in $\bU(1)^<$
as its reduct to $\mathcal{L}$ has in $\bU(1)$.
In general, for $n\ge 1$,
given a node  $t\in \bU^<(n)$,
let  $s$  denote the collection of formulas in $t$ using only relation  symbols in   $\mathcal{L}$
 and note that  $s\in \bU(n)$.
  The number of immediate successors of $t$ in
  $ \bU^<(n+1)$ is related to the number of immediate successors of $s$ in   $\bU(n+1)$ as follows:
  Let $(*)_n(t)$ denote the following property:
  \begin{enumerate}
  \item[$(*)_n(t)$:]
  \begin{center}
   $\{m<n:(x<v_m)\in t\}=\{m<n:(x<v_m)\in c^<_n\}$
   \end{center}
\end{enumerate}
If $(*)_n(t)$ holds,
then $t$ has twice as many immediate successors in
 $\bU^<(n+1)$ as $s$ has in  $\bU(n)$, owing to the fact that
 each $1$-type in $\bU(n+1)$ extending $s$  can be augmented by either of $(x<v_n)$ or $(v_n<x)$ to form an extension of $t$ in $\bU^<(n+1)$.
 If $(*)_n(t)$ does not  hold,
 then any  vertex  $v_i$, $i>n$, satisfying  $t$
  lies in an interval of the $<$-linearly ordered set $\{v_m:m<n\}$, where neither of  the endpoints  are  $v_n$.
 Thus, the order between $v_i$ and $v_n$ is already  determined by $t$;
 hence $t$ has the same number of immediate successors in
  $\bU^<(n+1)$ as $s$ has in  $\bU(n)$.

A diagonal coding subtree $\bT^<$ of $\bU^<$  can be constructed similarly as in Theorem \ref{thm.SFAPimpliesDCT}
with the following modifications:
Suppose $\bT^<$ has been constructed up to  a
level  set $W$, where either  $n=0$ and $W=\bU^<(0):=\{c^<_0\}$, or else $n\ge 1$ and  $W$ is the set of immediate successors of the $(n-1)$-st coding node  of $\bT^<$.
This set $W$ has the same size as $\bU^<(n)$;
let $\varphi: W\ra \bU^<(n)$ be
the $\prec$-preserving bijection.
As $\mathcal{K}$ is a free
amalgamation
class,
we may assume that for any $s\in\bU^<(n)$,
if $t,u$ in
$\bU^<(n+1)$  are
 immediate successors   of $s$  with
$(x< v_n)\in t$ and $(v_n<x)\in u$,
then $t\prec u$.
Note that the two $\prec$-least extensions of $s$ either both contain $(x<v_n)$, or else both contain
$(v_n<x)$.
Moreover, we may assume that the $\prec$-least immediate successor of $s$
contains negations of all relations in $s$ with $v_n$ as
a parameter.

Take a level set $X$ of nodes in $\bU^<$ extending $W$ so that  the following hold:
(a)
for each $w\in W$, the number of nodes in $X$ extending $w$ is the same as the number of nodes in $\bU^<(n+1)$ extending $\varphi(w)$, and
(b)
the tree $U$ generated by the meet-closure  of $X$ is diagonal,
 where each splitting node in $U$ is extended by
 its  two $\prec$-least immediate successors in $\bU^<$, and
  all non-splitting nodes are extended by the $\prec$-least extension in
  $\bU^<$.
As in Theorem \ref{thm.SFAPimpliesDCT},
 we may assume that the splitting nodes in $U$  increase in length as their $\prec$-order increases, though this has no bearing on the  theorems in the next section.

Let $p:=|\bU^<(n+1)|$ and index the nodes in $\bU^<(n+1)$ in $\prec$-increasing order as $\lgl s_i:i<p\rgl$.
Note that $X$ has $p$-many nodes; index them in $\prec$-increasing order as $\lgl x_i:i<p\rgl$.
Let $x_{i_*}$ denote  the $\prec$-least member of $X$ extending
$\varphi(c^{<}_{n})$, and
extend $x_{i_*}$ to a coding node  in $\bU^<$ satisfying the same $\gamma\in\Gamma$ as $c_n$; label it $y_{i_*}$.
This node $y_{i_*}$ will be the $n$-th coding node, $c^{\bT^<}_{n}$, of  the diagonal coding subtree $\bT^<$ of $\bU^<$ which we are constructing.
For each  $i\in p\setminus \{i_*\}$,
 take one  $y_i\in \bU^<$ of length $|c^{\bT^<}_{n}|$ extending $x_i$
so that  $y_i$ is the
 $\prec$-least  extension of $x_i$, subject to  the following:
Let $n_*$ be the index such that
 $c^{\bT^<}_{n}=c^{<}_{n_*}$.
For $i<p$, if  $(v_{n}<x)$ is in $s_i$,
then we take $y_i$ so that for some
$m<n_*$
such that $v_n<v_m$,
$(v_m<x)$ is in $y_i$.
This has the effect that
if  $(v_{n}<x)$ is in $s_i$,
then any vertex $v_j$ represented by a coding node extending $s_i$ will satisfy $v_m<v_j$;
and since $v_n<v_m$, it will follow that $v_n<v_j$;
hence  $(v_{n_*}<x)$ is automatically in $y_i$.
Likewise, if
 $(x<v_{n})$ is in $s_i$,
then we take $y_i$ so that for some
$m<n_*$
such that $v_m<v_n$,
$(x<v_m)$ is in $y_i$.

Let $Y=\{y_i:i<p\}$ and
define  the  set of nodes  in
$\bT^<$ at the level of $c^{\bT^<}_{n}$ to be $Y$.
For each $i<p$,
let $z_i$ be an immediate successor of $y_i$  in $\bU^<$ satisfying
\begin{equation}
z_i\re (\bK\re\{v_j^{\bT^<}:j\le n\})\sim
s_i\re \bK_{n+1},
\end{equation}
where $v_j^{\bT^<}$ is the vertex
represented by $c_j^{\bT^<}$.
This is possible by \SFAP.
For the linear order, this was taken care of by \EEAP\ and
our choice of $y_i$.
Let $\bT^<\re (|c^{\bT}_{n}|+1)=\{z_i:i<p\}$.

In this manner, we construct a coding subtree $\bT^<$ of $\bU^<$ which is diagonal, representing a substructure of $\bK^<$ which is again isomorphic to
$\bK^<$.
By extending coding nodes in $\bT^<$ by their $\prec$-least extensions in $\bU^<$,
we satisfy
 (3) of  the definition of diagonal coding tree .

 Hence,
 $\bK^<$
satisfies the Diagonal Coding Tree Property.
\end{proof}

We will work in  a diagonal coding subtree of $\bS$ whenever such a subtree exists.
This is always the case for \Fraisse\ classes satisfying SFAP.
For \Fraisse\ limits   with no unary relations satisfying SDAP,
note that  $\bS=\bU$; so in this case, a diagonal coding subtree of $\bU$ is the same as a diagonal coding subtree of $\bS$.
If $\bK$ is a \Fraisse\ class   with unary relations  satisfying SDAP and there is 
 a diagonal coding subtree of $\bU$ but 
no diagonal coding subtree of $\bS$,
then there  are subsets $P_0,\dots, P_j$ of the unary 
 relation symbols of $\mathcal{K}$ 
and a  diagonal coding
subtree $T\sse \bU$ such that  at some level  $\ell$  below the first coding node of $T$,  the following hold:
$T\re\ell$ has exactly $j+1$ nodes, say $t_0,\dots,t_j$, and for each $i\le j$, 
every coding node in the tree $T$ restricted above $t_i$  has unary relation in $P_i$ and moreover, each of the unary relations in $P_i$ occurs densely in  $T$  restricted above $t_i$.
By possibly adding unary relation symbols, we may assume that $P_0,\dots,P_j$ is a partition of the unary relation symbols. 
Thus, without loss of generality, we will hold to the following  convention  for the remainder of this article.

\begin{convention}\label{conv.Gamma_ts}
Let $\mathcal{K}$ be a \Fraisse\ class in a language $\mathcal{L}$ and $\bK$ a \Fraisse\ limit of $\mathcal{K}$.
Either there is a diagonal coding subtree of $\bS(\bK)$,
or else there is a diagonal coding subtree of $\bU(\bK)$ in which all unary relations occur densely. 
 \end{convention}


\subsection{The Extension Property}\label{subsec.EP}

In this section we will define the Extension Property, which is the last  of the requirements for \EEAP$^+$ to hold.

Let $T$ be a diagonal coding tree for the \Fraisse\ limit $\bK$ of some \Fraisse\
class $\mathcal{K}$.
Recall that the tree ordering on $T$ is simply inclusion.
We recapitulate notation from Subsection \ref{subsec.3.1}:
Each $t \in T$ can be thought of as a sequence $\lgl t(i) : i < |t| \rgl$ where $t(i) = (t \re \bK_i) \!\setminus\! (t \re \bK_{i-1})$.
For
 $t\in T$ and $\ell\le |t|$,
$t\re \ell$ denotes
$\bigcup_{i<\ell}t(i)$, which we can
think of
as the sequence $\lgl t(i):i< \ell\rgl$,
the initial segment of $t$ with domain $\ell$.
Note that
$t\re\ell\in \bS(\ell-1)$
(or $t\re \ell\in \bU(\ell-1)$).
(We let $\bS(-1)=\bU(-1)$ denote the set containing the empty set, just so that we do not have to always write $\ell\ge 1$.)

The following extends Notation
\ref{notn.cong<}
to
subsets of trees.
For a  finite subset $A\sse\bT$,  let
\begin{equation}
 \ell_A=\max\{|t|:t\in A\}\mathrm{\ \ and\ \ } \max(A)=
\{s\in A: |s|=\ell_A\}.
 \end{equation}
For $\ell\le \ell_A$,
let
\begin{equation}
A\re \ell=\{t\re \ell : t\in A\mathrm{\ and\ }|t|\ge \ell\}
\end{equation}
 and let
\begin{equation}
A\rl \ell=\{t\in A:|t|< \ell\}\cup A\re \ell.
\end{equation}
Thus, $A\re \ell$ is a level set, while $A\rl \ell$ is the set of nodes in $A$ with length less than $\ell$ along with the truncation
to $\ell$ of the  nodes in $A$ of length at least
 $\ell$.
Notice that
$A\re \ell=\emptyset$ for $\ell>\ell_A$, and
 $A\rl \ell=A$  for  $\ell\ge \ell_A$.
 Given $A,B\sse T$, we say that $B$ is an {\em initial segment} of $A$ if   $B=A\rl \ell$
 for some $\ell$ equal to
   the length of some node in $A$.
   In this case, we also say that
   $A$ {\em end-extends} (or just {\em extends}) $B$.
If $\ell$ is not the length of any node in $A$, then
  $A\rl \ell$ is not a subset  of $A$, but  is  a subset of $\widehat{A}$, where
  $\widehat{A}$ denotes $\{t\re n:t\in A\mathrm{\ and\ } n\le |t|\}$.

Define $\max(A)^+$ to be the set of nodes
$t$ in $T\re (\ell_A+1)$ such that $t$ extends $s$ for some $s \in \max(A)$.
Given a node $t\in T$ at the level of a coding node in $T$, $t$ has exactly one immediate successor in $\widehat{T}$, which  we recall
from Notation \ref{notn.cong<}
is  denoted as
$t^+$.

\begin{defn}[$+$-Similarity]\label{def.plussim}
Let $T$ be a diagonal coding tree for
the \Fraisse\ limit $\bK$ of
a \Fraisse\
class $\mathcal{K}$, and
suppose $A$ and $B$
are finite subtrees of $T$.
We write  $A\plussim B$ and say that
 $A$ and $B$ are
 {\em $+$-similar} if and only if
  $A\sim B$ and
 one of the following two cases holds:
 \begin{enumerate}
 \item[]
   \begin{enumerate}
\item[\bf Case 1.]
 If $\max(A)$  has a splitting node in $T$,
 then so does $\max(B)$,
  and the similarity map from $A$ to $B$  takes the splitting node in $\max(A)$ to the splitting node in $\max(B)$.
    \end{enumerate}
     \end{enumerate}
       \begin{enumerate}
    \item[]
        \begin{enumerate}
  \item[\bf Case 2.]
If $\max(A)$ has a coding node, say
 $c^A_n$,
and  $f:A\ra B$ is  the similarity map,
then
  $s^+(n;A)\sim f(s)^+(n;B)$ for each $s\in \max(A)$.
 \end{enumerate}
    \end{enumerate}

Note that $\plussim$ is an  equivalence relation, and  $A\plussim B$ implies $A\sim B$.
When $A\sim B$ ($A\plussim B$), we say that they have the same {\em similarity type} ({\em $+$-similarity type}).
\end{defn}

\begin{rem}\label{rem.SplussimT}
For infinite trees $S$ and $T$ with no terminal  nodes, $S\sim T$ implies that for
each $n$, letting $d^S_n$ and $d^T_n$ denote the $n$-th critical nodes of $S$ and $T$, respectively,
$S\re|d_n^S|\plussim T\re|d_n^T|$.
\end{rem}

We adopt the following notation  from topological Ramsey space theory (see \cite{TodorcevicBK10}).
Given  $k<\om$,
we define
$r_k(T)$ to be  the
restriction of $T$ to the levels of the first $k$ critical nodes of $T$;
that is,
\begin{equation}
r_k(T)=\bigcup_{m<k}T(m),
\end{equation}
where $T(m)$ denotes the set of all nodes in $T$ with length equal to $|d^T_m|$.
It follows from Remark \ref{rem.SplussimT} that  for any
 $S,T\in \mathcal{T}$,
$r_k(S)\plussim r_k(T)$.
Define $\mathcal{AT}_k$ to be the set of {\em $k$-th approximations}
to members of $\mathcal{T}$;
that is,
\begin{equation}
\mathcal{AT}_k=\{r_k(T):T\in\mathcal{T}\}.
\end{equation}
For $D\in\mathcal{AT}_k$
and $T\in\mathcal{T}$,
define the set
\begin{equation}
[D,T]=\{S\in \mathcal{T}:r_k(S)=D\mathrm{\ and\ } S\le T\}.
\end{equation}
Lastly, given
 $T\in\mathcal{T}$,
 $D=r_k(T)$,  and $n>k$,
  define
\begin{equation}
r_n[D,T]=\{r_n(S):S\in [D,T]\}.
\end{equation}

More generally, given any $A\sse T$, we use $r_k(A)$ to denote the first $k$ levels of the tree induced by the meet-closure of $A$.
We now have the necessary ideas to define the Extension Property.

Recall from Convention \ref{conv.Gamma_ts} that $\bT$ is either a fixed diagonal coding tree in $\bS$ or else is  a fixed diagonal coding tree in $\bU$ such that all unary relations occur densely in $\bT$,
for an enumerated \Fraisse\ limit $\bK$ of
a \Fraisse\ class $\mathcal{K}$.

\begin{defn}[Extension Property]\label{defn.ExtProp}
We say that
$\bK$
has the {\em Extension Property} when
the following condition  (EP) holds:
\begin{enumerate}
\item[(EP)]
Suppose $A$ is a finite or infinite subtree of   some
$T\in\mathcal{T}$.
Let $k$ be given  and suppose
$\max(r_{k+1}(A))$ has a splitting node.
Suppose that $B$ is a $+$-similarity copy of $r_k(A)$ in $T$.
Let  $u$ denote the splitting node in $\max(r_{k+1}(A))$,
and let
 $s$ denote  the node in $\max(B)^+$ which must be extended to a splitting node in order to obtain a $+$-similarity copy of $r_{k+1}(A)$.
If $s^*$ is a splitting node  in $T$  extending $s$,
then  there are extensions of the rest of the nodes in $\max(B)^+$ to the same length as $s^*$ resulting in a $+$-similarity copy
of $r_{k+1}(A)$ which
can be extended to a copy of $A$.
\end{enumerate}
\end{defn}

\begin{rem}
The Extension Property   easily holds  for \Fraisse\ limits of all \Fraisse\ classes satisfying \SFAP, as we show below in Lemma \ref{lem.SFAPEP};
and similarly  for their ordered
expansions.
The same is true for   \Fraisse\ limits of all unrestricted  \Fraisse\ classes and their ordered expansions.
In these cases, all splitting nodes   in $\bT$
allow for the
construction of a $+$-similarity copy of $A$.
The \Fraisse\ structures $\bQ_n$
also trivially have the Extension Property.

\end{rem}

\begin{lem}\label{lem.SFAPEP} 
\SFAP\ implies the Extension Property.
Similarly, the \Fraisse\ limit of any SFAP class with an ordered expansion satisfies the Extension Property. 
\end{lem}

\begin{proof}
We will actually prove a slightly stronger statement which implies
the Extension Property.
Let $A$ be a subtree of some $T\in\mathcal{T}$.
Without loss of generality,  we may assume that either $A$ is infinite  and has infinitely many coding nodes, or else $A$ is finite and the  node in $A$ of maximal length is a coding node.
Let $m$ either be $0$,
 or else  let $m$ be a positive integer such that $\max(r_m(A))$ has a coding node.
 Let $n>m$ be least above $m$ such that $\max(r_n(A))$ has a coding node; let $c^A_i$ denote this coding node.

 Now suppose that $B$ is a $+$-similarity copy of $r_m(A)$, and
suppose  $C$ is an extension of $B$ in $T$ such that $C$ is $+$-similar to $r_{n-1}(A)$.
(Such a $C$ is easy to construct since $\bS$ is a perfect tree whenever
$\bK$ has at least one non-trivial relation of arity greater than one.)
Let $X$ denote $\max(r_{n-1}(A))^+$,
let $Y$ denote $\max(C)^+$, and let
$\varphi$ be the $+$-similarity map from $X$ to $Y$.
Let $t$ denote the node in $X$ which extends to the coding node in $\max(r_n(A))$, and let
$y$  denote  $\varphi(t)$.
Extend $y$ to some coding node  $c^T_{i'}$ in $T$ such that the substructure  of $\bK$
represented by the coding nodes in $B$ along with $c^T_{i'}$ is isomorphic to the substructure of $\bK$ represented by the coding nodes in $r_n(A)$.

Fix any $u\in X$ such that $u\ne t$, and
let $z$ denote $\varphi(u)$.
Let $c^T_j$ denote the least coding node in $A$ extending $u$.
By \SFAP, there
is an extension  of $z$ to some coding node $c^T_{j'}$ representing a vertex $w'$ in $\bK$ such that
the substructure of $\bK$ represented by the coding nodes in $B$ along with $c^T_{i'}$ and $c^T_{j'}$
is isomorphic to  the substructure of $\bK$ represented by the coding nodes in $r_n(A)$ along with $c^T_j$.
Let $u'$ denote the unique extension of $u$ in $\max(r_n(A))$, and let $z'$ denote the
 truncation of $c^T_{j'}$ to the length $|c^T_{i'}|+1$.
 Then $(z')^+(c^T_{j'};B)\sim (u')^+(c^T_i;r_m(A))$.
  Therefore, the union of $C$ along with
 $\{u':u\in Y\setminus\{y\}\}\cup \{c^T_{i'}\}$ is $+$-similar to $r_n(A)$.
 It follows that the Extension Property holds.
 
 The proof for the  ordered expansion of an SFAP class is similar. 
\end{proof}


\subsection{Substructure Disjoint Amalgamation Property$^+$}\label{subsec.SDAPplus}

We now have all the components needed to define the strengthened version of 
Substructure Disjoint Amalgamation Property central to our results.

\begin{defn}[\EEAP$^+$]\label{defn_EEAP_newplus}
A \Fraisse\
structure $\bK$
has the
{\em  \EEAPnonacronym$^+$ (\EEAP$^+$)} if
its age
$\mathcal{K}$ satisfies \EEAP,
and $\bK$ has
the Diagonal Coding Tree Property and the Extension Property.
\end{defn}

We point out  that while the Diagonal Coding Tree Property and Extension Property are defined in terms of an enumerated \Fraisse\ structure, they are independent of the chosen enumeration, and hence \EEAP$^+$ is a property of a \Fraisse\ structure itself.

By previous lemmas, it follows that   SFAP implies SDAP$^+$.

\begin{thm}\label{thm.SFAPSDAPplus}
Let $\mathcal{K}$ be a \Fraisse\ class in a language with finitely many relation symbols of any finite arity satisfying SFAP.
Then the \Fraisse\ limit $\bK$ of $\mathcal{K}$ and the \Fraisse\ limit $\bK^{<}$ of the ordered expansion $\mathcal{K}^{<}$ both satisfy SDAP$^+$.
\end{thm}

\begin{proof}
It follows immediately from the definitions 
that if $\mathcal{K}$ satisfies SFAP, then both $\mathcal{K}$ and $\mathcal{K}^{<}$ satisfy SDAP.
By Lemmas \ref{thm.SFAPimpliesDCT},
\ref{lem.SFAPplusoderimpliesDCT}, and 
\ref{lem.SFAPEP},
their \Fraisse\ limits satisfy SDAP$^+$.
\end{proof}

The motivation behind    \EEAP$^+$
was  to distill
the essence of those \Fraisse\ classes for which
the forcing arguments in
Theorem \ref{thm.matrixHL}
 work.
As such, it yields quick proofs via forcing of  indivisibility (Part I) as well as efficient proofs of 
big Ramsey degrees which have simple characterizations, similar to those of the rationals and the Rado graph (Part II).
It is known
that \EEAP$^+$, and even \EEAP,
are not necessary for obtaining  finite  big Ramsey degrees.
For instance, generic  $k$-clique-free graphs    \cite{DobrinenH_k19} and the generic partial order \cite{Hubicka_CS20} have been shown to have  finite big Ramsey degrees,
and their ages do not have \EEAP.

We now present a 
coding tree version of \EEAP$^+$.
This version
is implied by
Definition \ref{defn_EEAP_newplus} and  will be used in  proofs.

\begin{defn}[\EEAP$^+$, Coding Tree Version]\label{def.EEAPCodingTree}
A \Fraisse\ class $\mathcal{K}$ satisfies the
 {\em
Coding Tree Version of  \EEAP$^+$}
 if and only if   $\mathcal{K}$ satisfies
 the disjoint amalgamation property
 and,
 letting $\bK$ be any enumerated \Fraisse\ limit of $\mathcal{K}$, $\bK$ satisfies
 the Diagonal Coding Tree Property,
the Extension Property, and
the following
condition:

Let $T$ be any  diagonal coding  subtree   of $\bU(\bK)$  (or of $\bS(\bK)$), and let $\ell<\om$ be given.
Let
$i,j$ be any distinct integers such that
$\ell<\min(|c^T_i|,|c^T_j|)$,
and let
$\bfC$ denote  the substructure of $\bK$  represented by the coding nodes in $T\rl \ell$ along with
 $\{c^T_i,c^T_j\}$.
Then there are $m \ge\ell$
and
$s',t'\in T\re m$ such that
$s'\contains s$ and $t'\contains t$
 and,
assuming (1) and (2), the conclusion holds:
 \begin{enumerate}
\item[(1)]
Suppose $n\ge m$ and
$s'',t''\in T\re n$
 with $s''\contains s'$ and
$t''\contains t'$.

\item[(2)]
Suppose $c^T_{i'}\in T$ is any coding node extending $s''$.
\end{enumerate}
Then
there is a coding node $c^T_{j'}\in T$, with $j'>i'$,
such that
$c_{j'}\contains t''$ and
the substructure  of $\bK$  represented by the coding nodes in $T\rl \ell$ along with $\{c^T_{i'},c^T_{j'}\}$ is isomorphic to
$\bfC$.
\end{defn}


\subsection{LSDAP$^+$}\label{subsect.LSDAP}

We now present the Labeled Substructure Disjoint Amalgamation Property$^+$ which is applicable to structures such as $\bQ_{\bQ}$, the \Fraisse\ limit of the \Fraisse\ class in language $\mathcal{L}=\{<,E\}$ of  equivalence relations where  each equivalence class is convex. 
It should be thought of as a weakening of SDAP$^+$, for 
if we were to allow $q=1$ in the following definitions, 
SDAP$^+$ would be recovered.

\begin{defn}[Labeled Diagonal Coding Tree]\label{defn.LDCT}
We say that a diagonal coding tree $T$ is {\em labeled} if the following hold:
There
is some $2\le q<\om$,
and a function
 $\psi$ defined on the set of splitting nodes in $\bT$ and  having range $q$,
 such that  the following holds:
 \begin{enumerate}
 \item[(a)]
 If $s\sse t$ are splitting nodes in $T$, then $\psi(s)\ge \psi(t)$.
 \item[(b)]
 For each splitting node $s\in T$ and each $n>|s|$, there is a splitting node $t\contains s$ with $|t|\ge n$ such that $\psi(t)=\psi(s)$.
\item[(c)]
The language for $\bK$ has relation symbols of arity at most two, and
 each $m<q$  
 corresponds  to a pair of partial $1$-types $(\sigma_m,\tau_m)$ involving only binary relation symbols over a $1$-element structure
 such that
whenever $s$ is a splitting node in $\bT$,
$\psi(s)=m$ if and only if the following hold:
whenever  $c^{\bT}_j,c^{\bT}_k$ are coding nodes in $\bT$ with $c^{\bT}_j\wedge c^{\bT}_k=s$ and $c^{\bT}_j\prec c^{\bT}_k$,
if $j<k$ then $c^{\bT}_k(|c^{\bT}_j|)\sim\tau_m$ and if $j>k$ then 
$c^{\bT}_j(|c^{\bT}_k|)\sim\sigma_m$.
\item[(d)]
The maximal splitting node $s$ below a coding node in $\bT$ has $\psi(s)=0$.
\end{enumerate}
\end{defn}

Given (a) and (b), the function $\psi$ can be extended to all nodes of $T$ as follows:
For each non-splitting node $t\in T$, define $\psi(t)$ to equal $\psi(s)$, where $s$ is the maximal splitting node in $T$ such that $s\sse t$.

\begin{notation}
For a labeled diagonal coding tree $\bT$, for $S,T$ subtrees of $\bT$,
write $S\Lsim T$ to mean that $S\sim T$ and the similarity map $f:S\ra T$ preserves $\psi$,
meaning that for each $s\in S$,
$\psi(s)=\psi(f(s))$.
\end{notation}

\begin{defn}[${L+}$-Similarity]\label{def.Lplussim}
Let $T$ be a labeled diagonal coding tree with labeling function $\psi$ for
the \Fraisse\ limit $\bK$ of
a \Fraisse\
class $\mathcal{K}$, and
suppose $A$ and $B$
are finite subtrees of $T$.
We write  $A\Lplussim B$ and say that
 $A$ and $B$ are
 {\em $\mathrm{L}+$-similar} if and only if
  $A\plussim B$ and 
  $A\Lsim B$.
\end{defn}

\begin{defn}[Labeled Extension Property]\label{LEP}
We say that $\bK$ has the {\em Labeled Extension Property} when the following condition (LEP) holds:
\begin{enumerate}
\item[(LEP)]
There
is some $2\le q<\om$
and a labeling function
 $\psi$ taking $\bT$ onto $q$ satisfying Definition \ref{defn.LDCT}
 such that  the following holds:
Suppose $A$ is a finite or infinite subtree of    some $T\in\mathcal{T}$.
Let $k$ be given  and suppose
$\max(r_{k+1}(A))$ has a splitting node.
Suppose that $B$ is an $\mathrm{L}+$-similarity copy of $r_k(A)$ in $T$.
Let  $u$ denote the splitting node in $\max(r_{k+1}(A))$,
and let
 $s$ denote  the node in $\max(B)^+$ which must be extended to a splitting node in order to obtain a $+$-similarity copy of $r_{k+1}(A)$, and note that $\psi(s)\ge \psi(u)$.
 Then  for each  $s'\contains s$ in $T$ with $\psi(s')\ge \psi(u)$,
 there exists  a splitting node $s^*\in T$ extending $s'$ such that $\psi(s^*)=\psi(u)$.
 Moreover,
 given such an $s^*$,
 there are extensions of the rest of the nodes in $\max(B)^+$ to the same length as $s^*$ resulting in an $\mathrm{L}+$-similarity copy of $r_{k+1}(A)$.
\end{enumerate}
\end{defn}


\begin{defn}[LSDAP$^+$]\label{defn_LSDAP_plus}
A \Fraisse\
structure $\bK$
has the
{\em  Labeled Substructure Disjoint Amalgamation Property$^+$ (LSDAP$^+$)} if
its age
$\mathcal{K}$ satisfies SDAP,
and $\bK$ has a labeled diagonal coding tree satisfying 
the Diagonal Coding Tree Property and the Labeled Extension Property.
\end{defn}

\begin{defn}[The space of diagonal coding trees for LSDAP$^+$ structures]\label{spacectLSDAP}
If $\bK$ satisfies LSDAP$^+$, then given a diagonal coding tree $\bT$ for $\bK$ with labeling $\psi$,
we let $\mathcal{T}$ denote the set of all subtrees $T$ of $\bT$ such that $T\Lsim \bT$.
\end{defn}


\section{Indivisibility via forcing the Level Set Ramsey Theorem}\label{sec.FRT}

In  Theorem \ref{thm.matrixHL},
 we  use the technique of forcing to  essentially conduct an unbounded search for a finite object,
 achieving within ZFC
one color per level set extension of a given finite
tree.
It is important to
note
that we never actually go to a generic extension.
In fact,  the forced generic object  is very much {\em not} a coding tree and will not represent a \Fraisse\ limit.
Rather, we use the forcing to
do two things:
(1)
Find a good set of nodes from which we can start to build a subtree which can have the  desired homogeneity properties; and
(2)
Use the forcing to guarantee
 the existence of  a finite object with certain properties.
 Once found, this object, being finite, must exist in the ground model.

We take here a sort of amalgamation of techniques  developed in  \cite{DobrinenJML20}, \cite{DobrinenH_k19}, and \cite{DobrinenRado19}, making adjustments as necessary.
The main differences from previous work are the following:
The forcing poset is on diagonal coding  trees of
 $1$-types; as such,
  we work
with the  general notion of  passing type, in place of  passing number used in the papers \cite{DobrinenRado19}, \cite{DobrinenH_k19}, \cite{DobrinenJML20}, and \cite{Zucker20} for  binary relational structures.
Moreover, Definitions \ref{def.plussim} and \ref{def.Lplussim} present stronger requirements  than just similarity.
These
address
both the fact that relations can be of any arity,
and the fact that we consider \Fraisse\  classes
which
have disjoint, but not necessarily free, amalgamation.

We now set up notation, definitions,  and assumptions for
Theorem \ref{thm.matrixHL}.
Recall Convention \ref{conv.Gamma_ts}.




By an {\em antichain} of coding nodes, we mean a set of  coding nodes  which is pairwise  incomparable with respect to  the tree partial order of inclusion.
\vskip.1in

\noindent\bf{Set-up for Theorem \ref{thm.matrixHL}.} \rm
Let $T$ be a diagonal coding tree  in $\mathcal{T}$.
Fix a finite antichain  of coding nodes  $\tilde{C}\sse T$.
We abuse notation and also write  $\tilde{C}$ to denote  the tree that its meet-closure induces in $T$.
Let $\tilde{A}$ be a fixed proper initial segment of
$\tilde{C}$, allowing for $\tilde{A}$ to be the empty set.
Thus, $\tilde{A}= \tilde{C}\rl \ell$, where $\ell$
is the length of some splitting or  coding node in
$\tilde{C}$
(let $\ell=0$ if $\tilde{A}$ is empty).
Let $\ell_{\tilde{A}}$ denote this $\ell$, and note that
any non-empty
 $\max(\tilde{A})$ either  has a coding node or a splitting node.
Let $\tilde{x}$ denote the shortest splitting or coding node in $\tilde{C}$ with length greater than $\ell_{\tilde{A}}$,
and define  $\tilde{X}=\tilde{C}\re |\tilde{x}|$.
Then $\tilde{A}\cup\tilde{X}$ is an initial segment of
 $\tilde{C}$; let $\ell_{\tilde{X}}$ denote $|\tilde{x}|$.
There are two cases:

\begin{enumerate}
\item[]
\begin{enumerate}
\item[\bf Case (a).]
$\tilde{X}$ has a splitting node.
\end{enumerate}
\end{enumerate}

\begin{enumerate}
\item[]
\begin{enumerate}
\item[\bf Case (b).]
$\tilde{X}$ has a coding node.
\end{enumerate}
\end{enumerate}

Let $d+1$ be the number of nodes in $\tilde{X}$ and  index these nodes as $\tilde{x}_i$, $i\le d$,
where $\tilde{x}_d$ denotes  the critical  node (recall that {\em critical node}  refers to  a splitting or  coding node).
Let
\begin{equation}\label{eq.tildeB}
\tilde{B}=\tilde{C} \re (\ell_{\tilde{A}}+1).
\end{equation}
Then
  $\tilde{X}$ is a level set   equal to or  end-extending  the level set $\tilde{B}$.
  For each $i\le d$, define
  \begin{equation}
  \tilde{b}_i=\tilde{x}_i\re \ell_{\tilde{B}}.
\end{equation}
Note that we
consider nodes in  $\tilde{B}$  as simply nodes to be extended; it
does not matter
whether the nodes in $\tilde{B}$ are coding, splitting, or neither in $T$.

\begin{defn}[Weak similarity]\label{defn.weaksim}
Given finite subtrees $S,T\in\mathcal{T}$ in which each coding node is terminal,
we say that  $S$ is {\em weakly similar} to $T$, and write $S\wsim T$,
if and only if
$S\setminus \max(S)\plussim T\setminus\max(T)$.
We say that  $S$ is {\em L-weakly similar} to $T$, and write $S\Lwsim T$,
if and only if
$S\setminus \max(S)\Lplussim T\setminus\max(T)$.
\end{defn}

In the following, we put the technicalities for the LSDAP$^+$ case in parentheses.

\begin{defn}[$\Ext_T(B;\tilde{X})$]\label{defn.ExtBX}
Let $T\in\mathcal{T}$ be  fixed  and let $D=r_n(T)$ for some $n<\om$.
Suppose $A$ is a subtree of $D$ such that
$A \plussim \tilde{A}$ ($A \Lplussim \tilde{A}$)
and
 $A$ is extendible to a similarity ($L$-similarity) copy of $\tilde{C}$ in $T$.
Let $B$ be a subset of  the level set $\max(D)^+$ such that $B$ end-extends  or equals
$\max(A)^+$
and $A\cup B\wsim\tilde{A}\cup\tilde{B}$
($A\cup B\Lwsim\tilde{A}\cup\tilde{B}$).
Let $X^*$ be a level set end-extending $B$ such that $A\cup X^*\plussim \tilde{A}\cup\tilde{X}$ 
($A\cup X^*\Lplussim \tilde{A}\cup\tilde{X}$).  
Let $U^*=T\rl (\ell_B-1)$.
Define
$\Ext_T(B;X^*)$ to be the collection of all level sets $X\sse T$
such that
\begin{enumerate}
\item
$X$ end-extends $B$;
\item
$U^*\cup X\plussim U^*\cup X^*$ ($U^*\cup X\Lplussim U^*\cup X^*$);
\item
$A\cup X$ extends to a copy of $\tilde{C}$.
\end{enumerate}
\end{defn}

For Case (b), condition (3) follows from (2).
For Case (a),
the Extension Property (Labeled Extension Property)
guarantees that
for any level set  $Y$ end-extending $B$,
there is a level set $X$ end-extending $Y$ such that $A\cup X$ satisfies
condition (3).
In both cases, condition (2) implies that $A\cup X\plussim \tilde{A}\cup\tilde{X}$ ($A\cup X\Lplussim \tilde{A}\cup\tilde{X}$).

The following theorem of \Erdos\  and Rado will  provide the pigeonhole principle for the forcing proof.

\begin{thm}[\Erdos-Rado, \cite{Erdos/Rado56}]\label{thm.ER}
For $r<\om$ and $\mu$ an infinite cardinal,
$$
\beth_r(\mu)^+\ra(\mu^+)_{\mu}^{r+1}.
$$
\end{thm}

We are now ready to prove the Ramsey theorem for level set extensions of a given finite tree.

\begin{thm}[Level Set Ramsey Theorem]\label{thm.matrixHL}
Suppose that
$\mathcal{K}$
has \Fraisse\ limit $\bK$ satisfying
\EEAP$^+$ (or LSDAP$^+$), and
$T\in\mathcal{T}$ is given.
Let $\tilde{C}$ be a finite antichain of coding nodes in $T$,   $\tilde{A}$ be an initial segment of $\tilde{C}$,
and $\tilde{B}$ and $\tilde{X}$ be defined as above.
Suppose
 $D=r_n(T)$ for some $n<\om$,
 and
 $A\sse D$
and $B\sse \max(D^+)$
satisfy
$A\cup B\wsim \tilde{A}\cup\tilde{B}$  ($A\cup B\Lwsim\tilde{A}\cup\tilde{B}$).
Let  $X^*$ be a level set end-extending $B$ such that $A\cup X^*\plussim \tilde{A}\cup\tilde{X}$ ($A\cup X^*\Lplussim \tilde{A}\cup\tilde{X}$).
Then given any coloring
  $h:  \Ext_T(B;X^*)\ra 2$,
  there is a coding tree $S\in [D,T]$ such that
$h$ is monochromatic on $\Ext_S(B;X^*)$.
\end{thm}

\begin{proof}
Enumerate the nodes in $B$
as $s_0,\dots,s_d$ so that  for
any $X\in\Ext_T(B;X^*)$,
the critical node in $X$ extends $s_d$.
Let $M$ denote the collection of all  $m\ge n$  for which  there is a member of
 $\Ext_T(B;X^*)$ with  nodes  in $T(m)$.
 Note that this set $M$ is the same for any $S\in\mathcal{T}$.
 Let $L=\{|t|:\exists m\in M\, (t\in T(m))\}$,
 the collection of lengths of nodes in the levels $T(m)$ for $m\in M$.

For
  $i\le d$,   let  $T_i=\{t\in T:t\contains s_i\}$.
Let $\kappa$ be large enough, so that the partition relation $\kappa\ra (\aleph_1)^{2d}_{\aleph_0}$ holds.
The following forcing notion $\bP$    adds $\kappa$ many paths through  each  $T_i$,  $i< d$,
and one path  through $T_d$.
\vskip.1in

In both   Cases (a) and (b), define
$\bP$ to be  the set of finite partial functions $p$   such that
$$
p:(d\times\vec{\delta}_p)\cup\{d\}\ra T(m_p),
$$
where
\begin{enumerate}
\item
  $m_p\in M$ and $\vec{\delta}_p$ is a finite subset of $\kappa$;
 \item
 $\{p(i,\delta) : \delta\in  \vec{\delta}_p\}\sse  T_i(m_p)$ for each $i<d$;
 \item
 $p(d)$ is the critical node in $T_d(m_p)$; and
\item
 For any choices of $\delta_i\in\vec{\delta}_p$,
 the level set $\{p(i,\delta_i):i<d\}\cup\{p(d)\}$ is a member of $\Ext_T(B;X^*)$.
 \end{enumerate}
Given $p\in\bP$,
 the {\em range of $p$} is defined as
$$
\ran(p)=\{p(i,\delta):(i,\delta)\in d\times \vec\delta_p\}\cup \{p(d)\}.
$$
Let $\ell_p$ denote the length of the nodes in $\ran(p)$.
If also $q\in \bP$ and $\vec{\delta}_p\sse \vec{\delta}_q$, then we let $\ran(q\re \vec{\delta}_p)$ denote
$\{q(i,\delta):(i,\delta)\in d\times \vec{\delta}_p\}\cup \{q(d)\}$.

In Case (a),
the partial
 ordering on $\bP$  is   defined by 
$q\le p$ if and only if (1) and (2) hold:
\begin{enumerate}
\item
$m_q\ge m_p$, $\vec{\delta}_q\contains \vec{\delta}_p$,
$q(d)\contains p(d)$.
(In the case of LSDAP$^+$, we also require that
$\psi(q(d))=\psi(p(d))$.)
\item
$q(i,\delta)\contains p(i,\delta)$ for each $(i,\delta)\in d\times \vec{\delta}_p$.
(In the case of LSDAP$^+$, we also require that $\psi(q(i,\delta))=\psi(p(i,\delta))$.)
\end{enumerate}
In Case (b), we define
$q\le p$  if and only if (1) and (2) hold and additionally, the following third requirement holds:
\begin{enumerate}
\item[(3)]
Letting $U=T\re (\ell_p-1)$,
$U\cup\ran(p)\plussim
U\cup \ran(q\re\vec{\delta}_p)$.
\end{enumerate}
(Requirement (3)  is stronger than that which was used for the Rado graph in \cite{DobrinenRado19}, because
for relations of  arity three or more, the extension $q$ must preserve information about $1$-types over the fixed finite structure which we wish to extend.)
Then   $(\bP,\le)$ is a separative, atomless partial order.

The next part of the proof (up to and including Lemma \ref{lem.compat}) follows that of  \cite{DobrinenRado19} almost verbatim.
The key difference between the work here and in \cite{DobrinenRado19} is that here, information which makes the proof work for relations of any arity is
embedded  in the definition of $\Ext_T(B,X^*)$.
For  $(i,\al)\in d\times \kappa$,
 let
\begin{equation}
 \dot{b}_{i,\al}=\{\lgl p(i,\al),p\rgl:p\in \bP\mathrm{\  and \ }\al\in\vec{\delta}_p\},
\end{equation}
 a $\bP$-name for the $\al$-th generic branch through $T_i$.
Let
\begin{equation}
\dot{b}_d=\{\lgl p(d),p\rgl:p\in\bP\},
\end{equation}
a $\bP$-name for the generic branch through $T_d$.
 Given a generic filter   $G\sse \bP$, notice that
 $\dot{b}_d^G=\{p(d):p\in G\}$,
 which  is a cofinal path of  critical   nodes in $T_d$.
Let $\dot{L}_d$ be a $\bP$-name for the set of lengths of critical   nodes in $\dot{b}_d$, and note that
$\bP$ forces  that $\dot{L}_d\sse L$.
Let $\dot{\mathcal{U}}$ be a $\bP$-name for a non-principal ultrafilter on $\dot{L}_d$.
Given  $p\in \bP$, recall that  $\ell_p$ denotes the lengths of the nodes in $\ran(p)$,
and  notice  that
\begin{equation}
 p\forces  \forall
 (i,\al)\in d\times \vec\delta_p\, (\dot{b}_{i,\al}\re \ell_p= p(i,\al)) \wedge
( \dot{b}_d\re \ell_p=p(d)).
\end{equation}

We will write sets $\{\al_i:i< d\}$ in $[\kappa]^d$ as vectors $\vec{\al}=\lgl \al_0,\dots,\al_{d-1}\rgl$ in strictly increasing order.
For $\vec{\al}\in[\kappa]^d$,  let
\begin{equation}
\dot{b}_{\vec{\al}}=
\lgl \dot{b}_{0,\al_0},\dots, \dot{b}_{d-1,\al_{d-1}},\dot{b}_d\rgl.
\end{equation}
For  $\ell<\om$,
 let
 \begin{equation}
 \dot{b}_{\vec\al}\re \ell=
 \lgl \dot{b}_{0,\al_0}\re \ell,\dots, \dot{b}_{d-1,\al_{d-1}}\re \ell,\dot{b}_d\re \ell\rgl.
 \end{equation}
 One sees that
$h$  is a coloring on level sets  of the form $\dot{b}_{\vec\al}\re \ell$
whenever this is
 forced to be a member of $\Ext_T(B; X^*)$.
Given $\vec{\al}\in [\kappa]^d$ and  $p\in \bP$ with $\vec\al\sse\vec{\delta}_p$,
let
\begin{equation}
X(p,\vec{\al})=\{p(i,\al_i):i<d\}\cup\{p(d)\},
\end{equation}
recalling  that this level set
$X(p,\vec{\al})$
 is a member of
  $\Ext_T(B;X^*)$.

For each $\vec\al\in[\kappa]^d$,
choose a condition $p_{\vec{\al}}\in\bP$ satisfying the following:
\begin{enumerate}
\item
 $\vec{\al}\sse\vec{\delta}_{p_{\vec\al}}$.
\item
There is an $\varepsilon_{\vec{\al}}\in 2$
 such that
$p_{\vec{\al}}\forces$
``$h(\dot{b}_{\vec{\al}}\re \ell)=\varepsilon_{\vec{\al}}$
for $\dot{\mathcal{U}}$ many $\ell$ in $\dot{L}_d$''.
\item
$h(X(p_{\vec\al},\vec{\al}))=\varepsilon_{\vec{\al}}$.
\end{enumerate}
Such conditions can be found as follows:
Fix some $X\in \Ext_T(B;X^*)$ and let $t_i$ denote the node in $X$ extending $s_i$, for each $i\le d$.
For $\vec{\al}\in[\kappa]^d$,  define
$$
p^0_{\vec{\al}}=\{\lgl (i,\delta), t_i\rgl: i< d, \ \delta\in\vec{\al} \}\cup\{\lgl d,t_d\rgl\}.
$$
Then
 (1) will  hold for all $p\le p^0_{\vec{\al}}$,
 since $\vec\delta_{p_{\vec\al}^0}= \vec\al$.
 Next,
let  $p^1_{\vec{\al}}$ be a condition below  $p^0_{\vec{\al}}$ which
forces  $h(\dot{b}_{\vec{\al}}\re \ell)$ to be the same value for
$\dot{\mathcal{U}}$  many  $\ell\in \dot{L}_d$.
Extend this to some condition
 $p^2_{\vec{\al}}\le p_{\vec{\al}}^1$
 which
 decides a value $\varepsilon_{\vec{\al}}\in 2$
 so that
 $p^2_{\vec{\al}}$ forces
  $h(\dot{b}_{\vec{\al}}\re \ell)=\varepsilon_{\vec{\al}}$
for $\dot{\mathcal{U}}$ many $\ell$ in $\dot{L}_d$.
Then
 (2) holds for all $p\le p_{\vec\al}^2$.
If $ p_{\vec\al}^2$ satisfies (3), then let $p_{\vec\al}=p_{\vec\al}^2$.
Otherwise,
take  some  $p^3_{\vec\al}\le p^2_{\vec\al}$  which forces
$\dot{b}_{\vec\al}\re \ell\in \Ext_T(B;X^*)$ and
$h'(\dot{b}_{\vec\al}\re \ell)=\varepsilon_{\vec\al}$
for
some $\ell\in\dot{L}$
 with
$\ell_{p^2_{\vec\al}}< \ell\le \ell_{p^3_{\vec\al}}$.
Since $p^3_{\vec\al}$  forces  that $\dot{b}_{\vec\al}\re \ell$ equals
$\{p^3_{\vec\al}(i,\al_i)\re \ell:i<d\}\cup \{p^3_{\vec\al}(d)\re \ell\}$,
which is exactly
$X(p^3_{\vec\al}\re \ell,\vec\al)$,
and this level set is in the ground model, it follows that $h(X(p^3_{\vec\al}\re \ell,\vec\al))
=\varepsilon_{\vec\al}$.
Let
$p_{\vec\al}$ be $p^3_{\vec\al}\re \ell$.
Then $p_{\vec\al}$ satisfies (1)--(3).

Let $\mathcal{I}$ denote the collection of all functions $\iota: 2d\ra 2d$ such that
for each $i<d$,
$\{\iota(2i),\iota(2i+1)\}\sse \{2i,2i+1\}$.
For $\vec{\theta}=\lgl \theta_0,\dots,\theta_{2d-1}\rgl\in[\kappa]^{2d}$,
$\iota(\vec{\theta}\,)$ determines the pair of sequences of ordinals $\lgl \iota_e(\vec{\theta}\,),\iota_o(\vec{\theta}\,)\rgl$, where
\begin{align}
\iota_e(\vec{\theta}\,)&=
\lgl \theta_{\iota(0)},\theta_{\iota(2)},\dots,\theta_{\iota(2d-2))}\rgl\cr
\iota_o(\vec{\theta}\,)&=
 \lgl\theta_{\iota(1)},\theta_{\iota(3)},\dots,\theta_{\iota(2d-1)}\rgl.
 \end{align}

We now proceed to  define a coloring  $f$ on
$[\kappa]^{2d}$ into countably many colors.
Let $\vec{\delta}_{\vec\al}$ denote $\vec\delta_{p_{\vec\al}}$,
 $k_{\vec{\al}}$ denote $|\vec{\delta}_{\vec\al}|$,
$\ell_{\vec{\al}}$ denote  $\ell_{p_{\vec\al}}$, and let $\lgl \delta_{\vec{\al}}(j):j<k_{\vec{\al}}\rgl$
denote the enumeration of $\vec{\delta}_{\vec\al}$
in increasing order.
Given  $\vec\theta\in[\kappa]^{2d}$ and
 $\iota\in\mathcal{I}$,   to reduce subscripts
 let
$\vec\al$ denote $\iota_e(\vec\theta\,)$ and $\vec\beta$ denote $\iota_o(\vec\theta\,)$, and
define
\begin{align}\label{eq.fiotatheta}
f(\iota,\vec\theta\,)= \,
&\lgl \iota, \varepsilon_{\vec{\al}}, k_{\vec{\al}}, p_{\vec{\al}}(d),
\lgl \lgl p_{\vec{\al}}(i,\delta_{\vec{\al}}(j)):j<k_{\vec{\al}}\rgl:i< d\rgl,\cr
& \lgl  \lgl i,j \rgl: i< d,\ j<k_{\vec{\al}},\ \mathrm{and\ } \delta_{\vec{\al}}(j)=\al_i \rgl,\cr
&\lgl \lgl j,k\rgl:j<k_{\vec{\al}},\ k<k_{\vec{\beta}},\ \delta_{\vec{\al}}(j)=\delta_{\vec{\beta}}(k)\rgl\rgl.
\end{align}
Fix some ordering of $\mathcal{I}$ and define
\begin{equation}
f(\vec{\theta}\,)=\lgl f(\iota,\vec\theta\,):\iota\in\mathcal{I}\rgl.
\end{equation}

By the \Erdos-Rado Theorem  \ref{thm.ER},  there is a subset $K\sse\kappa$ of cardinality $\aleph_1$
which is homogeneous for $f$.
Take $K'\sse K$ so that between each two members of $K'$ there is a member of $K$.
Given  sets of ordinals $I$ and $J$,  we write $I<J$  to mean that  every member of $I$ is less than every member of $J$.
Take  $K_i\sse K'$  be  countably infinite subsets
satisfying
  $K_0<\dots<K_{d-1}$.

Fix some  $\vec\gamma\in \prod_{i<d}K_i$, and define
\begin{align}\label{eq.star}
&\varepsilon^*=\varepsilon_{\vec\gamma},\ \
k^*=k_{\vec\gamma},\ \
t_d=p_{\vec\gamma}(d),\cr
t_{i,j}&=p_{\vec{\gamma}}(i,\delta_{\vec{\gamma}}(j))\mathrm{\ for\ }
i<d,\ j<k^*.
\end{align}
We show that the values  in equation (\ref{eq.star}) are the same for any choice of
$\vec\gamma$.

\begin{lem}\label{lem.onetypes}
 For all $\vec{\al}\in \prod_{i<d}K_i$,
 $\varepsilon_{\vec{\al}}=\varepsilon^*$,
$k_{\vec\al}=k^*$,  $p_{\vec{\al}}(d)=t_d$, and
$\lgl p_{\vec\al}(i,\delta_{\vec\al}(j)):j<k_{\vec\al}\rgl
=
 \lgl t_{i,j}: j<k^*\rgl$ for each $i< d$.
\end{lem}

\begin{proof}
Let
 $\vec{\al}$ be any member of $\prod_{i<d}K_i$, and let $\vec{\gamma}$ be the set of ordinals fixed above.
Take  $\iota\in \mathcal{I}$
to be the identity function on $2d$.
Then
there are $\vec\theta,\vec\theta'\in [K]^{2d}$
such that
$\vec\al=\iota_e(\vec\theta\,)$ and $\vec\gamma=\iota_e(\vec\theta'\,)$.
Since $f(\iota,\vec\theta\,)=f(\iota,\vec\theta'\,)$,
it follows that $\varepsilon_{\vec\al}=\varepsilon_{\vec\gamma}$, $k_{\vec{\al}}=k_{\vec{\gamma}}$, $p_{\vec{\al}}(d)=p_{\vec{\gamma}}(d)$,
and $\lgl \lgl p_{\vec{\al}}(i,\delta_{\vec{\al}}(j)):j<k_{\vec{\al}}\rgl:i< d\rgl
=
\lgl \lgl p_{\vec{\gamma}}(i,\delta_{\vec{\gamma}}(j)):j<k_{\vec{\gamma}}\rgl:i< d\rgl$.
\end{proof}

Let $l^*$ denote the length of the  node $t_d$, and notice that
the  node
 $t_{i,j}$ also has length $l^*$,  for each   $(i,j)\in d\times k^*$.

\begin{lem}\label{lem.j=j'}
Given any $\vec\al,\vec\beta\in \prod_{i<d}K_i$,
if $j,k<k^*$ and $\delta_{\vec\al}(j)=\delta_{\vec\beta}(k)$,
 then $j=k$.
\end{lem}

\begin{proof}
Let $\vec\al,\vec\beta$ be members of $\prod_{i<d}K_i$   and suppose that
 $\delta_{\vec\al}(j)=\delta_{\vec\beta}(k)$ for some $j,k<k^*$.
For  $i<d$, let  $\rho_i$ be the relation from among $\{<,=,>\}$ such that
 $\al_i\,\rho_i\,\beta_i$.
Let   $\iota$ be the member of  $\mathcal{I}$  such that for each $\vec\theta\in[K]^{2d}$ and each $i<d$,
$\theta_{\iota(2i)}\ \rho_i \ \theta_{\iota(2i+1)}$.
Fix some
$\vec\theta\in[K']^{2d}$ such that
$\iota_e(\vec\theta)=\vec\al$ and $\iota_o(\vec\theta)= \vec\beta$.
Since between any two members of $K'$ there is a member of $K$, there is a
 $\vec\zeta\in[K]^{d}$ such that  for each $i< d$,
 $\al_i\,\rho_i\,\zeta_i$ and $\zeta_i\,\rho_i\, \beta_i$.
Let   $\vec\mu,\vec\nu$ be members of $[K]^{2d}$ such that $\iota_e(\vec\mu)=\vec\al$,
$\iota_o(\vec\mu)=\vec\zeta$,
$\iota_e(\vec\nu)=\vec\zeta$, and $\iota_o(\vec\nu)=\vec\beta$.
Since $\delta_{\vec\al}(j)=\delta_{\vec\beta}(k)$,
the pair $\lgl j,k\rgl$ is in the last sequence in  $f(\iota,\vec\theta)$.
Since $f(\iota,\vec\mu)=f(\iota,\vec\nu)=f(\iota,\vec\theta)$,
also $\lgl j,k\rgl$ is in the last  sequence in  $f(\iota,\vec\mu)$ and $f(\iota,\vec\nu)$.
It follows that $\delta_{\vec\al}(j)=\delta_{\vec\zeta}(k)$ and $\delta_{\vec\zeta}(j)=\delta_{\vec\beta}(k)$.
Hence, $\delta_{\vec\zeta}(j)=\delta_{\vec\zeta}(k)$,
and therefore $j$ must equal $k$.
\end{proof}

For each $\vec\al\in \prod_{i<d}K_i$, given any   $\iota\in\mathcal{I}$, there is a $\vec\theta\in[K]^{2d}$ such that $\vec\al=\iota_o(\vec\al)$.
By the second line of equation  (\ref{eq.fiotatheta}),
there is a strictly increasing sequence
$\lgl j_i:i< d\rgl$  of members of $k^*$ such that
$\delta_{\vec\gamma}(j_i)=\al_i$.
By
homogeneity of $f$,
this sequence $\lgl j_i:i< d\rgl$  is the same for all members of $\prod_{i<d}K_i$.
Then letting
 $t^*_i$ denote $t_{i,j_i}$,
 one sees that
\begin{equation}
p_{\vec\al}(i,\al_i)=p_{\vec{\al}}(i, \delta_{\vec\al}(j_i))=t_{i,j_i}=t^*_i.
\end{equation}
Let $t_d^*$ denote $t_d$.

\begin{lem}\label{lem.compat}
For any finite subset $\vec{J}\sse \prod_{i<d}K_i$,
$p_{\vec{J}}:=\bigcup\{p_{\vec{\al}}:\vec{\al}\in \vec{J}\,\}$
is a member of $\bP$ which is below each
$p_{\vec{\al}}$, $\vec\al\in\vec{J}$.
\end{lem}

\begin{proof}
Given  $\vec\al,\vec\beta\in \vec{J}$,
if
 $j,k<k^*$ and
 $\delta_{\vec\al}(j)=\delta_{\vec\beta}(k)$, then
 $j$ and $k$ must be equal, by
 Lemma  \ref{lem.j=j'}.
Then  Lemma \ref{lem.onetypes} implies
that for each $i<d$,
\begin{equation}
p_{\vec\al}(i,\delta_{\vec\al}(j))=t_{i,j}=p_{\vec\beta}(i,\delta_{\vec\beta}(j))
=p_{\vec\beta}(i,\delta_{\vec\beta}(k)).
\end{equation}
Hence,
 for all
$\delta\in\vec{\delta}_{\vec\al}\cap
\vec{\delta}_{\vec\beta}$
and  $i<d$,
$p_{\vec\al}(i,\delta)=p_{\vec\beta}(i,\delta)$.
Thus,
$p_{\vec{J}}:=
\bigcup \{p_{\vec{\al}}:\vec\al\in\vec{J}\}$
is a  function with domain $\vec\delta_{\vec{J}}\cup\{d\}$, where
$\vec\delta_{\vec{J}}=
\bigcup\{
\vec{\delta}_{\vec\al}:
\vec\al\in\vec{J}\,\}$; hence
, $p_{\vec{J}}$ is a member of $\bP$.
Since
for each $\vec\al\in\vec{J}$,
$\ran(p_{\vec{J}}\re \vec{\delta}_{\vec\al})=\ran(p_{\vec\al})$,
it follows that
$p_{\vec{J}}\le p_{\vec\al}$ for each $\vec\al\in\vec{J}$.
\end{proof}

This ends the material drawn directly from \cite{DobrinenRado19}.

We now proceed to  build a (diagonal coding)  tree $S\in [D,T]$ so that  the coloring
$h$ will be  monochromatic on $\Ext_S(B;X^*)$.
Recall that $n$ is the integer such that $D=r_n(T)$.
Let $\{ m_j:j<\om\}$ be the strictly increasing enumeration of $M$, noting that $m_0\ge n$.
For each $i\le d$,
extend the node $s_i\in B$ to the node $t^*_i$.
Extend each node $u$
 in $\max(D)^+\setminus B$  to some  node  $u^*$ in $T\re \ell^*$.
 If $X^*$ has a coding node and $m_0=n$,
 require also that
 $(u^*)^+(u^*;D)\sim u^+(u;D)$;
 \EEAP\ ensures that such $u^*$ exist.
 Set
\begin{equation}
U^*=\{t^*_i:i\le d\}\cup\{u^*:u\in \max(D)^+\setminus B\}
\end{equation}
and note that $U^*$
end-extends   $\max(D)^+$.

If $m_0=n$,
then
$D\cup U^*$ is a member of $r_{m_0+1}[D,T]$.
In
 this case, let $U_{m_0+1}=D\cup U^*$, and
 let $U_{m_1}$ be any member of $r_{m_1}[U_{m_0+1},T]$.
Note  that   $U^*$ is  the only member of $\Ext_{U_{m_1}}(B;X^*)$, and it has $h$-color $\varepsilon^*$.
Otherwise,
 $m_0>n$.
 In this case,
 take some  $U_{m_0}\in r_{m_0}[D,T]$ such that
 $\max(U_{m_0})$ end-extends $U^*$,
 and
 notice that   $\Ext_{U_{m_0}}(B;X^*)$ is empty.

Now assume that  $j<\om$ and
 we have constructed $U_{m_j}\in r_{m_j}[D,T]$
  so that every member of $\Ext_{U_{m_j}}(B;X^*)$
 has $h$-color $\varepsilon^*$.
Fix some  $V\in r_{m_j+1}[U_{m_j} ,T]$ and let $Y=\max(V)$.
We will extend the nodes in $Y$  to construct
$U_{m_j+1}\in r_{m_j+1}[U_{m_j},T]$
with the property that all members of $\Ext_{U_{m_j+1}}(B;X^*)$ have the same
 $h$-value $\varepsilon^*$.
This will be achieved  by constructing
 the condition $q\in\bP$, below, and then extending it to some condition $r \le q$ which decides that  all members of
 $\Ext_T(B;X^*)$ coming from the nodes in $\ran(r)$ have $h$-color $\varepsilon^*$.

Let $q(d)$ denote  the  splitting node or coding  node in $Y$ and let $\ell_q=|q(d)|$.
For each $i<d$,
let  $Y_i$ denote  $Y\cap T_i$.
For each $i<d$,
take  a set $J_i\sse K_i$ of  size card$(Y_i)$
and label the members of $Y_i$ as
$\{z_{\al}:\al\in J_i\}$.
Let $\vec{J}$ denote $\prod_{i<d}J_i$.
By   Lemma \ref{lem.compat},
the set $\{p_{\vec\al}:\vec\al\in\vec{J}\}$ is compatible,  and
$p_{\vec{J}}:=\bigcup\{p_{\vec\al}:\vec\al\in\vec{J}\}$ is a condition in $\bP$.

Let
 $\vec{\delta}_q=\bigcup\{\vec{\delta}_{\vec\al}:\vec\al\in \vec{J}\}$.
For $i<d$ and $\al\in J_i$,
define $q(i,\al)=z_{\al}$.
It follows  that for each
$\vec\al\in \vec{J}$ and $i<d$,
\begin{equation}
q(i,\al_i)\contains t^*_i=p_{\vec\al}(i,\al_i)=p_{\vec{J}}(i,\al_i),
\end{equation}
and
\begin{equation}
q(d)\contains t^*_d=p_{\vec\al}(d)=p_{\vec{J}}(d).
\end{equation}

For   $i<d$ and $\delta\in\vec{\delta}_q\setminus
J_i$,
we need to extend each node $p_{\vec{J}}(i,\delta)$ to some  node of length $\ell_q$ in order to construct a condition $q$ extending $p_{\vec{J}}$.
These nodes will not be a part of the construction of $U_{m_j+1}$, however; they only are only a technicality allowing us to find some $r\le q\le p_{\vec{J}}$ from which we will build $U_{m_j+1}$.
In Case (a),
let
$q(i,\delta)$ be  any  extension
 of $p_{\vec{J}}(i,\delta)$ in $T$ of length $\ell_q$.
 In Case (b),
 let $q(i,\delta)$ be  any  extension
 of $p_{\vec{J}}(i,\delta)$ in $T$ of length $\ell_q$  with
 \begin{equation}\label{eq.qidelta}
 q(i,\delta)^+(q(d); T\re(\ell^*-1))
 \sim
 p_{\vec{J}}(i,\delta)^+(p_{\vec{J}}(d); T\re(\ell^*-1)).
 \end{equation}
 The \EEAP\ guarantees the existence of such $q(i,\delta)$.
 (In the case of LSDAP$^+$, in addition to (\ref{eq.qidelta}), we also require  that  $\psi(q(i,\delta))=\psi(p(i,\delta))$ for all $i<d$ and $\delta\in\vec{\delta}_q\setminus
J_i$.
In this case, LSDAP$^+$ guarantees the existence of such a $q(i,\delta)$.)
Define
\begin{equation}
q=\{q(d)\}\cup \{\lgl (i,\delta),q(i,\delta)\rgl: i<d,\  \delta\in \vec{\delta}_q\}.
\end{equation}
This $q$ is a condition in $\bP$, and $q\le p_{\vec{J}}$.

Now take an $r\le q$ in  $\bP$ which  decides some $\ell_j$ in $\dot{L}_d$ for which   $h(\dot{b}_{\vec\al}\re \ell_j)=\varepsilon^*$, for all $\vec\al\in\vec{J}$.
This is possible since for all $\vec\al\in\vec{J}$,
$p_{\vec\al}$ forces $h(\dot{b}_{\vec\al}\re \ell)=\varepsilon^*$ for $\dot{\mathcal{U}}$ many $\ell\in \dot{L}_d$.
By the same argument as in creating the conditions $p_{\vec\al}$,
 we may assume that
 the nodes in the image of $r$ have length  $\ell_j$.
Since
$r$ forces $\dot{b}_{\vec{\al}}\re \ell_j=X(r,\vec\al)$
for each $\vec\al\in \vec{J}$,
and since the coloring $h$ is defined in the ground model,
it follows that
$h(X(r,\vec\al))=\varepsilon^*$ for each $\vec\al\in \vec{J}$.
Let
\begin{equation}
Y_0=\{q(d)\}\cup \{q(i,\al):i<d,\ \al\in J_i\},
\end{equation}
 and let
\begin{equation}
Z_0=\{r(d)\}\cup \{r(i,\al):i<d,\ \al\in J_i\}.
\end{equation}

Now we consider the two cases separately.
In Case (a),
let $Z$ be the level set
consisting of  the nodes in $Z_0$ along with
a  node  $z_y$ in $T\re \ell_j$ extending  $y$, for each
 $y\in
Y\setminus Y_0$.
Then $Z$ end-extends $Y$.
By \EEAP, it does not matter how the nodes  $z_y$ are chosen (except that in the case of LSDAP$^+$, we also require that $\psi(z_y)=\psi(y)$).
Letting $U_{m_j+1}=U_{m_j}\cup Z$,
we see that $U_{m_j+1}$ is a member of $r_{m_j+1}[U_{m_j},T]$ such that
$h$ has value $\varepsilon^*$ on $\Ext_{U_{m_j+1}}(B;X^*)$.

In Case (b),
 $r(d)$ is a coding node.
 Since $r\le q$,
the nodes in $\ran(r\re\delta_q)$ have the same passing types over $T\rl \ell_q$
as the nodes in $\ran(q)$ have over $T\rl \ell_q$.
We now need to extend all the other members of
$Y\setminus Y_0$
 to nodes with the required
 passing types at $r(d)$.
 For each
$y\in Y\setminus Y_0$,
 choose
a  member  $z_y\supset y$  in $T_d\re \ell_j$
so that
\begin{equation}\label{eq.zyplus}
z_y^+(r(d);U_{m_j})\sim y^+(q(d); U_{m_j}).
\end{equation}
\EEAP\ ensures the existence of  such $z_y$.
(In the case of LSDAP$^+$, 
in addition to (\ref{eq.zyplus}),
we also require that $\psi(z_y)=\psi(y)$.)
Let
$Z$ be the level set
consisting of the nodes in $Z_0$ along with
the nodes $z_y$ for
 $y\in Y\setminus Y_0$.
Then $Z$ end-extends $Y$ and moreover,
$U_{m_j}\cup Z\plussim V$.
Letting $U_{m_j+1}=U_{m_j}\cup Y$,
we see that $U_{m_j+1}$ is a member of $r_{m_j+1}[U_{m_j},T]$  and
$h$ has value $\varepsilon^*$ on $\Ext_{U_{m_j+1}}(B;X^*)$.

Let $U_{m_{j+1}}$ be any member of
$r_{m_{j+1}}[U_{m_j+1},T]$.
This completes the inductive construction.
Let $S=\bigcup_{j<\om}U_{m_j}$.
Then $S$ is a member of $[D,T]$ and
 for each $X\in\Ext_{S}(B)$,  $h(X)=\varepsilon^*$.
Thus, $S$ satisfies the theorem.
\end{proof}

\begin{rem}\label{rem.plussimfornoncn}
By  the construction in the previous proof, in Case (b) the coding nodes in any member $X\in\Ext_S(B;X^*)$ extend the coding node $t^*_d$.
It then follows from
(3) in Definition \ref{defn.sct}
that for every level set $X\sse S$ with $A\cup X\sim\tilde{A}\cup X^*$, the coding node $c$ in $X$  automatically satisfies
$c^+(c;A)\sim  (t^*_d)^+(t^*_d;A)
\sim \tilde{x}_d^+(\tilde{x}_d;\tilde{A})$, where $\tilde{x}_d$ denotes the  coding node in $X^*$.
Thus, $A\cup X\plussim \tilde{A}\cup X^*$ if and only if the
non-coding nodes in $X$ have immediate successors with similar passing types over $A\cup\{c\}$ as their counterparts in $X^*$ have over $\tilde{A}\cup\{\tilde{x}_d\}$.

Moreover, for languages with only unary and binary relations, in Case (b)  the set
$\Ext_T(B;X^*)$ is exactly the set of all end-extensions $X$ of $B$ such that $A\cup X\plussim \tilde{A}\cup\tilde{X}$ ($A\cup X\Lplussim \tilde{A}\cup\tilde{X}$ in the case of LSDAP$^+$).
\end{rem}

The main theorem of this paper  follows immediately from the previous theorem.
\vskip.1in

\begin{thm.indivisibility}
Suppose  $\mathcal{K}$ is a  \Fraisse\ class
in a finite relational language
with relation symbols  in any arity such that
its  \Fraisse\ limit  $\bK$ satisfies
\EEAP$^+$.
Then $\bK$ is indivisible.
\end{thm.indivisibility}

\begin{proof} 
Let $\bfC$ be a singleton structure in $\mathcal{K}$, and suppose $h$ is a coloring of all copies of $\bfC$ inside $\bK$ into two colors.
Let $T$ be a diagonal coding subtree of $\bS$ representing $\bK$, if one exists.
Otherwise,  we may without loss of generality assume that $T$ is a diagonal coding subtree of $\bU$ 
in which coding nodes representing $\bfC$ occur densely above any coding node.
Let $X^*$ be the coding node in $T$ of least length representing a copy of $\bfC$.
Let $A=D=r_0(T)$ be the empty sequence.
In the case that $T$ is a subtree of $\bS$ and $\Gamma$ is of size at least two,
let $B$ consist of $X^*$ along with one node $t_{\gamma}$ of the same length as $X^*$ extending  $\gamma$, for each $\gamma\in \Gamma$.
Otherwise, 
let $B$ be the initial segment of $X^*$ of length one.
Then
Theorem \ref{thm.matrixHL} provides us with a coding tree $S\in [D,T]$ such  that $h$ is monochromatic on $\Ext_S(B;X^*)$.
Since $D=r_0(T)$, every coding node in $S$ representing a copy of
$\bfC$
is a member of  $\Ext_S(B;X^*)$.
Thus, $\bK$ is indivisible.
 \end{proof}

\begin{rem}
The conclusion of Theorem \ref{thm.matrixHL} also holds for  \Fraisse\ structures satisfying LSDAP$^+$  in languages with finitely many relation symbols of arity at most two, but more work is required for the proof.  Indivisibility for such structures will follow from Theorem 3.6  in Part II. 
\end{rem}

\section{Conclusion}

The main theorem, Theorem 1.2, of this paper 
showing that all \Fraisse\ structures satisfying SDAP$^+$ with finitely many relations of any finite arities are indivisible 
followed from 
the Level Set Ramsey
Theorem \ref{thm.matrixHL}.
In Part II, \cite{CDPII}, we will start with Theorem
\ref{thm.matrixHL} as the basis for an induction proof of upper bounds for the big Ramsey degrees of \Fraisse\ structures satisfying SDAP$^+$ with finitely many relations of arity at most two.
Those upper bounds are given in terms of finite diagonal antichains of coding nodes representing a given finite structure.
Such upper bounds will moreover be proved to be exact, leading to big Ramsey structures in the sense of  Zucker \cite{Zucker19} which  have a simple presentation. 
Towards the end of \cite{CDPII}, a catalogue of results on indivisibility and on big Ramsey degrees will be presented, 
showing which
previous results are recovered by our methods and which results are new to our Parts I and II.

\newpage

\bibliography{references}
\bibliographystyle{ijmart}

\end{document}